\documentclass[fontsize=12pt,a4paper,headings=normal,
twoside=false,leqno,parskip=half-,abstract=true]{scrartcl}
\usepackage[english]{babel}
\usepackage[utf8]{inputenc}
\usepackage{hyperref}
\hypersetup{
 bookmarks=true,
 pdftitle={Sturm 3-ball global attractors 1: Dynamic complexes and meanders},
 pdfauthor={Bernold Fiedler, Carlos Rocha},
 colorlinks=true,
 linkcolor=blue,
 citecolor=blue,
 filecolor=blue,
 urlcolor=blue}

\usepackage{graphicx}
\usepackage[format=plain,labelfont=bf,font=small]{caption}
\usepackage{xcolor}

\usepackage{amsmath,amsthm}
\usepackage{amssymb} 
\newcommand{\transv}{\mathrel{\text{\tpitchfork}}}
\makeatletter
\newcommand{\tpitchfork}{%
  \vbox{
    \baselineskip\z@skip
    \lineskip-.52ex
    \lineskiplimit\maxdimen
    \m@th
    \ialign{##\crcr\hidewidth\smash{$-$}\hidewidth\crcr$\pitchfork$\crcr}
  }%
}
\makeatother
\usepackage{latexsym}

\usepackage[notref,notcite,color,
final 
]{showkeys}

\definecolor{refkey}{rgb}{1,0,0}
\definecolor{labelkey}{rgb}{1,0,0}

\usepackage[textwidth=2cm,textsize=small,backgroundcolor=none]{todonotes}

  \mathchardef\ordinarycolon\mathcode`\:
  \mathcode`\:=\string"8000
  \begingroup \catcode`\:=\active
    \gdef:{\mathrel{\mathop\ordinarycolon}}
  \endgroup

\theoremstyle{plain}
\newtheorem{thm}{Theorem}[section]
\newtheorem{lem}[thm]{Lemma}

\newtheorem{prop}[thm]{Proposition}
\newtheorem{cor}[thm]{Corollary}
\newtheorem{defi}[thm]{Definition}

\begin{document}

\title{{\LARGE{Sturm 3-ball global attractors 1:\\ Thom-Smale complexes and meanders}}}

\author{
 \\
\emph{-- Dedicated to 	Waldyr M. Oliva, mentor and friend  --}\\
{~}\\
Bernold Fiedler* and Carlos Rocha**\\
\vspace{2cm}}

\date{version of \today}
\maketitle
\thispagestyle{empty}

\vfill

*\\
Institut für Mathematik\\
Freie Universität Berlin\\
Arnimallee 3\\ 
14195 Berlin, Germany\\
\\
**\\
Center for Mathematical Analysis, Geometry and Dynamical Systems\\
Instituto Superior T\'ecnico\\
Universidade de Lisboa\\
Avenida Rovisco Pais\\
1049--001 Lisbon, Portugal\\


\newpage
\pagestyle{plain}
\pagenumbering{roman}
\setcounter{page}{1}

\begin{abstract}
This is the first of three papers on the geometric and combinatorial characterization of global Sturm attractors which consist of a single closed 3-ball.
The underlying scalar PDE is parabolic,
$$ u_t = u_{xx} + f(x,u,u_x)\,, $$
on the unit interval $0 < x<1$ with Neumann boundary conditions.
Equilibria are assumed to be hyperbolic.\\
\newline
Geometrically, we study the resulting Thom-Smale dynamic complex with cells defined by the unstable manifolds of the equilibria.
The Thom-Smale complex turns out to be a regular cell complex.
Our geometric description involves a bipolar orientation of the 1-skeleton, a hemisphere decomposition of the boundary 2-sphere by two polar meridians, and a meridian overlap of certain 2-cell faces in opposite hemispheres.\\
\newline
The combinatorial description is in terms of the Sturm permutation, alias the meander properties of the shooting curve for the equilibrium ODE boundary value problem.
It involves the relative positioning of extreme 2-dimensionally unstable equilibria at the Neumann boundaries $x=0$ and $x=1$, respectively, and the overlapping reach of polar serpents in the shooting meander.\\
\newline
In the present paper we show the implications
$$ \text{Sturm attractor}\quad \Longrightarrow \quad \text{Thom-Smale complex} \quad \Longrightarrow \quad \text{meander}\,.$$
The sequel, part 2, closes the cycle of equivalences by the implication
$$ \text{meander} \quad \Longrightarrow \quad \text{Sturm attractor}\,.$$
Each implication, or mapping, involves certain constructions which are tuned such that the final 3-ball Sturm global attractor defined by the meander combinatorics coincides with the originally given Sturm 3-ball.
Many explicit examples and illustrations will be discussed in part 3.
The present 3-ball trilogy extends our previous trilogy on planar Sturm global attractors towards the still elusive goal of geometric and combinational characterizations of all Sturm global attractors of arbitrary dimension.

\end{abstract}

\newpage

\tableofcontents


\newpage
\pagenumbering{arabic}
\setcounter{page}{1}

\section{Introduction}
\label{sec1}

\numberwithin{equation}{section}
\numberwithin{figure}{section}

For our general introduction we first follow \cite{firo14}.
\emph{Sturm global attractors} $\mathcal{A}_f$ are the global attractors of scalar parabolic equations
	\begin{equation}
	u_t = u_{xx} + f(x,u,u_x)
	\label{eq:1.1}
	\end{equation}
on the unit interval $0<x<1$.
Just to be specific we consider Neumann boundary conditions $u_x=0$ at $x=0,1$.
Standard semigroup theory provides local solutions $u(t,x)$ for $t \geq 0$ and given initial data at time $t=0$, in suitable solution spaces $u(t, \cdot) \in X \subseteq C^1 ([0,1], \mathbb{R})$.
Under suitable dissipativeness assumptions on $f \in C^2$, any solution eventually enters a fixed large ball in $X$.
In fact that large ball of initial conditions itself limits onto the maximal compact and invariant subset $\mathcal{A}_f$ which is called the global attractor.
In general, the global attractor consists of all \emph{eternal solutions}, i.e. of all solutions $u(t, \cdot )$ which exist globally and remain uniformly bounded for all real times $t \in \mathbb{R}$, both in the positive and in the negative (backwards) time direction.
Since \eqref{eq:1.1} possesses a Lyapunov~function, alias a variational structure, the global attractor consists of equilibria and of solutions $u(t, \cdot )$, $t \in \mathbb{R}$, with forward and backward limits, i.e.
	\begin{equation}
	\underset{t \rightarrow \pm \infty}{\mathrm{lim}} u(t, \cdot ) = v_\pm 	\,.
	\label{eq:1.2}
	\end{equation}
In other words, the $\alpha$- and $\omega$-limit sets of $u(t,\cdot )$ are two distinct equilibria $v_\pm$.
We call $u(t, \cdot )$ a \emph{heteroclinic} or \emph{connecting} orbit and write $v_- \leadsto v_+$ for such heteroclinically connected equilibria.
Equilibria $v = v(x)$ are time-independent solutions, of course, and hence satisfy the ODE
	\begin{equation}
	0 = v_{xx} + f(x,v,v_x)\,,
	\label{eq:1.3}
	\end{equation} 
for $0\leq x \leq 1$, again with Neumann~boundary.
See \cite{he81, pa83, ta79} for a general background, \cite{ma78, mana97, ze68, hu11, fietal14} for the gradient-like Lyapunov~structure of \eqref{eq:1.1} under separated boundary conditions, and \cite{bavi92, chvi02, edetal94, ha88, haetal02, la91, ra02, seyo02, te88} for global attractors in general.

Here and below we assume that all equilibria $v$ of \eqref{eq:1.1}, \eqref{eq:1.3} are \emph{hyperbolic}, i.e. without eigenvalues (of) zero (real part) of their linearization.
Let $\mathcal{E} = \mathcal{E}_f$ denote the set of equilibria.
Our generic hyperbolicity assumption and dissipativeness of $f$ imply that $\mathcal{A}_f$ is contractible. In particular $N$:= $|\mathcal{E}_f|$ is odd.

We attach the name of \emph{Sturm} to the PDE \eqref{eq:1.1}, and to its global attractor $\mathcal{A}_f$, due to a crucial nodal property of its solutions which we express by the \emph{zero number} $z$.
Let $0 \leq z (\varphi) \leq \infty$ count the number of strict sign changes of $\varphi : [0,1] \rightarrow \mathbb{R}, \, \varphi \not\equiv 0$.
Then
	\begin{equation}
	t \quad \longmapsto \quad z(u^1(t, \cdot ) - u^2(t, \cdot ))\,
	\label{eq:1.4}
	\end{equation}
is finite and nonincreasing with time $t$, for $t>0$ and any two distinct solutions $u^1$, $u^2$ of \eqref{eq:1.1}.
Moreover $z$ drops strictly with increasing $t$, at any multiple zero of $x \longmapsto u^1(t_0 ,x) - u^2(t_0 ,x)$; see \cite{an88}.
See Sturm \cite{st1836} for a linear autonomous version.
The case $z=0$ is known as strong monotonicity or parabolic comparison principle for scalar parabolic equations, and holds in any space dimension.
The full Sturm~structure \eqref{eq:1.4}, however, restricts applicability to one space dimension, a few types of delay equations, and certain tridiagonal Jacobi type ODE systems.
The dynamic consequences of the Sturm~structure, however, are enormous.
For a first introduction see also \cite{ma82, brfi88, fuol88, mp88, brfi89, ro91, fisc03, ga04} and the many references there.

As a convenient notational variant of the zero number $z$, we also write
	\begin{equation}
	z(\varphi) = j_{\pm}
	\label{eq:1.4+}
	\end{equation}
to indicate $j$ strict sign changes of $\varphi$, by $j$, and $\pm \varphi (0) >0$, by the index $\pm$.
For example $z(\pm \varphi_j) = j_{\pm}$, for the $j$-th Sturm-Liouville eigenfunction $\varphi_j$.

In a series of papers, we have given a combinatorial description of Sturm global attractors $\mathcal{A}_f$; see \cite{firo96, firo99, firo00}.
Define the two \emph{boundary orders} $h^f_0, h^f_1$: $\lbrace 1, \ldots, N \rbrace \rightarrow \mathcal{E}_f$ of the equilibria such that
	\begin{equation}
	h^f_\iota (1) < h^f_\iota (2) < \ldots < h^f_\iota (N) \qquad \text{at}
	\qquad x=\iota = 0,1\,.
	\label{eq:1.5}
	\end{equation}
In other words, $(h^f_\iota)^{-1}(v)$ is the ranking of the equilibria $v$, by increasing boundary values at $x=\iota$.
See figs.~\ref{fig:3.1} and \ref{fig:6.5} for specific examples, where $h_0^f(7)=24$ and $h_1^f(7)=23$.

The combinatorial description is based on the \emph{Sturm~permutation} $\sigma_f \in S_N$ which was introduced by Fusco and Rocha in \cite{furo91} and is defined as
	\begin{equation}
	\sigma_f:= (h^f_0)^{-1} \circ h^f_1\,.
	\label{eq:1.6}
	\end{equation}
Using a shooting approach to the ODE boundary value problem \eqref{eq:1.3}, the Sturm~permutations $\sigma_f$ have been characterized as \emph{dissipative Morse meanders} in \cite{firo99}; see also \eqref{eq:1.19a}--\eqref{eq:1.24} below.
In \cite{firo96} we have shown how to determine which equilibria $v_\pm$ possess a heteroclinic orbit connection \eqref{eq:1.2}, explicitly and purely combinatorially from $\sigma_f$.

More geometrically, global Sturm attractors $\mathcal{A}_f$ and $\mathcal{A}_g$ with the same Sturm permutation $\sigma_f = \sigma_g$ are $C^0$ orbit-equivalent \cite{firo00}.
For $C^1$-small perturbations, from $f$ to $g$, this global rigidity result is based on $C^0$ structural stability of Morse-Smale systems; see e.g. \cite{pasm70, pame82, ol83}.
For large perturbations, substantially new arguments were required, and provided, in \cite{firo00}.
A remaining puzzle are different, and even nonconjugate, Sturm permutations which give rise to $C^0$ orbit-equivalent Sturm attractors; see also fig.~\ref{fig:5.2} below.
We will address this puzzle in our sequel \cite{firo3d-2}.

It is the Sturm property of \eqref{eq:1.4} which implies the Morse-Smale property, for hyperbolic equilibria.
In fact, stable and unstable manifolds $W^u(v_-)$, $W^s(v_+)$, which intersect precisely along heteroclinic orbits $v_- \leadsto v_+$, are automatically transverse:
$W^u(v_-) \transv W^s(v_+)$.
See \cite{he85, an86}.
In the Morse-Smale setting, Henry already observed, that a heteroclinic orbit $v_- \leadsto v_+$ is equivalent to $v_+$ belonging to the boundary $\partial W^u(v_-)$ of the unstable manifold $W^u(v_-)$; see \cite{he85}.

In most of our previous papers, heteroclinic orbits were described by the \emph{connection~graph} $\mathcal{H}_f$ with vertices given by the set $\mathcal{E}_f$ of equilibria, all hyperbolic.
Let $i(v)=\mathrm{dim} \,W^u (v)$ denote the \emph{Morse~index} of $v$, i.e. the dimension of the unstable manifold $W^u$ of $v$.
Then the edges of the directed connection graph $\mathcal{H}_f$ are given by the unique heteroclinic orbits $u:\, v_- \leadsto v_+$ between equilibria of adjacent Morse indices $i(v_+) = i(v_-)-1$.
In other words, an edge between such vertices $v_\pm$ exists if, and only if, $v_\pm$ possess a heteroclinic orbit connecting them.
The "connects to" relation $\leadsto$ is transitive and satisfies a cascading principle; see \cite{brfi89, firo96}.
Therefore it is sufficient to know the connection graph $\mathcal{H}_f$ in order to conclude for any pair $v_\pm$ of equilibria whether or not they possess a heteroclinic connecting orbit.
Indeed $ v_- \leadsto v_+$ if and only if there exists a directed path from $v_-$ to $v_+$ in $\mathcal{H}_f$.

For planar Sturm~attractors $\mathcal{A}_f$, i.e. for equilibrium sets $\mathcal{E}_f$ with a maximal Morse index two \cite{br90, jo89, ro91}, a more geometric approach had been initiated in the planar Sturm trilogy \cite{firo08, firo09, firo10}.
It was clarified which planar graphs $\mathcal{H}$ do arise as connection graphs $\mathcal{H}=\mathcal{H}_f$ of planar Sturm attractors $\mathcal{A}_f$, and which ones do not.
Meanwhile, a \emph{Schoenflies~theorem} has also been proved to hold for the closure $\mathrm{clos\,}{W}^u(v) \subseteq X$ of the unstable manifold $W^u$ of any hyperbolic equilibrium $v$; see \cite{firo13}.
In particular $\mathrm{clos\,}{W}^u(v)$ is the homeomorphic Euclidean embedding of a closed unit ball $\bar{B}^{i(v)}$ of dimension $i(v)$.
In \cite{firo14} this allowed us to reformulate the combinatorial results of \cite{firo08, firo09, firo10}, in a more geometric and topological language, as follows.

Consider \emph{finite~regular} CW-\emph{complexes}
	\begin{equation}
	\mathcal{C} = \bigcup\limits_{v\in \mathcal{E}} c_v\,,
	\label{eq:1.7}
	\end{equation}
i.e. finite disjoint unions of \emph{cell~interiors} $c_v$ with additional gluing properties. We think of the labels $v\in \mathcal{E}$ as \emph{barycenter} elements of $c_v$. For CW-complexes we require the closures $\mathrm{clos\,}{c}_v$ in $\mathcal{C}$ to be the continuous images of closed unit balls $\bar{B}_v$ under \emph{characteristic maps} $\bar{B}_v \rightarrow \mathrm{clos\,} c_v$.
We call $\mathrm{dim}\,\bar{B}_v$ the dimension of the (open) cell $c_v$. 
For positive dimensions of $\bar{B}_v$ we require $c_v$ to be the homeomorphic images of the interiors $B_v$. 
For dimension zero we write $B_v := \bar{B}_v$ so that any 0-cell $c_v= B_v$ is just a point.
The \emph{m-skeleton} $\mathcal{C}^m$ of $\mathcal{C}$ consists of all cells of dimension at most $m$.
We require $\partial c_v := \mathrm{clos\,}{c}_v \setminus c_v \subseteq \mathcal{C}^{m-1}$ for any $m$-cell $c_v$.
Thus, the boundary $(m-1)$-sphere $S_v := \partial B_v = \bar{B}_v \setminus B_v$ of any $m$-ball  $B_v$, $m>0$, maps into the $(m-1)$-skeleton,
	\begin{equation}
	\partial B_v \quad \longrightarrow \quad \partial c_v \subseteq \mathcal{C}^{m-1}\,,
	\label{eq:1.8}
	\end{equation}
for the $m$-cell $c_v$, by restriction of the continuous characteristic map.
The map \eqref{eq:1.8} is called the \emph{attaching} (or \emph{gluing}) \emph{map}.
For \emph{regular} CW-complexes, in contrast, the characteristic maps $ \bar{B}_v  \rightarrow  \mathrm{clos\,}{c}_v $ are required to be homeomorphisms, up to and including the \emph{attaching} (or \emph{gluing}) \emph{homeomorphism}. We moreover require $\partial{c_v}$  to be a sub-complex of $\mathcal{C}^{m-1}$, then. 
See \cite{frpi90} for a background on this terminology.

In variational or gradient-like settings with hyperbolicity of equilibria it is tempting to expect the disjoint dynamic decomposition
	\begin{equation}
	\mathcal{A}_f = \bigcup\limits_{v \in \mathcal{E}_f} W^u(v)
	\label{eq:1.9}
	\end{equation}
of the global attractor $\mathcal{A}_f$, into the unstable manifolds $W^ u$ of its equilibria $v$, to be a finite regular CW-complex. If this expectation holds true, then \eqref{eq:1.9} is called the \emph{Thom-Smale complex} or \emph{dynamic complex} of the global attractor $\mathcal{A}_f$. See \cite{fr79, bo88, bizh92} for further background.
Unfortunately, there are some theoretical obstacles, and manifest counterexamples, to this grand expectation in general variational settings. See for example \cite{bahu04}.

In our Sturm setting \eqref{eq:1.1} with hyperbolic equilibria $v_1, \ldots, v_N$, however, automatic transversality of stable and unstable manifolds comes to our assistance. The zero number moreover implies that \eqref{eq:1.9} is a \emph{regular} dynamic complex, i.e. the dynamic decomposition \eqref{eq:1.9} of $\mathcal{A}_f$ is a finite \emph{regular} CW-complex with (open) cells $c_v$ given by the unstable manifolds $W^ u (v)$ of the equilibria $v$. 
The proof is closely related to the Schoenflies result of \cite{firo13}; see \cite{firo14}.
We can therefore define the \emph{Sturm~complex} $\mathcal{C}_f$ to be the regular dynamic complex, alias the Thom-Smale complex, $\mathcal{C}_f = \bigcup_{v \in \mathcal{E}_f}\, W^u(v)$ of the Sturm global attractor $\mathcal{A}_f$, provided all equilibria $v \in \mathcal{E}_f$ are hyperbolic. 
Again we call the equilibrium $v$ the \emph{barycenter} of the cell $c_v=W^u(v)$.
A planar Sturm complex $\mathcal{C}_f$, for example, is the Thom-Smale regular complex of a planar $\mathcal{A}_f$, i.e. of a Sturm global attractor for which all equilibria $v \in \mathcal{E}_f$ have Morse indices $i(v) \leq 2$.
See section~\ref{sec2} for a detailed discussion.

Actually, the Schoenflies result \cite{firo13} provides a finer structure than a mere regular CW-complex. It actually provides a disjoint hemisphere decomposition
	\begin{equation}
	\partial W^u(v) = \bigcup\limits_{0 \leq j < i(v)}^{\bullet}
	\Sigma _{\pm}^j(v)
	\label{eq:1.9a}
	\end{equation}
of the topological boundary $\partial W^u$:= $ \mathrm{clos\,} W^u(v) \smallsetminus W^u(v)$ of the unstable manifold $W^u(v) = c_v$, for any hyperbolic equilibrium $v$.
The construction of the disjoint hemispheres $\Sigma_{\pm}^j = \Sigma_{\pm}^j(v)$ can be summarized as follows.
For $0 \leq j \leq i(v)$, let $W^j$ denote the $j$-dimensional fast unstable manifold of $v$.
The tangent space to $W^j$ at $v$ is spanned by the eigenfunctions $\varphi_0, \ldots , \varphi_{j-1}$ of the linearization of \eqref{eq:1.3} at $v$, for the first $j$ eigenvalues $\lambda_0 > \ldots > \lambda_{j-1} >0$.
Of course, $W^0 := \{v\}$.
Consider any orbit $u(t, \cdot) \in W^{j+1} \smallsetminus W^j$, $t \in \mathbb{R}$.
Then
	\begin{equation}
	\lim_{t \rightarrow -\infty} \left( u \left( t,\cdot\right) -v 				\right)\, 	/\, |u\left(t, \cdot\right) -v| = \pm \varphi_j\,; 
	\label{eq:1.9b}
	\end{equation}		
by normalization of $\varphi_j$ in the appropriate norm of the phase space $X \hookrightarrow C^1$.
Here and below we fix signs such that $\varphi_j(0) >0$.
In particular, the signed zero number $z$ of \eqref{eq:1.4} satisfies
	\begin{equation}
	\lim_{t \rightarrow\, - \infty} z\left( u \left( t, \cdot\right) -v 			\right) = j_\pm\,.
	\label{eq:1.9c}
	\end{equation}
See \cite{brfi86} for further details on the construction of $W^j$.

The \emph{signed hemispheres} $\Sigma_\pm^j$ are defined, recursively, by the disjoint unions
	\begin{equation}
	\Sigma^j:= \partial W^{j+1} = \Sigma_-^j \, \dot{\cup} \, \Sigma_+^j 			\, \dot{\cup}\,  \Sigma^{j-1}\,,
	\label{eq:1.9d}
	\end{equation}
for $0 \leq j < i(v)$, with the convention $\Sigma ^{-1}$:= $ \emptyset$.
The hemisphere closures,
	\begin{equation}
	\mathrm{clos\,} \Sigma_\pm^j = \Sigma_\pm^j \, \dot{\cup}\, 						\Sigma^{j-1}\,,
	\label{eq:1.9e}
	\end{equation}
can be obtained as $\omega$-limit sets of protocap hemispheres which are $C^1$-small, nearly parallel perturbations of $\mathrm{clos\,} W^j$ in $\mathrm{clos\,} W^{j+1}$, in the eigendirections $\pm \varphi_j$, respectively.
In particular \eqref{eq:1.9b}, \eqref{eq:1.9c} hold in the interior of the protocaps, and for any heteroclinic orbit $v \leadsto \Tilde{v} \in \Sigma_\pm^j$.
In proposition~3.1(iv) we will characterize equilibria $\Tilde{v} \in \Sigma_\pm^j$ by their signed zero number as
	\begin{equation}
	z(\Tilde{v}-v) = j_\pm\,.
	\label{eq:1.9f}
	\end{equation}
Loosely speaking, we call the above dynamically defined complex $\Sigma_\pm^j(v)$ of signed hemispheres a signed Thom-Smale complex.	
	
The $m$-dimensional \emph{Chafee-Infante global attractor} $\mathcal{A}_{\text{CI}}^m$ is an illustrative example.
It arises from PDE \eqref{eq:1.1} for cubic nonlinearities $f(u) = \lambda u(1-u^2)$.
Consider $v$ = $\mathcal{O}$:= $0$ and observe $i(v) = m \geq 1$ for $(m-1)^2 < \lambda / \pi^2 < m^2$.
The $2m$ remaining equilibria $v_\pm^j$ are characterized by $z(v_\pm^j - \mathcal{O}) = j_\pm$, all hyperbolic.
The Thom-Smale complex \eqref{eq:1.9} of $\mathcal{A}_{\text{CI}}^m = \mathrm{clos\,} W^u(\mathcal{O})$ consists of the single $m$-cell $W^u (\mathcal{O})$ and the $m$-cell boundary $\partial W^u(\mathcal{O})$ given by \eqref{eq:1.9a}.
The hemisphere decomposition is simply the remaining dynamic decomposition
	\begin{equation}
	\Sigma_\pm^j = W^u (v_\pm^j)\,,
	\label{eq:1.9gg}
	\end{equation}
$0 \leq j < m = i(v)$, in the Chafee-Infante case.
See also \cite{chin74, he81, he85}.
The Chafee-Infante attractor $\mathcal{A}_{\text{CI}}^m$ is the $m$-dimensional Sturm attractor with the smallest possible number $N = 2m+1$ of equilibria.
Equivalently, among all Sturm attractors with $N = 2m+1$ equilibria, it possesses the largest possible dimension.
Interestingly the dynamics on each closed hemisphere $\mathrm{clos\,} \Sigma_\pm^j$ is itself $C^0$ orbit equivalent to the Chafee-Infante dynamics on $\mathcal{A}_{\text{CI}}^j$.

Our main objective, in the present trilogy of papers, is a geometric and combinatorial characterization of those global Sturm attractors, which are the closure
	\begin{equation}
	\mathcal{A}_f = \mathrm{clos\,} W^u (\mathcal{O})
	\label{eq:1.10}
	\end{equation}
of the unstable manifold $W^u$ of a single equilibrium $v = \mathcal{O}$ with Morse index $i(\mathcal{O}) =3$.
We call such an $\mathcal{A}_f$ a 3-\emph{ball Sturm attractor}.
Recall that we assume all equilibria $v_1, \ldots, v_N$ to be hyperbolic:
\emph{sinks} have Morse index $i=0$, \emph{saddles} have $i=1$, and \emph{sources}  $i=2$.
This terminology also applies when viewed within the flow-invariant and attracting boundary 2-sphere
	\begin{equation}
	\Sigma^2 = \partial W^u(\mathcal{O}):= \left(
	\mathrm{clos\,} W^u(\mathcal{O})\right) \smallsetminus
	W^u (\mathcal{O})\,.
	\label{eq:1.11}
	\end{equation}
Correspondingly we call the associated cells $c_v = W^u(v)$ of the Thom-Smale complex, or of any regular cell complex, \emph{vertices}, \emph{edges}, and {\emph{faces}.
The graph of vertices and edges, for example, defines the 1-skeleton $\mathcal{C}^1$ of the 3-ball cell complex $\mathcal{C} = \bigcup_v \, c_v$.

For 3-ball Sturm attractors, the signed hemisphere decomposition \eqref{eq:1.9a} reads
	\begin{equation}
	\Sigma^2 = \partial W^u (\mathcal{O})=
	\bigcup\limits_{j=0,1,2}^\bullet \Sigma_\pm^j\,.
	\label{eq:1.13}
	\end{equation}
Here $\Sigma_\pm^0 = \lbrace \mathbf{N}, \mathbf{S}\rbrace$ is the boundary of the one-dimensional fastest unstable manifold $W^1 = W^1(\mathcal{O})$, tangent to the positive eigenfunction $\varphi_0$ at $\mathcal{O}$.
Solutions $t \mapsto u(t,x)$ in $W^1$ are monotone in $t$, for any fixed $x$.
Accordingly
	\begin{equation}
	z(\mathbf{N}- \mathcal{O}) = 0_-\,, 
	\quad z(\mathbf{S}- \mathcal{O}) = 0_+\,,
	\label{eq:1.14}
	\end{equation}
i.e. $\mathbf{N} < \mathcal{O} < \mathbf{S}$.
The \emph{poles} $\mathbf{N},\mathbf{S}$ split the boundary circle $\Sigma^1 = \partial W^2 (\mathcal{O})$ of the 2-dimensional fast unstable manifold into the two \emph{meridian} half-circles $\Sigma_\pm^1$.
The circle $\Sigma^1$, in turn splits the boundary sphere $\Sigma^2 = \partial W^u(\mathcal{O})$ of the whole unstable manifold $W^u(\mathcal{O})$ into the two hemispheres $\Sigma_\pm^2$.
We recall the characterizing zero number property \eqref{eq:1.9f} for equilibria on the hemispheres $\Sigma_\pm^j (\mathcal{O})$, $j=0,1,2$.
This describes the signed hemisphere decomposition on the $\Sigma^2$ boundary in the 3-cell $c_{\mathcal{O}} = W^u(\mathcal{O})$ of the dynamic Thom-Smale complex \eqref{eq:1.9}, for any 3-ball Sturm attractor.

Any circular face boundary $\Sigma^1 = \Sigma^1(v) = \partial c_v$, $i(v)=2$, likewise possesses a decomposition
	\begin{equation}
	\Sigma^1 (v) = \bigcup\limits_{j=0,1}^\bullet \Sigma _\pm^j(v)\,.
	\label{eq:1.15}
	\end{equation}
The boundary circle $\partial c_v = \partial W^u(v)$ is split into two half circles by the two \emph{local poles} $\Sigma_\pm^0(v)$ which are located strictly below and above $v$.
Any 1-cell edge $c_v=W^u(v)$, $i(v)=1$ of a saddle $v$, finally, possesses two boundary equilibria,
	\begin{equation}
	\Sigma^0 (v) = \Sigma_-^0(v)\, \dot{\cup} \, \Sigma_+^0(v)\,.
	\label{eq:1.16}
	\end{equation}
The dynamics on $c_v$ is strictly monotone, with $z=0$ there.
The signed hemisphere decompositions describe a refined, signed version of the dynamic Thom-Smale complex of $\mathcal{A}_f = \mathrm{clos\,} W^u(\mathcal{O})$.

Only in the Chafee-Infante case is each hemisphere $\Sigma_\pm^j(\mathcal{O})$ given by the unstable manifold of a single equilibrium.
We formalize the general structure as follows.

\begin{defi}\label{def:1.1}
Let $\mathcal{A} = \mathcal{A}_f$ be a Sturm global attractor with equilibrium set $\mathcal{E}$, all hyperbolic.
For any $v \in \mathcal{E}$, $0 \leq j < i(v)$, let
	\begin{equation}
	\mathcal{E}_\pm^j(v) := \mathcal{E} \cap \Sigma_\pm^j(v)
	\label{eq:1.9g}
	\end{equation}
denote the equilibria in the hemispheres $\Sigma_\pm^j(v)$.
The sets $\mathcal{E}_\pm^j(v)$, for fixed $v$, partition the target set of equilibria $\Tilde{v}$ which $v$ connects to heteroclinically, $v \leadsto \Tilde{v}$.
We call these partitions, including their labels $v$, $j$, and $\pm$, the \emph{signed hemisphere template} of $\mathcal{A}$.

In the special case of a 3-ball Sturm attractor $\mathcal{A}$ we call the partitions $\mathcal{E}_\pm^j(v)$ the \emph{signed 2-hemisphere template}.
\end{defi}

The above signed hemisphere template structure is entirely discrete.
Indeed, the characterization of the hemisphere equilibrium sets $\mathcal{E}_\pm^j(v)$ in proposition~3.1 will assert
	\begin{equation} 
	\Tilde{v} \in \mathcal{E}_\pm^j(v) \quad \Longleftrightarrow \quad
	( v \leadsto \Tilde{v} \quad \text{and} \quad
	z(\Tilde{v}-v) = j_\pm ) \,.
	\label{eq:1.9h}
	\end{equation}
Since heteroclinicity $v \leadsto \Tilde{v}$ can be decided based on signed zero numbers, as well, the signed hemisphere structure can be viewed as contained in, but possibly coarser than, the \emph{signed zero matrix} of all signed zero numbers $z(\Tilde{v}-v)$, for $(\Tilde{v},v) \in \mathcal{E} \times \mathcal{E}$, including the Morse entries $i(v)$ on the diagonal $(v,v) \in \mathcal{E} \times \mathcal{E}$.

For the geometric characterization of 3-ball Sturm attractors $\mathcal{A}_f$ in \eqref{eq:1.10}, by their dynamic Thom-Smale complexes \eqref{eq:1.9}, we now drop all Sturmian PDE interpretations.
Instead we define 3-cell templates, abstractly, in the class of regular cell complexes and without any reference to PDE or dynamics terminology.
See fig.~\ref{fig:1.1} for an illustration.
In theorem~\ref{thm:4.1} below, we will then claim that the dynamic Thom-Smale complex $c_v = W^u(v)$ of any 3-ball Sturm attractor $\mathcal{A}_f$ indeed provides a 3-cell template.

\begin{figure}[t!]
\centering \includegraphics[width=\textwidth]{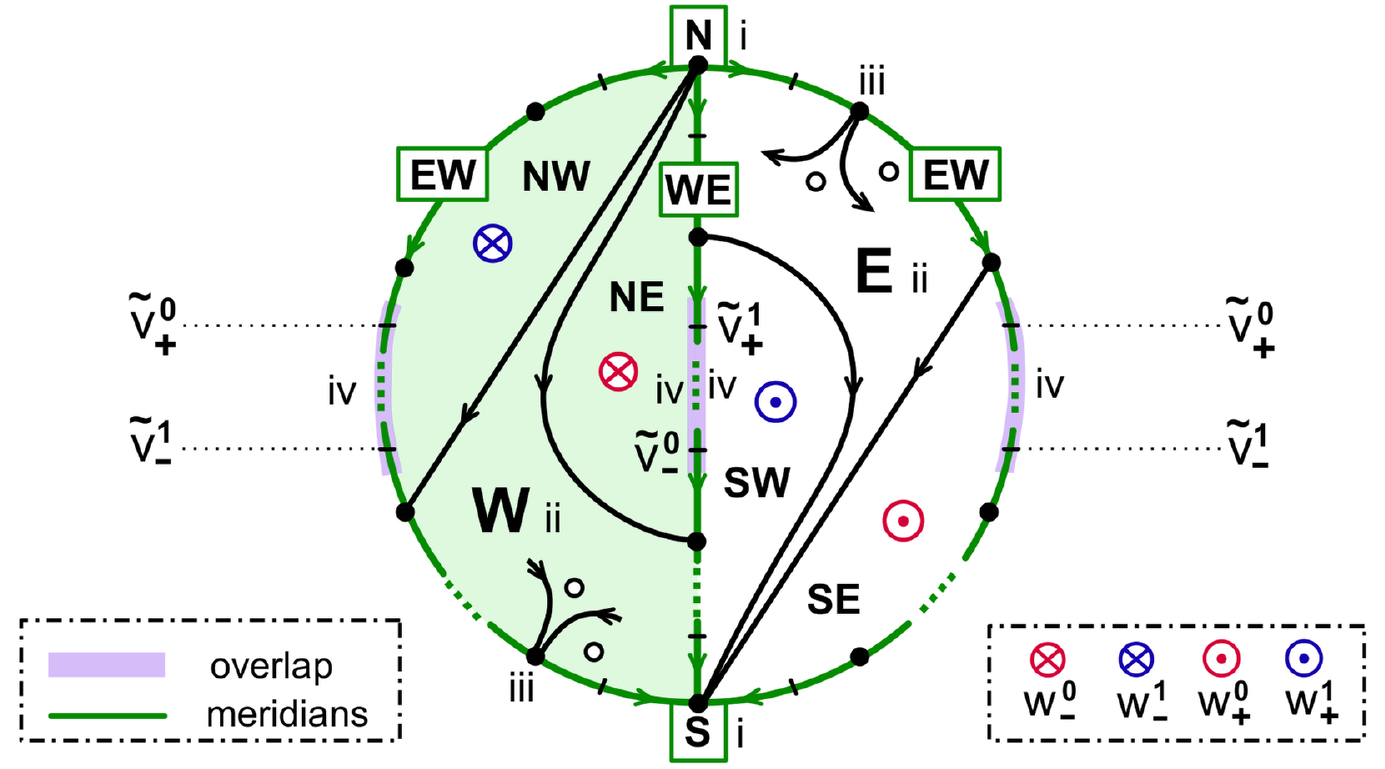}
\caption{\emph{
A 3-cell template. Shown is the $S^2$ boundary of the single 3-cell $c_\mathcal{O}$ with poles $\mathbf{N}$, $\mathbf{S}$, hemispheres $\mathbf{W}$ (green), $\mathbf{E}$ and separating meridians $\mathbf{EW}$, $\mathbf{WE}$ (green).
The right and the left boundaries denote the same $\mathbf{EW}$ meridian and have to be identified.
Dots $\bullet$ are sinks, and small circles $\circ$ are sources.
Note the hemisphere decomposition~(ii), the edge orientations~(iii) at meridian boundaries, and the meridian overlaps~(iv) of the $\mathbf{N}$-adjacent meridian faces $\otimes = w_-^\iota$ with their $\mathbf{S}$-adjacent counterparts $\odot =w_+^\iota$.
For $w_\pm^\iota, \Tilde{v}_\pm^\iota$ see also \eqref{eq:1.24a}, corollary~\ref{cor:4.4}, and fig.~\ref{fig:4.2}.
For specific examples see figs.~\ref{fig:5.2}, \ref{fig:6.1}, \ref{fig:6.3}.
}}
\label{fig:1.1}
\end{figure}

\begin{defi}\label{def:1.2}
A finite disjoint union $\mathcal{C} = \bigcup_{v \in \mathcal{E}} c_v$ of cells $c_v$ is called a \emph{3-cell template} if $\mathcal{C}$ is a regular cell complex and the following four conditions all hold.
\begin{itemize}
\item[(i)] $\mathcal{C} = \mathrm{clos\,} c_{\mathcal{O}}= S^2 \,\dot{\cup}\, c_{\mathcal{O}}$ is the closure of a single 3-cell $c_{\mathcal{O}}$.
\item[(ii)] The 1-skeleton $\mathcal{C}^1$ of $\mathcal{C}$ possesses a \emph{bipolar orientation} from a pole vertex $\mathbf{N}$ (North) to a pole vertex $\mathbf{S}$ (South), with two disjoint directed \emph{meridian paths} $\mathbf{WE}$ and $\mathbf{EW}$ from $\mathbf{N}$ to $\mathbf{S}$.
The circle of meridians decomposes the boundary sphere $S^2$ into remaining \emph{hemisphere} components $\mathbf{W}$ (West) and $\mathbf{E}$ (East).
\item[(iii)] Edges are oriented towards the meridians, in $\mathbf{W}$, and away from the meridians, in $\mathbf{E}$, at end points on the meridians other than the poles $\mathbf{N}$, $\mathbf{S}$.
\item[(iv)] Let $\mathbf{NE}$, $\mathbf{SW}$ denote the unique faces in $\mathbf{W}$, $\mathbf{E}$, respectively, which contain the first, last edge of the meridian $\mathbf{WE}$ in their boundary.
Then the boundaries of $\mathbf{NE}$ and $\mathbf{SW}$ overlap in at least one shared edge of the meridian $\mathbf{WE}$.

Similarly, let $\mathbf{NW}$, $\mathbf{SE}$ denote the unique faces in $\mathbf{W}$, $\mathbf{E}$, adjacent to the first, last edge of the other meridian $\mathbf{EW}$, respectively.
Then their boundaries overlap in at least one shared edge of $\mathbf{EW}$.
\end{itemize}
\end{defi}

We recall here that an edge orientation of the 1-skeleton $\mathcal{C}^1$ is called bipolar if it does not contain directed cycles, and possesses a single ``source'' vertex $\mathbf{N}$ and a single ``sink'' vertex $\mathbf{S}$, both on the boundary of $\mathcal{C}$. 
Here ``source'' and ``sink'' are understood, not dynamically but, with respect to edge orientation.
To avoid any confusion with dynamic $i=0$ sinks and $i=2$ sources, below, we call $\mathbf{N}$ and $\mathbf{S}$ the North and South pole, respectively. See \cite{fretal95} for a survey on a closely related notion of bipolarity.

With definitions~\ref{def:1.1} and \ref{def:1.2} at hand, we can now formulate the passage from 3-ball Sturm attractors $\mathcal{A}$ to 3-cell templates $\mathcal{C}$ as the passage
	\begin{equation}
	\text{signed 2-hemisphere template} \quad \Longrightarrow \quad
	\text{3-cell template}\,.
	\label{eq:1.12}
	\end{equation} 
The hemisphere translation table between $\mathcal{A}$ and $\mathcal{C}$ will be the following:

\begin{equation}
\begin{aligned}
(\Sigma_-^0, \Sigma_+^0) \quad &\mapsto \quad (\mathbf{N}, \mathbf{S})\\
(\Sigma_-^1, \Sigma_+^1) \quad &\mapsto \quad (\mathbf{EW}, \mathbf{WE})\\
(\Sigma_-^2, \Sigma_+^2) \quad &\mapsto \quad (\mathbf{W}, \mathbf{E})
\end{aligned}
\label{eq:1.17}
\end{equation}

Here $\Sigma_\pm^j$ abbreviates $\Sigma_\pm^j(\mathcal{O})$. Theorem~\ref{thm:4.1} below asserts that the finite regular dynamic Thom-Smale complex $c_v= W^u(v)$ of $\mathcal{A}$, with the above translation of the hemisphere decomposition of $\partial W^u(\mathcal{O})$, indeed satisfies conditions (i)--(iv) of definition~\ref{def:1.2} on a 3-cell template.
We already note here that the 3-cell condition~(i) on $c_{\mathcal{O}} = W^u(\mathcal{O})$ is obviously satisfied.
The bipolar orientation~(ii) of the edges $c_v$ of the 1-skeleton, alias the one-dimensional unstable manifolds $c_v = W^u(v)$ of $i(v)=1$ saddles $v$, is simply the strict monotone order from vertex $\Sigma_-^0 (v)$ to vertex $\Sigma_+^0(v)$, uniformly for $0 \leq x \leq 1$.

In \cite{firo14} we have already shown how any 3-cell regular complex, i.e. any regular cell complex satisfying definition~\ref{def:1.1}(i), does appear as the dynamic Thom-Smale complex of \emph{some} 3-ball Sturm attractor with these prescribed cells as unstable manifolds.
The complete characterization of 3-ball Sturm attractors by the remaining, more specific, orientation and decomposition conditions (ii)--(iv) was not discussed there.

The second implication which we address in the present paper is the passage
	\begin{equation}
	\text{3-cell template} \quad \Longrightarrow \quad
	\text{3-meander template}\,.
	\label{eq:1.18}
	\end{equation} 
As for 3-cell templates, we temporarily ignore all Sturm attractor connotations and define 3-meander templates, abstractly, without any reference to ODE shooting.

Abstractly, a \emph{meander} is an oriented planar $C^1$ Jordan curve $\mathcal{M}$ which crosses a positively oriented horizontal axis at finitely many points.
The curve $\mathcal{M}$ is assumed to run from Southwest to Northeast, asymptotically, and all $N$ crossings are assumed to be transverse; see \cite{ar88, arvi89}.
Note $N$ is odd.
Enumerating the $N$ crossing points $v \in \mathcal{E}$, by $h_0$ along the meander $\mathcal{M}$ and by $h_1$ along the horizontal axis, respectively, we obtain two labeling bijections
	\begin{equation}
	h_0,h_1: \quad \lbrace 1, \ldots , N \rbrace \rightarrow \mathcal{E}\,.
	\label{eq:1.19a}
	\end{equation}
Define the \emph{meander permutation} $\sigma \in S_N$ as 
	\begin{equation}
	\sigma := h_0^{-1} \circ h_1.
	\label{eq:1.19b}
	\end{equation}
We call the meander $\mathcal{M}$ \emph{dissipative} if
	\begin{equation}
	\sigma(1) =1, \quad \sigma(N) =N
	\label{eq:1.20}
	\end{equation}
are fixed under $\sigma$.

For $\mathcal{M}$-adjacent crossings $v=h_0(m)$, $\Tilde{v}= h_0(m+1)$ we define \emph{Morse numbers} $i_{\Tilde{v}}$, $i_v$, such that
	\begin{equation}
	i_{\Tilde{v}} \,=\, i_v +(-1)^{m+1} \,
	\text{sign} (h_1^{-1} (\Tilde{v}) -h_1^{-1} (v))\,.
	\label{eq:1.22a}
	\end{equation}
Recursively, this defines all Morse numbers $i_v$ of the meander $\mathcal{M}$ uniquely, with any one of the two equivalent normalizations
	\begin{equation}
	i_{h_0(1)}= 0\,, \quad i_{h_0(N)}=0\,.
	\label{eq:1.23}
	\end{equation}
See \eqref{eq:1.26} for adjacent crossings on the $h_1$-axis. We call the meander $\mathcal{M}$ \emph{Morse}, if
	\begin{equation}
	i_v \geq 0\,,
	\label{eq:1.24}
	\end{equation}
for all $v \in \mathcal{M}$.

We call $\mathcal{M}$ \emph{Sturm meander}, if $\mathcal{M}$ is a dissipative Morse meander; see \cite{firo96}. 
Conversely, given any permutation $\sigma \in S_N$, we label $N$ crossings along the axis in the order of $\sigma$. Define an associated curve $\mathcal{M}$ of arches over the horizontal axis which switches sides at the labels $\{1,\ldots,N\}$, successively. This fixes $h_0=\mathrm{id}$ and $h_1=\sigma$.
A \emph{Sturm permutation} $\sigma$ is a permutation such that the associated curve $\mathcal{M}$ is a Sturm meander.
The main paradigm of \cite{firo99} is the equivalence of Sturm meanders $\mathcal{M}$ with shooting curves $\mathcal{M}_f$ of the Neumann ODE problem \eqref{eq:1.3}.
In fact, the Neumann shooting curve is a Sturm meander, for any dissipative nonlinearity $f$ with hyperbolic equilibria.
Conversely, for any permutation $\sigma$ of a Sturm meander $\mathcal{M}$ there exist dissipative $f$ with hyperbolic equilibria such that $\sigma = \sigma_f$ is the Sturm permutation of $f$.
In that case, the intersections $v$ of the meander $\mathcal{M}_f$ with the horizontal $v$-axis are the boundary values of the equilibria $v \in \mathcal{E}_f$ at $x=1$, and the Morse number
	\begin{equation}
	i_v = i(v)
	\label{eq:1.22b}
	\end{equation}
is the Morse index of $v$.
For that reason we have used closely related notation to describe either case.

In particular, \eqref{eq:1.22a} extends the terminology of \emph{sinks} $i_v =0$, \emph{saddles} $i_v =1$, and \emph{sources} $i_v=2$ to abstract Sturm meanders.
We insist, however, that our above definition~\eqref{eq:1.19a}--\eqref{eq:1.24} is completely abstract and independent of this ODE/PDE interpretation.

We return to abstract Sturm meanders $\mathcal{M}$ as in \eqref{eq:1.19a}--\eqref{eq:1.24}.
For example, consider the case $i_{\mathcal{O}} =3$ of a single intersection $v= \mathcal{O}$ with Morse number $3$.
Suppose $i_v \leq 2$ for all other Morse numbers. Then \eqref{eq:1.22a} implies $i=2$ for the two $h_0$-\emph{neighbors} $h_0(h_0^{-1} (\mathcal{O})\pm1)$ of $\mathcal{O}$ along the meander $\mathcal{M}$.
In other words, these neighbors are both sources.
The same statement holds true for the two $h_1$-\emph{neighbors} $h_1(h_1^{-1}(\mathcal{O})\pm 1)$ of $\mathcal{O}$ along the horizontal axis.
To fix notation, we denote these $h_\iota$-neighbors by
	\begin{equation}
	w_\pm^\iota:= h_\iota (h_\iota^{-1} (\mathcal{O}) \pm 1)\,,
	\label{eq:1.24a}
	\end{equation}
for $\iota = 0,1$.
The $h_\iota$-\emph{extreme sources} are the first and last source intersections $v$ of the meander $\mathcal{M}$ with the horizontal axis, in the order of $h_\iota$.

\begin{figure}[t!]
\centering \includegraphics[width=\textwidth]{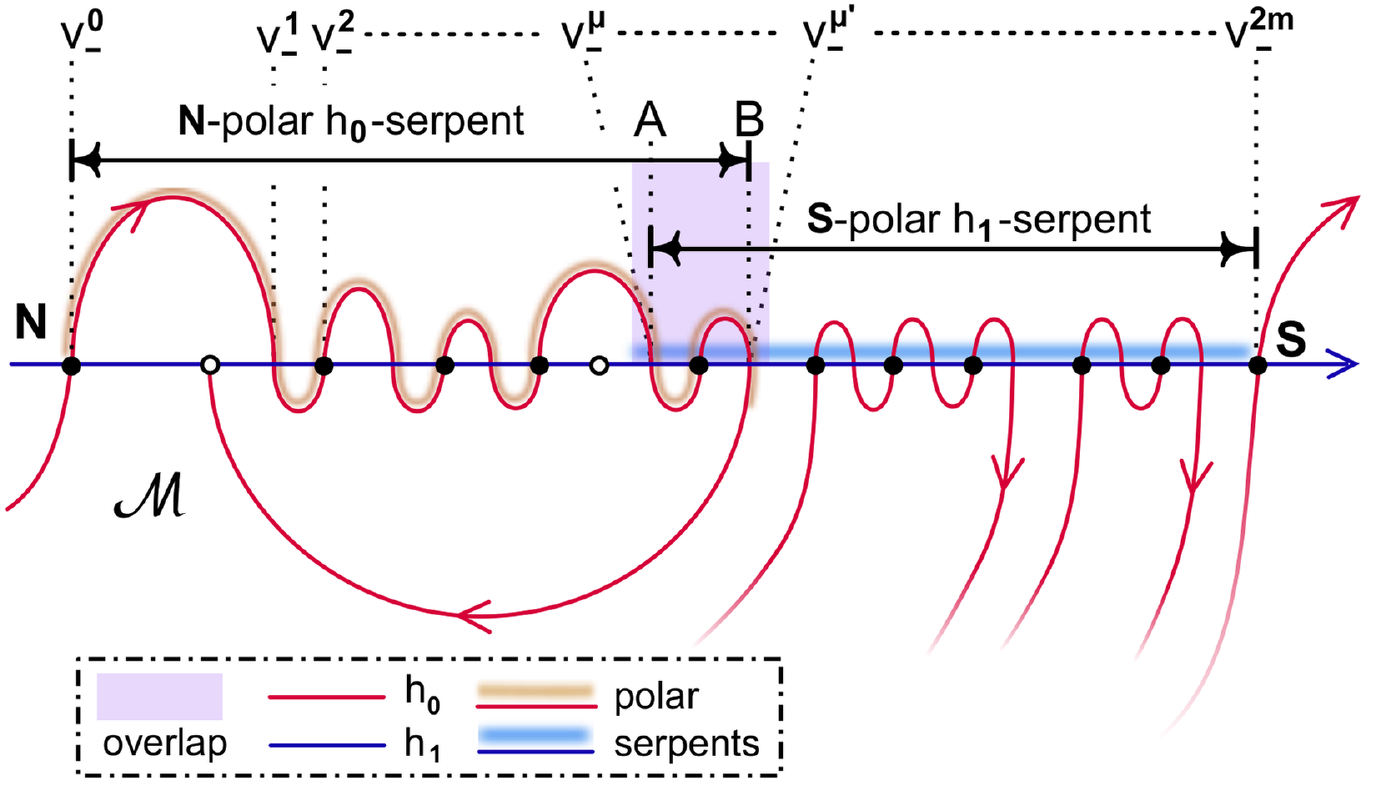}
\caption{\emph{
An $\mathbf{N}$-polar $h_0$-serpent with last axis intersection at the saddle $B=v_-^{\mu'}$, and an anti-polar, i.e. $\mathbf{S}$-polar, $h_1$-serpent with first meander intersection at the saddle $A=v_-^{\mu}$.
Solid dots $\bullet$ indicate sinks, $i=0$, and small circles $\circ$ denote sources, $i=2$.
Saddle crossings are not marked.
Note how $B$ is succeeded by an $i=2$ source, along the meander arc of $h_0$, by maximality of the polar $h_0$-serpent.
Similarly, $A$ is preceded by an $i=2$ source along the horizontal $h_1$-axis.
All polar $h_0$-serpents are oriented left to right.
The polar $h_0$-serpent overlaps its anti-polar $h_1$-serpent from $A$ to $B$.
See the definition in \eqref{eq:1.25}.
}}
\label{fig:1.2}
\end{figure}

Reminiscent of cell template terminology, we call the extreme sinks $\mathbf{N} = h_0(1) = h_1(1)$ and $\mathbf{S} = h_0(N) = h_1(N)$ the (North and South) \emph{poles} of the Sturm meander $\mathcal{M}$.
A \emph{polar} $h_\iota$-\emph{serpent}, for $\iota =0,1$, is a set of $v =h_\iota (m) \in \mathcal{E}$ for a maximal interval of integers $m$ which contains a pole, $\mathbf{N}$ or $\mathbf{S}$, and satisfies
	\begin{equation}
	i_{h_\iota(m)} \in \lbrace 0,1\rbrace
	\label{eq:1.25}
	\end{equation}
for all $m$.
To visualize the serpent we often include the meander or axis path joining the elements $v$ of the serpent.
To determine $h_1$-serpents, the following variant of \eqref{eq:1.22a} for $h_1$-neighbors $v = h_1(m)$, $\Tilde{v} = h_1(m+1)$ is useful:
	\begin{equation}
	i_{\Tilde{v}} = i_v +(-1)^{m+1} \,
	\text{sign} (h_0^{-1} (\Tilde{v}) -h_0^{-1} (v))\,.
	\label{eq:1.26}
	\end{equation}
See figs.~\ref{fig:1.2} and \ref{fig:1.3} for examples.
We call $\mathbf{N}$-polar serpents and $\mathbf{S}$-polar serpents anti-polar to each other.
An \emph{overlap} of anti-polar serpents simply indicates a nonempty intersection.
For later reference, we call a polar $h_\iota$-serpent \emph{full} if it extends all the way to the saddle which is $h_{1-\iota}$-adjacent to the opposite pole.
Full $h_\iota$-serpents always overlap with their anti-polar $h_{1-\iota}$-serpent, of course, at least at that saddle.

\begin{figure}[t!]
\centering \includegraphics[width=\textwidth]{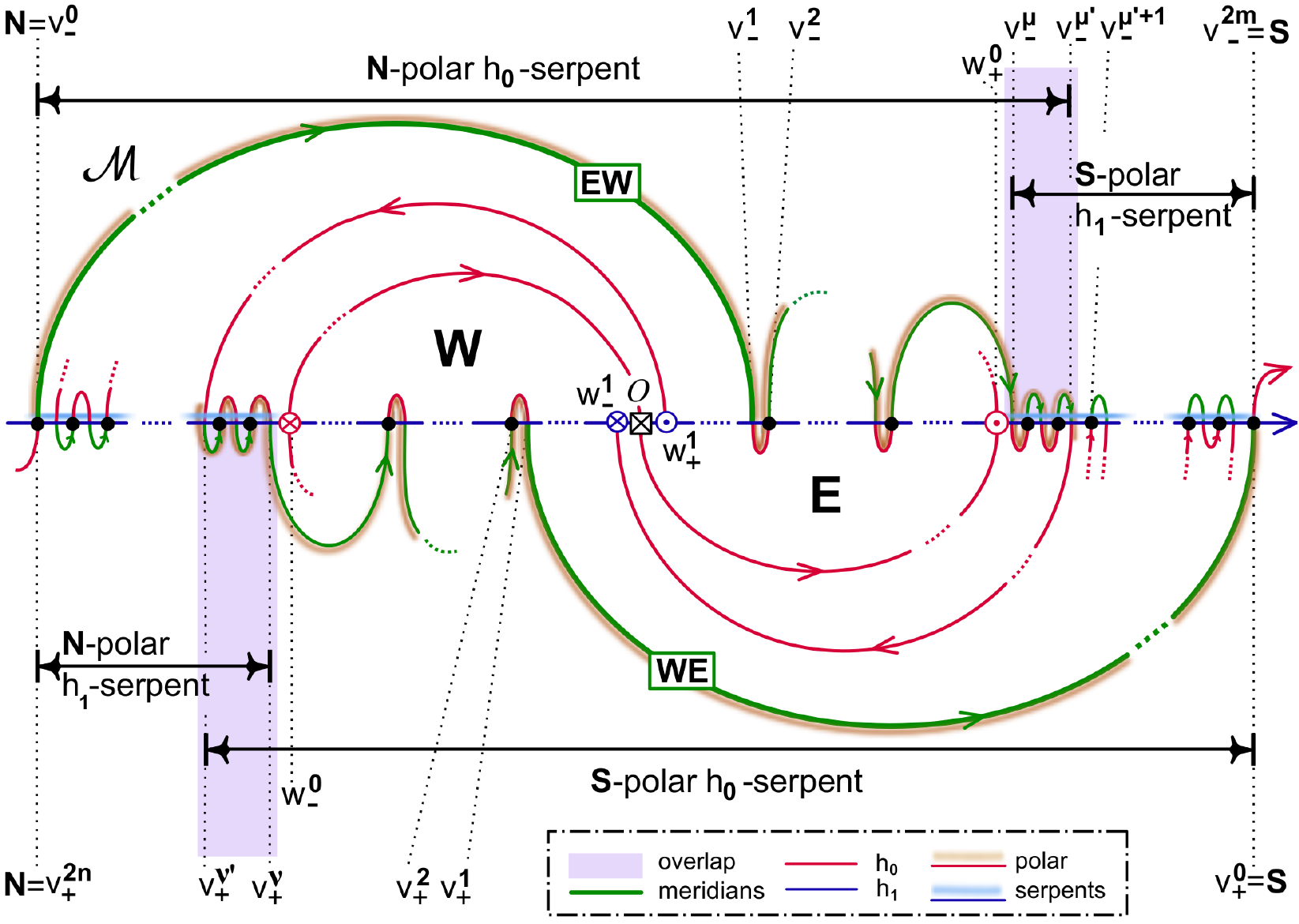}
\caption{\emph{
A 3-meander template.
Note how the $\mathbf{N}$-polar $h_1$-serpent $\mathbf{N} = v_+^{2n} \ldots  v_+^\nu$ is terminated at $v_+^\nu$ by the $h_1$-adjacent source $w_-^0$ which is, both, $h_1$-extreme minimal and the lower $h_0$-neighbor of $\mathcal{O}$.
From $v_+^{\nu '}$ to $v_+^\nu$, this serpent overlaps the anti-polar, i.e. $\mathbf{S}$-polar, $h_0$-serpent $v_+^{\nu '} \ldots  v_+^\nu \ldots  v_+^0  = \mathbf{S}$.
Similarly, the $\mathbf{N}$-polar $h_0$-serpent $\mathbf{N} = v_-^0 \ldots  v_-^{\mu '}$ overlaps the anti-polar, i.e. $\mathbf{S}$-polar, $h_1$-serpent $v_-^\mu  \ldots  v_-^{\mu '} \ldots  v_-^{2n} = \mathbf{S}$, from $v_-^\mu$ to $v_-^{\mu '}$.
The $h_1$-neighbors $w_\pm^1$ of $\mathcal{O}$ are  the $h_0$-extreme sources, by the two polar $h_0$-serpents.
Similarly, the $h_0$-neighbors $w_\pm^0$ of $\mathcal{O}$ define the $h_1$-extreme sources.
See also  fig.~\ref{fig:6.4} for a specific example.
}}
\label{fig:1.3}
\end{figure}

\begin{defi}\label{def:1.3}
An abstract Sturm meander $\mathcal{M}$ with axis intersections $v \in \mathcal{E}$ is called a \emph{3-meander template} if the following four conditions hold, for $\iota = 0,1$.

\begin{itemize}
\item[(i)] $\mathcal{M}$ possesses a single axis intersection $v = \mathcal{O}$ with Morse number $i_{\mathcal{O}} =3$, and  no other Morse number exceeds $2$.
\item[(ii)] Polar $h_\iota$-serpents overlap with their anti-polar $h_{1-\iota}$-serpents in at least one shared vertex.
\item[(iii)] The intersection $v=\mathcal{O}$ is located between the two intersection points, in the order of $h_{1-\iota}$, of the polar arc of any polar $h_\iota$-serpent.
\item[(iv)] The $h_\iota$-neighbors $w^\iota_\pm$ of $v=\mathcal{O}$ are the $i=2$ sources which terminate the polar $h_{1-\iota}$-serpents.
\end{itemize}
\end{defi}

See fig.~\ref{fig:1.3} for an illustration of 3-meander templates.
Property (iv), for example, asserts that the $h_\iota$-neighbor sources $w_\pm^\iota$ of $\mathcal{O}$ are the $h_{1-\iota}$-extreme sources, for $\iota = 0,1$. For the Sturm boundary orders $h^f_\iota$
this is a useful exercise in polar serpents, as we will show in lemma~\ref{lem:4.3}(iii) and \eqref{eq:4.24} below.

In theorem~\ref{thm:5.2} below we will establish the passage
	\begin{equation}
	\text{3-cell template} \quad \Longrightarrow \quad
	\text{3-meander template}\,.
	\label{eq:1.27}
	\end{equation} 
This is based on a detailed construction of paths $h_0$ and $h_1$ in the given 3-cell template.
The construction relies heavily on our trilogy \cite{firo09, firo08, firo10} for the planar case.
In fact we construct $h_0$ and $h_1$, separately, for each closed hemisphere $\mathrm{clos\,} \mathbf{W}$ and $\mathrm{clos\,} \mathbf{E}$.
Each closed hemisphere disk, by itself, can be viewed as a planar Sturm attractor.
In section~\ref{sec5} we then weld the closed hemispheres $\mathrm{clos\,} \mathbf{W}$ and $\mathrm{clos\,} \mathbf{E}$ along their meridian boundaries, and stitch the planar hemisphere meanders, to explicitly derive the 3-meander template.
Although this step is pervasively motivated by its ODE and PDE background, it proceeds in the abstract setting of 3-cell templates and 3-meander templates, entirely.

We conclude the paper with a nonexistence result for the solid 3-dimensional  octahedron $\mathbb{O}$, in section~\ref{sec6}.
In fact, choose the poles $\Sigma_\pm^0 = \lbrace v_\pm \rbrace$, alias $\mathbf{N}$ and $\mathbf{S}$, to be antipodal sink vertices of the octahedron.
In view of dissipativeness, this extremal choice for the monotone $z=0$ order may appear most natural.
Surprisingly, however, it is then impossible to choose any bipolar orientation of the octahedral 1-skeleton, from $\mathbf{N}$ to $\mathbf{S}$, and a meridian decomposition into hemispheres $\mathbf{W}$, $\mathbf{E}$ such that the octahedron $\mathbb{O}$ becomes a 3-cell template in the sense of definition~\ref{def:1.2}.
Theorem~\ref{thm:4.1} therefore defeats our antipodal choice of the poles $\mathbf{N}, \mathbf{S}$ for octahedral Sturm 3-balls,
because the dynamic Thom-Smale complex $\mathcal{C}$ of any octahedral Sturm 3-ball attractor with antipodal extreme equilibria $\mathbf{N}, \mathbf{S}$ would have to provide just such a 3-cell template.

In \cite{firo14}, on the other hand, we proved that any regular cell complex which is the closure of a single 3-cell $c_{\mathcal{O}}$ actually does possess a realization as the regular dynamic Thom-Smale complex of some Sturm 3-ball global attractor.
Our construction there amounts to poles which are adjacent corners of a single octahedral surface triangle.
That triangle face constitutes the whole Western hemisphere $\mathbf{W} = \Sigma_-^2$; see fig.~6.3 in \cite{firo14}.

We conclude our long introduction with a brief preview of the remaining two papers of our 3-ball trilogy.
In \cite{firo3d-2} we further explore the 3-meander template $\mathcal{M}$ of definition~\ref{def:1.3}.
Since $\mathcal{M}$ is a Sturm meander, $\mathcal{M}$ defines a Sturm global attractor $\mathcal{A}$ which turns out to be a 3-ball Sturm attractor.
More precisely, the meander $\mathcal{M}$ determines heteroclinic connectivity and signed zero numbers between all equilibria:
	\begin{equation}
	\text{3-meander template} \quad \Longrightarrow \quad
	\text{signed 2-hemisphere template}\,.
	\label{eq:1.28}
	\end{equation} 
Invoking steps \eqref{eq:1.27}, \eqref{eq:1.28}, and \eqref{eq:1.12}, in this order, provides a dynamic Thom-Smale complex $\mathcal{C}_f$ which originates from an abstractly prescribed 3-cell complex $\mathcal{C}$. Here the dissipative nonlinearity $f$ in \eqref{eq:1.1} is chosen such that $\sigma_f=\sigma$ is the Sturm permutation associated to the 3-meander template $\mathcal{M}$ of \eqref{eq:1.27}. The cell complexes $\mathcal{C}_f$ and $\mathcal{C}$ coincide by a cell-to-cell homeomorphism. This completes the design of a Sturm 3-ball global attractor $\mathcal{A}_f$ with prescribed dynamic  Thom-Smale complex $\mathcal{C}_f=\mathcal{C}$.

In \cite{firo3d-3} we collect many further examples to illustrate our theory.
In particular we construct all solid 3-dimensional tetrahedra, octahedra, and cubes, together with their bipolar orientations and meridian decompositions, as Sturm global attractors.
We also construct all Sturm 3-balls with at most 13 equilibria, and discuss some first steps towards a characterization of all 3-dimensional Sturm global attractors, with more than a single 3-cell.


\textbf{Acknowledgments.}
With great pleasure we express our profound gratitude to  Wal{\-}dyr~M.~Oliva, whose deep geometric insights and friendly challenges are a visible inspiration for us since so many years.
Extended delightful hospitality by the authors is mutually acknowledged.
Suggestions concerning the Thom-Smale complex were generously provided by Jean-Michel Bismut.
Gustavo~Granja generously shared his deeply topological view point, precise references included.
Anna~Karnauhova has contributed all illustrations with great patience, ambition, and her inimitable artistic touch.
Typesetting was expertly accomplished by Ulrike~Geiger. This work
was partially supported by DFG/Germany through SFB 910 project A4 and by FCT/Portugal through project UID/MAT/04459/2013.


\section{Planar Sturm attractors}
\label{sec2}

As a prelude to 3-ball Sturm global attractors we review the planar case, in theorem~\ref{thm:2.1}.
A central construction, in definition~\ref{def:2.2}, assigns a ZS-Hamiltonian pair of paths $h_0, h_1$: $\lbrace 1, \ldots , N \rbrace \rightarrow \mathcal{E}$ through the equilibrium vertices of a prescribed planar bipolar cell complex $\mathcal{C}$.
The construction of $h_0, h_1$ ensures that the permutation $\sigma$:= $h_0^{-1} \circ h_1 \in S_N$ is Sturm, and hence defines a Sturm meander $\mathcal{M}$.
Moreover, the associated Sturm global attractor is planar with dynamic Thom-Smale complex as prescribed by $\mathcal{C}$.
See theorem~\ref{thm:2.4}.
We also discuss in what sense $h_0, h_1$ are unique.
We conclude the section with a special class of planar Sturm disk attractors which we call $\mathbf{W}$- and $\mathbf{E}$-cell templates, for ``West'' and ``East''.
They feature full polar serpents and will serve as closed hemispheres $\mathrm{clos\,} \Sigma_\pm^2(\mathcal{O})$, welded along their shared meridian boundary circle, in 3-ball Sturm global attractors $\mathcal{A}_f = \mathrm{clos\,} W^u(\mathcal{O})$.

In \cite[theorem~1.2]{firo14}, we combined the planar results of \cite{firo09, firo08, firo10} with the Schoenflies result \cite{firo13} as follows.

\begin{thm}\label{thm:2.1}
A regular finite cell complex $\mathcal{C}$ is the dynamic Thom-Smale complex of a planar Sturm global attractor if, and only if, $\mathcal{C} \subseteq \mathbb{R}^2$ is planar, contractible, and the 1-skeleton $\mathcal{C}^1$ of $\mathcal{C}$ possesses a bipolar orientation.
\end{thm}

Both poles $\mathbf{N}$, $\mathbf{S}$ of the bipolar orientation are required, here, to lie on the boundary of the planar embedding $\mathcal{C} \subseteq \mathbb{R}^2$.
We say that the bipolar orientation runs from $\mathbf{N}$ to $\mathbf{S}$.
See fig.~\ref{fig:2.1} for a simple disk example, and \cite{firo10} for a planar octahedral complex with edge adjacent poles $\mathbf{N}$ and $\mathbf{S}$.

Given a planar Sturm global attractor, the bipolar orientation of the 1-skeleton $\mathcal{C}^1$ is easily defined.
Edges are the one-dimensional unstable manifolds $W^u(v)$ of saddles 
$i(v) =1$.
On $W^u(v)$ we have $z(u^1-u^2)=0$, for any two nonidentical spatial profiles $x \mapsto u^\iota(x)$.
Therefore the spatial profiles $u(x)$ in $W^u(v)$ are totally ordered, strictly monotonically, uniformly for any fixed $0 \leq x \leq 1$.
We may orient the edge towards increasing $u$.
This definition can also be derived from just the signed hemisphere template $\Sigma_\pm^j(v)$ of $\mathcal{A}$, as an orientation of $W^u(v)$ from $\Sigma_-^0(v)$ to $\Sigma_+^0(v)$; see our comments to \eqref{eq:1.13}, \eqref{eq:1.12}.

Conversely suppose we are given the planar regular complex $\mathcal{C}$ with bipolar orientation of $\mathcal{C}^1$.
To label the vertices $v \in \mathcal{E}$ of $\mathcal{C}$, we construct a pair of Hamiltonian paths
	\begin{equation}
	h_0, h_1: \quad \lbrace 1, \ldots , N \rbrace \rightarrow \mathcal{E}
	\label{eq:2.1}
	\end{equation}
as follows.
Let $\mathcal{O}$ indicate any source, i.e. (the barycenter of) a 2-cell  face $c_{\mathcal{O}}$ in $\mathcal{C}$.
By planarity of $\mathcal{C}$ it turns out that the bipolar orientation of $\mathcal{C}^1$ defines unique extrema on the boundary circle $\partial c_{\mathcal{O}}$ of the 2-cell $c_\mathcal{O}$.
Let $w_-^0$ be the saddle on $\partial c_{\mathcal{O}}$ (of the edge) to the right of the minimum, and $w_+^0$ the saddle to the left of the maximum.
Similarly, let $w_-^1$ be the saddle to the left of the minimum, and $w_+^1$ to the right of the maximum.
See fig.~\ref{fig:2.1}.

\begin{figure}[t!]
\centering \includegraphics[width=0.43\textwidth]{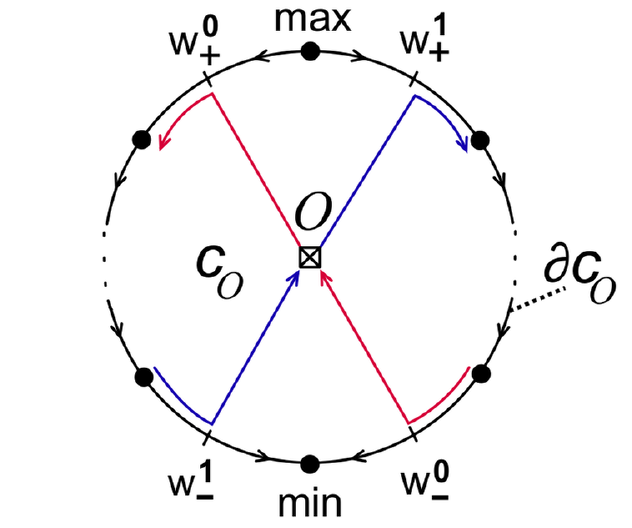}
\caption{\emph{
Traversing a face vertex $\mathcal{O}$ by a ZS-pair $h_0, h_1$.
Note the resulting shapes ``Z'' of $h_0$ (red) and ``S'' of $h_1$ (blue).
The paths $h_\iota$ may also continue into neighboring faces, beyond $w_\pm^\iota$, without turning into the face boundary $\partial c_{\mathcal{O}}$.
}}
\label{fig:2.1}
\end{figure}

\begin{figure}[]
\centering \includegraphics[width=\textwidth]{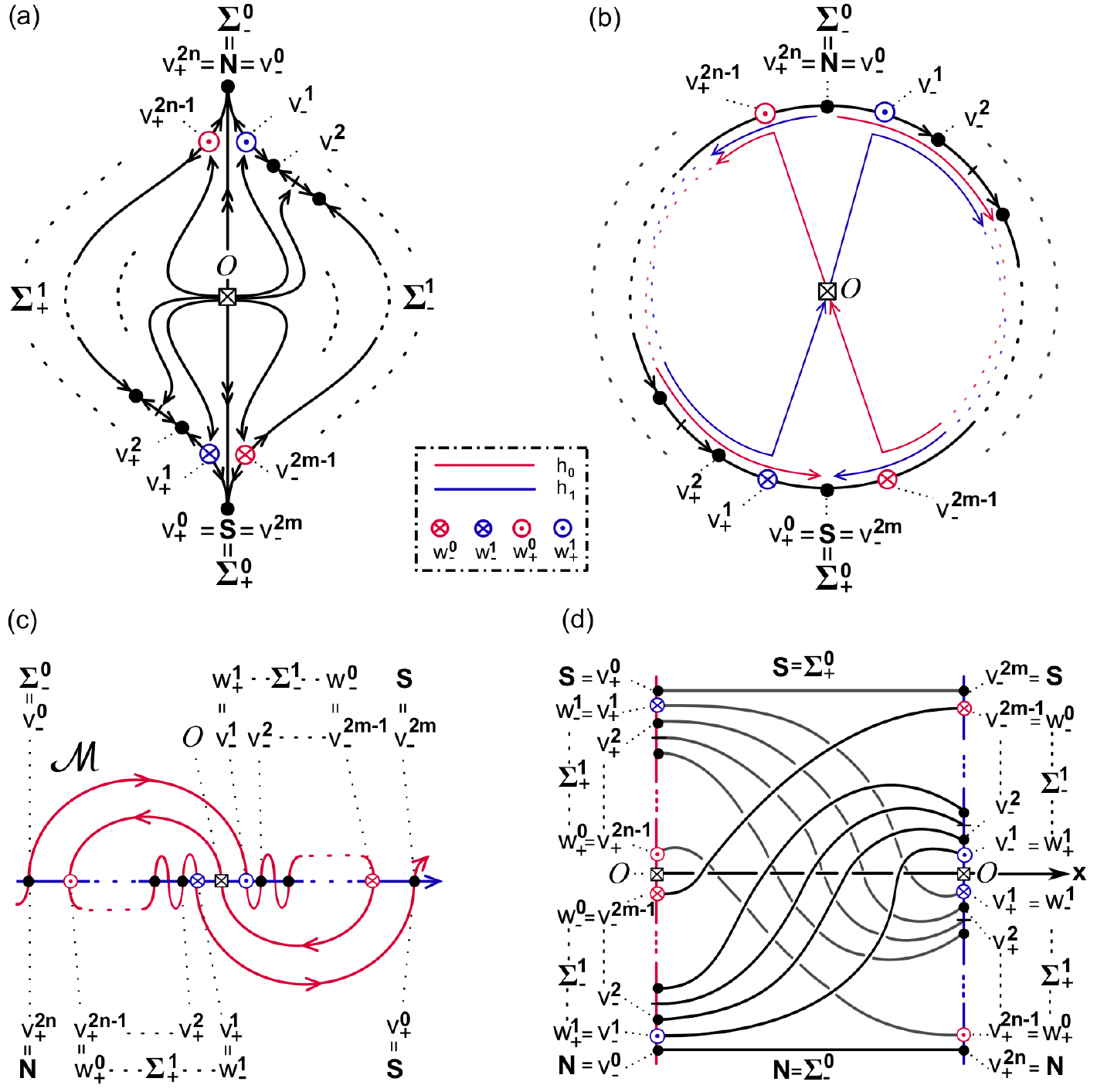}
\caption{\emph{
The Sturm disk with source $\mathcal{O}$, $m+n$ sinks, $m+n$ saddles, and hemisphere decomposition $\Sigma_\pm^j$, $j=0,1$, of $\mathcal{A} = \mathrm{clos\,} W^u(\mathcal{O})$.
(a) Dynamical view of the Sturm global attractor $\mathcal{A}$ with source equilibrium $\mathcal{O}$.
Saddles and sinks are enumerated by $v_\pm^k$ with odd and even exponents $k$, respectively.
Arrows indicate time evolution.
In addition to the connection graph $\mathcal{H}_f$, double arrows indicate the one-dimensional fast unstable manifold of $\mathcal{O}$ with boundary equilibria $\Sigma^0 = \lbrace \mathbf{N}, \mathbf{S}\rbrace$.
(b) The associated dynamic Thom-Smale complex $\mathcal{C}$.
Arrows on the circle boundary indicate the bipolar orientation of the edges of the 1-skeleton.
Edges are the whole one-dimensional unstable manifolds of the saddles; the orientation of the edge runs against the time direction on half of each edge.
The poles $\mathbf{N}$, $\mathbf{S}$ are the extrema of the bipolar orientation.
The bipolar orientation determines the ZS-pair $(h_0,h_1)$, by definition~\ref{def:2.2}.
Colors: $h_0$ (red), $h_1$ (blue).
(c) The meander $\mathcal{M}$ defined by the ZS-pair $(h_0, h_1)$ of (b).
Equilibria $v \in \mathcal{E}$ are ordered according to the oriented path $h_1$ (blue), increasing along the horizontal axis.
The oriented path $h_0$ (red) defines the arcs of the meander $\mathcal{M}$.
Note the two full polar $h_0$-serpents $v_-^0v_-^1 \ldots v_-^{2m-1}$ and $v_+^0v_+^1 \ldots v_+^{2n-1}$.
The two full polar $h_1$-serpents are $ v_+^{2n} \ldots v_+^1$ and $v_-^{1} \ldots v_-^{2m}$. 
(d) The equilibrium ``spaghetti'' $\mathcal{E}$.
The paths $h_0$ and $h_1$ are the orderings of $v \in \mathcal{E}$ by increasing boundary values $v(x)$ at the Neumann boundaries $x=0$ and $x=1$, respectively.
Note how the $h_\iota$-neighboring saddles to the source $\mathcal{O}$, at $x=\iota$, become the $h_{1-\iota}$-extreme saddles at the opposite boundary.
}}
\label{fig:2.2}
\end{figure}

\begin{defi}\label{def:2.2}
The bijections $h_0, h_1$ in \eqref{eq:2.1} are called a \emph{ZS-pair} $(h_0, h_1)$ in the finite, regular, planar and bipolar cell complex $\mathcal{C} = \bigcup_{v \in \mathcal{E}} c_v$ if the following three conditions all hold true:

\begin{itemize}
\item[(i)] $h_0$ traverses any face from $w_-^0$ to $w_+^0$;
\item[(ii)] $h_1$ traverses any face from $w_-^1$ to $w_+^1$
\item[(iii)] both $h_\iota$ follow the bipolar orientation of the 1-skeleton $\mathcal{C}^1$, if not already defined by (i), (ii).
\end{itemize}

We call $(h_0,h_1)$ an \emph{SZ-pair}, if $(h_1, h_0)$ is a ZS-pair, i.e. if the roles of $h_0$ and $h_1$ in the rules (i) and (ii) of the face traversals are reversed.
\end{defi}

The significance of ZS-pairs $(h_0,h_1)$ in the proof of theorem~\ref{thm:2.1} lies in their associated permutation
	\begin{equation}
	\sigma := h_0^{-1} \circ h_1 \in S_N\,;
	\label{eq:2.2}
	\end{equation}
see \eqref{eq:1.6}, \eqref{eq:1.19b}.
It turns out that $\sigma$ is a Sturm permutation, i.e. a dissipative Morse meander $\mathcal{M}$.
Let $\mathcal{A}$ be the associated Sturm global attractor, and $\mathcal{C}_{\text{ZS}}$ the associated dynamic Thom-Smale complex of $\mathcal{A}$.
Then
	\begin{equation}
	\mathcal{C}_{\text{ZS}} = \mathcal{C}
	\label{eq:2.3}
	\end{equation}
proves the if-part of theorem~\ref{thm:2.1}.
See \cite{firo14, firo13} for full details.
Equality in \eqref{eq:2.3} is understood in the sense of homeomorphic equivalence of regular cell complexes.
The cells in $\mathcal{C} = \cup_{v\in \mathcal{E}}c_v$ are indexed by an abstract finite set of $v \in \mathcal{E}$.
In the dynamic Thom-Smale complex $\mathcal{C}_{\text{ZS}} = \cup_{v\in \mathcal{E}}c_v^{\text{ZS}}$, the index $v$ is an intersection of the Sturm meander $\mathcal{M}$ with the horizontal axis, alias a Neumann equilibrium of the Sturm attractor realization \eqref{eq:1.1}.
For a more detailed discussion of our notion of equivalence, and a signed hemisphere refinement, we refer to our sequel \cite{firo3d-2}.

In fig.~\ref{fig:2.2} we illustrate theorem~\ref{thm:2.1} and definition~\ref{def:2.2} for the simple case of a single 2-disk, with $m+n$ sinks and $m+n$ saddles on the boundary, and with a single source $\mathcal{O}$.
The bipolar orientation of the 1-skeleton, in (b), in fact follows from the boundary $\Sigma^0 = \lbrace \mathbf{N}, \mathbf{S}\rbrace$ of the fast unstable manifold $W^{uu}(\mathcal{O})$, in (a).
Indeed $z(v-\mathcal{O}) = 0_\pm$ uniquely characterizes $v \in \Sigma_\pm^0$; see also proposition~\ref{prop:3.1}(iii).

Geometrically, however, there remain some general choices here.
The $u$-flip
	\begin{equation}
	u \mapsto -u
	\label{eq:2.4}
	\end{equation}
in the PDE \eqref{eq:1.9} induces a linear isomorphism $\mathcal{A} \rightarrow -\mathcal{A}$ of the Sturm attractors, reverses all bipolar orientations in the Sturm complex (b), rotates the Sturm meander $\mathcal{M}$ by $180^\circ$, and reverses the boundary orders of $h_0,h_1$ in (d) by
	\begin{equation}
	h_0 \mapsto h_0\kappa\,, \quad
	h_1 \mapsto h_1 \kappa\,.
	\label{eq:2.5}
	\end{equation}
Here the involution $\kappa	(j)$:= $ N+1-j$ in $S_N$ flips $\lbrace 1, \ldots , N\rbrace$.
This conjugates the Sturm permutation $\sigma = h_0^{-1} h_1$ by
	\begin{equation}
	\sigma \mapsto \kappa \sigma \kappa\,.
	\label{eq:2.6}
	\end{equation}
For the $(m,n)$ 	Sturm disk (a), after planar rotation by $180^\circ$, this amounts to the flip $(m,n) \mapsto (n,m)$.

Another ambiguity arises from the orientation of the planar embeddings $\mathcal{A}, \mathcal{C} \subseteq \mathbb{R}^2$.
Reversing orientation of $\mathcal{C}$, e.g. by reflection through the vertical $\mathbf{SN}$-axis, interchanges
	\begin{equation}
	h_0 \leftrightarrow h_1
	\label{eq:2.7}
	\end{equation}
to become an SZ-pair.
In terms of the PDE \eqref{eq:1.1} this is effected by the $x$-flip
	\begin{equation}
	x \mapsto 1-x\,.
	\label{eq:2.8}
	\end{equation}
The bipolar orientation remains unaffected, but the Sturm permutation $\sigma$ gets replaced by its inverse
	\begin{equation}
	h_0^{-1} \circ h_1=\sigma \quad
	\mapsto \quad \sigma^{-1} = h_1^{-1} \circ h_0\,.
	\label{eq:2.9}
	\end{equation}
Specifically, this flips the $(m,n)$ Sturm disk to $(n,m)$.
For a more general example, note how the transformation \eqref{eq:2.7}--\eqref{eq:2.9} relates the two recursion formulae \eqref{eq:1.22a} and \eqref{eq:1.26} for the Morse indices $i(v) = i_v$.

Together, the commuting involutions \eqref{eq:2.4}, \eqref{eq:2.8} on the attractor level of $\mathcal{A}_f$, $\mathcal{H}_f$, alias \eqref{eq:2.5}, \eqref{eq:2.7} on the level of bipolar complexes $\mathcal{C}$, alias \eqref{eq:2.6}, \eqref{eq:2.9} on the level of Sturm meanders $\mathcal{M}$, generate the Klein 4-group $\mathbb{Z}_2 \times \mathbb{Z}_2$.
The composition of the two involutions, for example, is an automorphism of the $(m,n)$ disk.
The following definition applies in the general setting of arbitrary Sturm global attractors.

\begin{defi}\label{def:2.3}
We call the Klein 4-group of involutions generated by \eqref{eq:2.4}--\eqref{eq:2.6} and \eqref{eq:2.7}--\eqref{eq:2.9} the \emph{trivial equivalences} of Sturm global attractors, their dynamic Thom-Smale complexes, and their Sturm meanders, respectively.
\end{defi}

Next we study planar Sturm global attractors and complexes which are \emph{topological disks}.
By this we mean that $\mathcal{A}, \mathcal{C}$ are allowed to contain several sources of Morse index $i=2$, but $\mathcal{A}, \mathcal{C}$ are homeomorphic to the standard closed disk.
We recall definition~\ref{def:1.1} of the signed hemisphere template $\mathcal{E}_\pm^j(v)$ of $\mathcal{A}$, according to equilibria in the hemisphere decomposition $\Sigma_\pm^j(v)$ of $\partial W^u(v)$, for all equilibria $v \in \mathcal{E}$ and $0 \leq j < i(v)$.

\begin{thm}\label{thm:2.4}

\begin{itemize}
\item [] \phantom{linebreak}
\item[(i)] Let $(h_0,h_1)$ be the ZS-pair of a given planar bipolar topological disk complex $\mathcal{C} \subseteq \mathbb{R}^2$ with poles $\mathbf{N}$, $\mathbf{S}$ on the circular boundary of $\mathcal{C}$.
Then the Sturm permutation $\sigma := h_0^{-1}h_1$ defines a unique topological disk Sturm global attractor $\mathcal{A}$ with dynamic Thom-Smale complex $\mathcal{C}$, and hence a unique signed hemisphere template $\mathcal{E}_\pm^j(v)$.
\item[(ii)] Conversely, let $\mathcal{E}_\pm^j(v)$ be the signed hemisphere template of a given planar Sturm global attractor $\mathcal{A}$.
Then $\mathcal{E}_\pm^j(v)$ define a unique bipolar orientation of the planar dynamic Thom-Smale complex $\mathcal{C}$ of $\mathcal{A}$, and hence a unique ZS-pair $(h_0,h_1)$.
\end{itemize}

\end{thm}

\begin{proof}[\textbf{Proof.}]
The proof is essentially contained in theorem~\ref{thm:2.1}.
Uniqueness of $\mathcal{A}$ is understood in the sense of $C^0$ orbit equivalence, enhanced by the sign information on the zero numbers $z(\Tilde{v}-v) = j_\pm$ of the equilibria $\Tilde{v} \in \mathcal{E}_\pm^j(v)$.
\end{proof}

As a variant to theorem~\ref{thm:2.4}~(ii), let us assume the sets $\mathcal{E}_\pm^j(v)$ are known, but the information on the precise sign labels $+$ versus $-$ got lost.
Then the proof of (ii) has to address the nonuniqueness of bipolar orientations for the 1-skeleton $\mathcal{C}^1$ of the dynamic Thom-Smale complex $\mathcal{C}$ of $\mathcal{A}$.
Consider a single 2-cell $c_{\mathcal{O}}$, first.
Since $\Sigma^0(\mathcal{O})$ of the fast unstable manifold $W^{uu}(\mathcal{O})$ indicates the two extrema on $\partial c_{\mathcal{O}} \subseteq \mathcal{C}^1$, under any admissible bipolar orientation, we are allowed to choose which extremum is maximal (and hence which is minimal) on $\partial c_{\mathcal{O}}$.
This determines the bipolar orientation on $\partial c_{\mathcal{O}}$, up to a simultaneous orientation reversal of all edges.
By 2-connectedness of $\mathcal{C}^1$, this determines the bipolar orientation everywhere.

However, there remains the orientation ambiguity of the precise planar embedding $c_{\mathcal{O}} \subseteq \mathbb{R}^2$ of our chosen 2-cell.
Reversing that single orientation, however, reverses the embedding orientation of the planar 2-cell complex $\mathcal{C} \subseteq \mathbb{R}^2$, globally.
Together, the two choices above are covered by the trivial equivalences of definition~\ref{def:2.3}.
Since $\Sigma^0(\mathcal{O})$ of 2-cells $c_{\mathcal{O}}$ are the bounding target equilibria of the fast unstable manifolds $W^0 (\mathcal{O})$, this proves the following corollary to theorem~\ref{thm:2.4}(ii).

\begin{cor}\label{cor:2.5}
Consider a planar Sturm global attractor $\mathcal{A}$ which is a topological disk.
Let the connection graph $\mathcal{H}_f$ of $\mathcal{A}$ be given.
Also assume the target equilibrium sets $\Sigma^0(\mathcal{O})$ of the fast unstable manifolds $W^{uu}(\mathcal{O})$ are known, for any $i=2$ source $\mathcal{O}$.
This information defines a bipolar orientation of the planar dynamic Thom-Smale complex $\mathcal{C}$ of $\mathcal{A}$.
The bipolar orientation, and its associated ZS-pair $(h_0,h_1)$ and meander $\mathcal{M}$, are unique, up to the trivial equivalences generated by \eqref{eq:2.4}--\eqref{eq:2.9}.
\end{cor}

\begin{figure}[t!]
\centering \includegraphics[width=\textwidth]{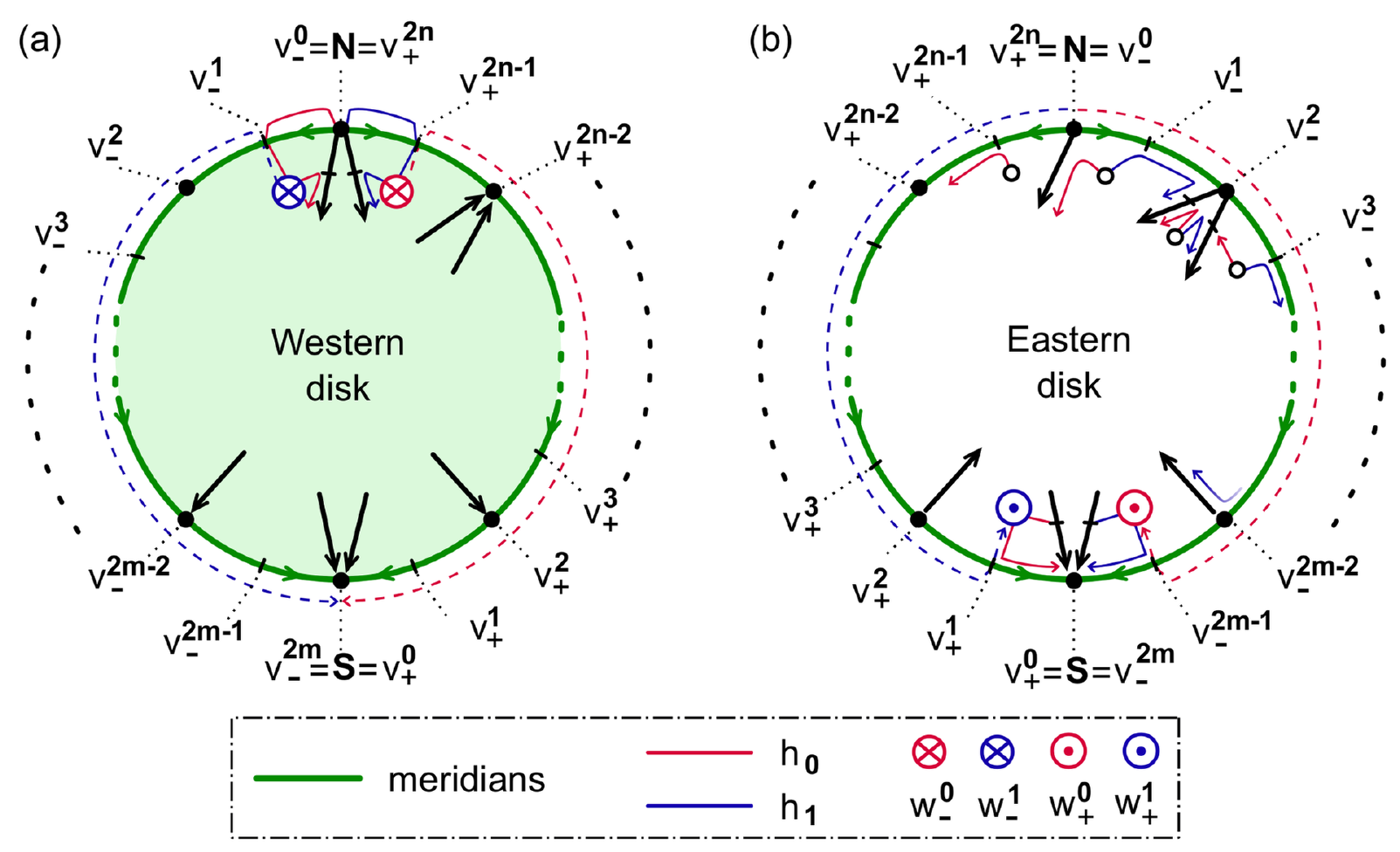}
\caption{\emph{
Western $(\mathbf{W})$ and Eastern $(\mathbf{E})$ planar topological disk complexes.
In $\mathbf{W}$, (a), all edges of the 1-skeleton $\mathbf{W}^1$ with a vertex $v \neq \mathbf{N}$ on the disk boundary are oriented towards $v$, i.e. towards the boundary of $\mathbf{W}$.
In $\mathbf{E}$, (b), all 1-skeleton edges with a vertex $v\neq \mathbf{S}$ on the disk boundary are oriented away from $v$,
i.e. towards the interior of $\mathbf{E}$.
Note the respective full $\mathbf{S}$-polar $h_0,h_1$-serpents $v_+^{2n-1} \ldots v_+^0 = \mathbf{S}$, $v_-^1 \ldots v_-^{2m} = \mathbf{S}$, dashed red/blue in (a), and the full $\mathbf{N}$-polar $h_0, h_1$-serpents $\mathbf{N} = v_-^0 \ldots v_-^{2m-1}$,
$\mathbf{N} = v_+^{2n} v_+^{2n-1} \ldots v_+^1$, dashed red/blue in (b).
Here we use ZS-pairs $(h_0,h_1)$ in $ \mathbf{E}$, but SZ-pairs $(h_0,h_1)$ in $\mathbf{W}$.
}}
\label{fig:2.3}
\end{figure}

In the 2-sphere boundary $\Sigma^2(\mathcal{O})$ of the Sturm 3-ball we will later weld closed Western and Eastern hemispheres $\mathrm{clos\,} \mathbf{W} = \mathrm{clos\,} \Sigma_-^2(\mathcal{O})$ and $\mathrm{clos\,} \mathbf{E} = \mathrm{clos\,}\Sigma_+^2(\mathcal{O})$ along their shared meridian $\Sigma^1(\mathcal{O}) = \mathrm{clos\,}\Sigma_-^2(\mathcal{O})\, \cap\, \mathrm{clos\,}\Sigma_+^2(\mathcal{O})$.
See definition~\ref{def:1.2}(ii),(iii).
For now, we define closed \emph{Western} and \emph{Eastern} planar topological disk complexes $\mathrm{clos\,}\mathbf{W}$ and $\mathrm{clos\,}\mathbf{E}$, accordingly, to fit that earlier definition.
However, we consider these planar disks separately, for now, each with its own associated pair $(h_0,h_1)$ of Hamiltonian paths.

\begin{defi}\label{def:2.6}
A bipolar topological disk complex $\mathrm{clos\,} \mathbf{E}$ with poles $\mathbf{N}, \mathbf{S}$ on the circular boundary $\partial \mathbf{E}$ is called \emph{Eastern disk}, if any edge of the 1-skeleton in $\mathbf{E}$, with at least one vertex $v \in \partial \mathbf{E} \setminus \mathbf{S}$, is oriented away from that boundary vertex $v$, 
i.e. towards the interior of $\mathbf{E}$.
Similarly, we call such a complex $\mathrm{clos\,} \mathbf{W}$ \emph{Western disk}, if any edge of the 1-skeleton in $\mathbf{W}$, with at least one vertex $v \in \partial \mathbf{W} \setminus \mathbf{N}$, is oriented towards that boundary vertex $v$,
i.e. towards the boundary of $\mathbf{W}$.
\end{defi}

\begin{lem}\label{lem:2.7}
Let  $\mathrm{clos}\, \mathbf{W}, \mathrm{clos\,}\mathbf{E}$ denote two arbitrary bipolar topological disk complexes with poles $\mathbf{N}, \mathbf{S}$ on their circular boundaries.
Let $(h_0,h_1)$ denote the associated SZ-pair for $\mathrm{clos\,}\mathbf{W}$, or the ZS-pair of $\mathrm{clos\,}\mathbf{E}$.
Let $\sigma$:= $h_0^{-1} h_ 1$ denote the associated permutation with Sturm meander $\mathcal{M}$.

Then the planar disk $\mathrm{clos\,}\mathbf{W}$ is Western, if and only if the $\mathbf{S}$-polar $h_\iota$-serpents are full, for $\iota = 0,1$, i.e. they contain all points of their respective boundary half-circle, except the antipodal pole $\mathbf{N}$.

Similarly, the planar disk $\mathrm{clos\,} \mathbf{E}$ is Eastern, if and only if the $\mathbf{N}$-polar $h_\iota$-serpents are full, for $\iota = 0,1$, i.e. they contain all points of their respective boundary half-circle, except the antipodal pole $\mathbf{S}$.
\end{lem}

\begin{proof}[\textbf{Proof.}]
Interchanging $h_0$ and $ h_1$ does not affect the claims.
(The use of SZ-pairs in $\mathrm{clos\,} \mathbf{W}$ is owed to our later use of $\mathbf{W}, \mathbf{E}$ as 3-ball hemispheres.)
Orientation reversal, by trivial equivalences as in definition~\ref{def:2.3}, interchanges $\mathbf{W}$ and $\mathbf{E}$ as well as $h_0$ and $h_1$.
It is therefore sufficient to consider $\mathbf{E}$ and $h_0$.
See also fig.~\ref{fig:2.3}.

Assume first that $\mathrm{clos\,} \mathbf{E}$ is Eastern, and inspect the right half-circle boundary of $\mathbf{E}$ from $\mathbf{N}$ to $\mathbf{S}$.
By \cite[lemma~3.3]{firo09}, the boundary is oriented from $\mathbf{N}$ to $\mathbf{S}$.
In $\mathbf{E}$, edges are oriented away from the boundary, towards the interior of $\mathbf{E}$.
Therefore that right boundary of $\mathbf{E}$ does not contain any $i=0$ sink vertex (other than $\mathbf{S}$) which qualifies as a boundary minimum of any face adjacent to that right boundary.
By definition~\ref{def:2.2} of the Z-path $h_0$, therefore, the path $h_0$ cannot leave the right boundary towards any adjacent face $c_{\mathcal{O}}$ with $i=2$ source $\mathcal{O}$, before reaching the right boundary neighbor $v_-^{2m-1}$ of the pole $\mathbf{S}$.
The serpent property of the initial part $h_0 = v_-^0 \ldots v_-^{2m-1} \ldots$ follows from Morse number formula \eqref{eq:1.22a}.
Indeed $h_1^{-1}$ increases monotonically, along the downward bipolar orientation of the right boundary of $\mathbf{E}$.
Hence $i=0$ sinks and $i=1$ saddles alternate along the right boundary.
Therefore the $\mathbf{N}$-polar $h_0$-serpent is full.

Conversely, assume the ZS-pair $(h_0,h_1)$ in $\mathbf{E}$ provides full $\mathbf{N}$-polar serpents in the bipolar disk complex $\mathbf{E}$.
To show $\mathbf{E}$ is Eastern, indirectly, suppose that the 1-skeleton of $\mathbf{E}$ possesses any edge oriented towards an $i=0$ sink $v \neq \mathbf{S}$ on the boundary of $\mathcal{E}$.
Then we also have a boundary adjacent face $c_{\mathcal{O}}$, with $i=2$ source $\mathcal{O}$, such that $v$ is the bipolar minimum on the 1-skeleton face boundary $\partial c_{\mathcal{O}}$.
Here we use \cite[lemma~2.1]{firo09} to establish boundary adjacency of $c_\mathcal{O}$, by a shared boundary edge.
Suppose $v$ is on the right boundary of $\mathbf{E}$, i.e. $v=v_-^{2k}$ for some $0 <k<m$.
The Z-path $h_0$ originates from $\mathbf{N}$ along the right boundary.
By definition~\ref{def:2.2} of a Z-path, $h_0$ must then leave the right boundary at the saddle $v_-^{2k-1}$ immediately preceding $v$ on that boundary.
Since $v \neq \mathbf{S}$, this contradicts our assumption of a full serpent $h_0$.

If $v$ is on the left half-circle boundary of $\mathbf{E}$, we argue via the S-path $h_1$, to reach an analogous contradiction.
This proves the lemma.
\end{proof}


\section{Zero numbers on hemispheres}\label{sec3}

In this short section we study the closure of the $n$-cell
	\begin{equation}
	\mathrm{clos\,} c_{\mathcal{O}} = \mathrm{clos\,} W^u(\mathcal{O})
	\label{eq:3.1}
	\end{equation}
of any hyperbolic equilibrium $\mathcal{O} \in \mathcal{E}$, $i(\mathcal{O}) =n$, in a Sturm global attractor $\mathcal{A}$.
As always, we assume hyperbolicity of all equilibria.
We investigate the Morse indices $i(v)$ and zero numbers $z(v_1-v_2)$ related to the hemisphere decomposition
	\begin{equation}
	\partial W^u(\mathcal{O}) =
	\bigcup\limits_{j=0}^{n-1} \Sigma_\pm^j
	\label{eq:3.2}
	\end{equation}
of the $(n-1)$-dimensional Schoenflies boundary sphere $\partial c_{\mathcal{O}} = \partial W^u(\mathcal{O}) = 
\Sigma^{n-1}(\mathcal{O})$.
See \eqref{eq:1.9a}--\eqref{eq:1.9e} and, for the special case $n=3$, also \eqref{eq:1.13}, \eqref{eq:1.17} .
See also the templates of figs.~\ref{fig:1.1}, \ref{fig:1.3}, and \ref{fig:3.1}.

\begin{figure}[t!]
\centering \includegraphics[width=\textwidth]{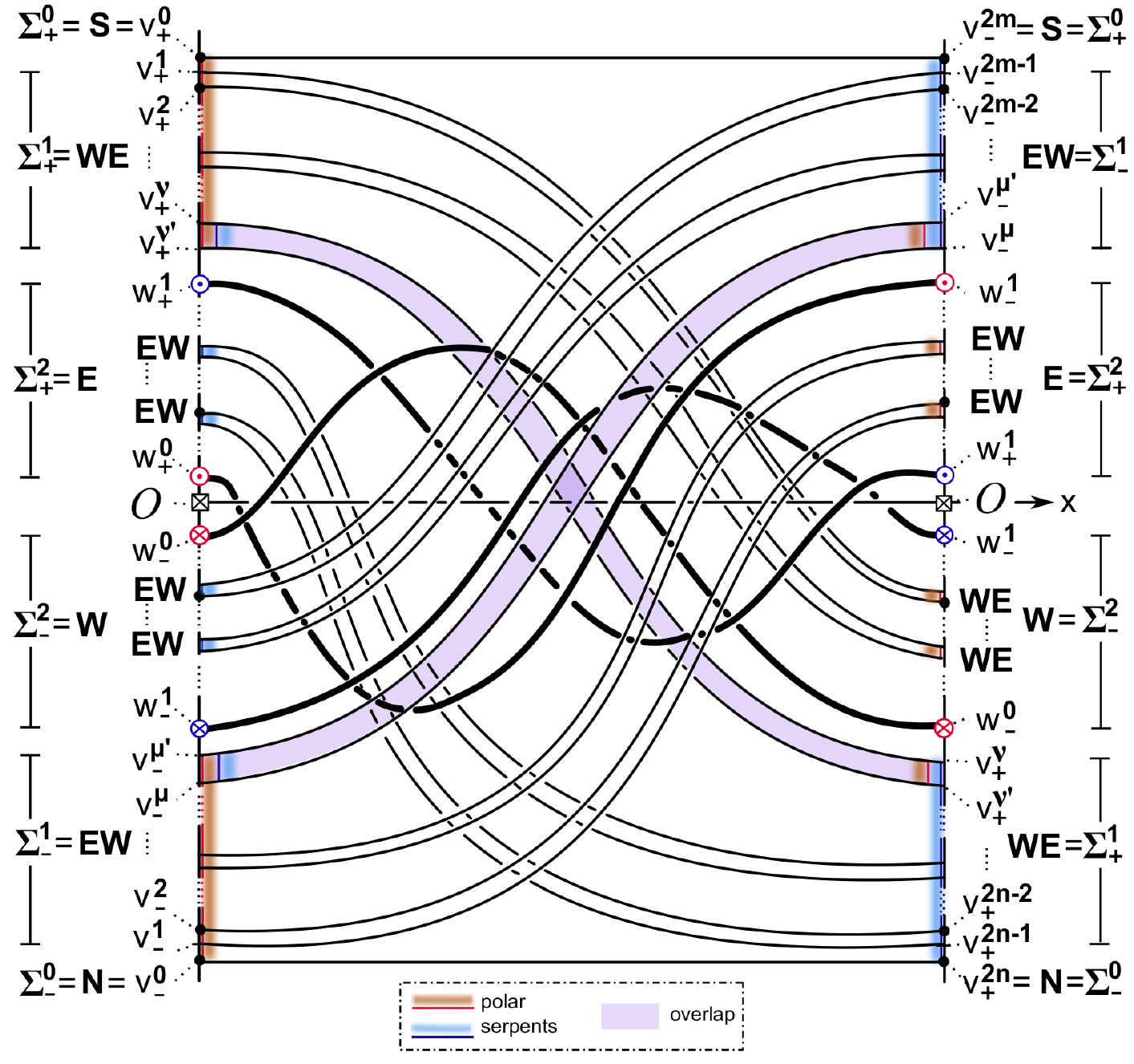}
\caption{\emph{
An impressionist sketch of the spatial profiles $v(x)$, for all equilibria $v\in \mathcal{E}_f$ of a general Sturm 3-ball global attractor.
The drawing illustrates the results of proposition~\ref{prop:3.1}, as well as certain aspects of lemma~\ref{lem:4.3} and our proof of theorem~\ref{thm:4.1}.
For the specific case of a solid octahedron see also fig.~\ref{fig:6.5}.
}}
\label{fig:3.1}
\end{figure}

\begin{prop}\label{prop:3.1}
Under the above assumptions the following statements hold true for all $0 \leq j<i(\mathcal{O})$, equilibria $v,v_1, v_2,$ and for $\delta=\pm$.

\begin{itemize}
\item[(i)] $v \in \Sigma_{\phantom{\pm}}^j \quad \Longrightarrow \quad i(v) \leq j\ $;
\item[(ii)] $v \in \Sigma_{\phantom{\pm}}^j \quad \Longrightarrow \quad z(v-\mathcal{O}) \leq j\ $;
\item[(iii)] $v \in \Sigma_\pm^j \quad \Longrightarrow \quad z(v-\mathcal{O}) =j_\pm\ $;
\item[(iv)] $v_1, v_2 \in \mathrm{clos\,} \Sigma_\delta^j
\quad \Longrightarrow \quad z(v_1-v_2) \leq j-1\ $.
\end{itemize} 
\end{prop}

\begin{proof}[\textbf{Proof.}]
To prove claim (i) we only have to observe that $v \in \Sigma^j$ implies $W^u(v) \subseteq \Sigma^j$, and take dimensions on both sides of this inclusion.
The inclusion follows from the $\lambda$-Lemma and transversality $W^{j+1} \transv W^s(v)$, as provided by the Morse-Smale property.

Claim (ii) follows from claim (iii). More directly, $z(u-\mathcal{O}) < \dim W^{j+1} = j+1$ for $u\in W^{j+1}$, by \cite{brfi86}.
This extends  to $u=v \in \partial W^{j+1} = \Sigma ^j$, because all zeros of $v-\mathcal{O}$ are simple; see ODE \eqref{eq:1.3}.

Claim (iv) follows from the same statement in $W^j$, and in the near parallel protocaps in $W^{j+1}$, which $\omega$-limit to $\mathrm{clos\,}\Sigma _\pm^j$.
See also \cite{firo13} for further details on the above claims.

Claim (iii), which precisely characterizes equilibria in hemispheres, follows, by induction over $j$ from Wolfrums's characterization of heteroclinicity $v_- \leadsto v_+$.
By \cite{wo02}, $v_- \leadsto v_+$ holds, if and only if there exists a heteroclinic orbit $u(t, \cdot) \rightarrow v_\pm$, for $t \rightarrow \pm \infty$, such that
	\begin{equation}
	z(u(t, \cdot ) -v_\pm) = z(v_+-v_-)
	\label{eq:3.7}
	\end{equation}
holds for all real $t$, in the unsigned version \eqref{eq:1.4} of the zero number $z$.
See lemma~\ref{lem:4.2} below for a more detailed statement, and \cite[appendix]{firo3d-2} for a more detailed review of  \cite{wo02}.
In the signed version \eqref{eq:1.4+} of the zero number, the same statement reads
	\begin{equation}
	z(v_+-u(t, \cdot )) = z(v_+-v_-) = z(u(t, \cdot )-v_-)\,.
	\label{eq:3.8}
	\end{equation}
For $v_-$:= $\mathcal{O}$ and $v_+$:= $v \in \Sigma_+^j \subseteq \partial W^{j+1}$, say, the right hand equality implies	
	\begin{equation}
	z(v-\mathcal{O})= z(u(t, \cdot ) -\mathcal{O}) \leq 
	\dim W^{j+1}-1 =j\,,
	\label{eq:3.9}
	\end{equation}
again by \cite{brfi86}.
We claim equality.
Suppose, indirectly, that the $t$-independent value $z(u(t, \cdot)-\mathcal{O})$ satisfies	
	\begin{equation}
	j':= z(v-\mathcal{O}) = z(u(t, \cdot )-\mathcal{O}) \leq j-1
	\label{eq:3.10}
	\end{equation}
for all $t$.
Then $u(t, \cdot ) -\mathcal{O}$ is backwards tangent to the Sturm-Liouville eigenfunction $\pm \varphi_{j'}$ of the linearization at $\mathcal{O}$, for $t \rightarrow -\infty$.
In particular $u(t, \cdot ) \in W^{j'+1}$ and 	
	\begin{equation}
	v= \lim\limits_{t \rightarrow +\infty} u(t, \cdot) 
	\in \partial W^{j'+1} = \Sigma^{j'} \subseteq \Sigma^{j-1}\,.
	\label{eq:3.11}
	\end{equation}
But this is excluded for $v \in \Sigma_\pm^j$, by the disjoint union \eqref{eq:3.2}.
This contradiction shows $j' = z(v-\mathcal{O})=j$.
The signed claim $z(v-\mathcal{O}) = j_\pm$ for $v \in \Sigma_\pm^j$ follows easily, because $t \mapsto z(u(t, \cdot )-\mathcal{O})$ cannot drop and hence has to preserve sign.
This proves claim (iii), and the proposition.
\end{proof}


\section{From signed 2-hemisphere templates to 3-cell templates}\label{sec4}

After the preparations on planar Sturm global attractors and on zero numbers in single closed cells, we can now embark on the first arrow \eqref{eq:1.12} in the cyclic template list
	\begin{equation}
	\begin{aligned}
	\text{signed 2-hemisphere template} \quad 
	&\Longrightarrow \quad \text{3-cell template}\,;\\
	\text{3-cell template} \quad 
	&\Longrightarrow \quad \text{3-meander template}\,;\\
	\text{3-meander template} \quad
	&\Longrightarrow \quad \text{signed 2-hemisphere template}\,;
	\end{aligned}
	\label{eq:4.1}
	\end{equation}
see definitions \ref{def:1.1}, \ref{def:1.2}, \ref{def:1.3}, and \eqref{eq:1.18}, \eqref{eq:1.27}.
Each arrow consists of a construction, which defines the map of the arrow, and a theorem, which establishes the defining properties of the target.
See definitions~\ref{def:1.1}--\ref{def:1.3}.

The map, for the first arrow, was specified in the translation table \eqref{eq:1.17}, as far as the hemisphere correspondence between the signed 2-hemisphere decomposition $\Sigma_\pm^j(\mathcal{O})$, $j=0,1,2$, of $\partial W^u(\mathcal{O})$ with the boundary decomposition
	\begin{equation}
	\partial c_{\mathcal{O}} =  \mathbf{N} \,\dot{\cup} \,
	\mathbf{S} \,\dot{\cup}\, \mathbf{EW}\, \dot{\cup}\, \mathbf{WE}\,
	 \dot{\cup}\,	\mathbf{W} \,\dot{\cup}\, \mathbf{E}
	\label{eq:4.2}
	\end{equation}
of the 3-cell $c_{\mathcal{O}} \in \mathcal{C}$ is concerned.
Here and below we omit braces of singleton sets. For example we write $\mathbf{N}$ for $\lbrace \mathbf{N} \rbrace$. We recall the orientation of 1-skeleton edges $c_v \in \mathcal{C}^1$, alias unstable manifolds $W^u(v)$ of $i=1$ saddles $v$, from equilibrium $\Sigma_-^0 (v)$ to $\Sigma_+^0(v)$.
The following theorem asserts that this passage from the dynamic Thom-Smale complex $\mathcal{A} = \cup_{v \in \mathcal{E}} W^u(v)$ of a 3-ball Sturm attractor $\mathcal{A}$ to the finite regular cell complex $\mathcal{C}$ with cells $c_v$:= $W^u(v)$ indeed satisfies the properties of a 3-cell template.

\begin{thm}\label{thm:4.1}
The dynamic Thom-Smale complex $\mathcal{C}$ of any 3-ball Sturm attractor $\mathcal{A}$ is a 3-cell template.
In particular, 

\begin{itemize}
\item[(i)] the above edge orientation of the 1-skeleton is bipolar from the pole $ \mathbf{N} = \Sigma_-^0 (\mathcal{O})$ to the pole $\mathbf{S} = \Sigma_+^0 (\mathcal{O})$;
\item[(ii)] the disjoint meridian paths $\mathbf{EW} = \Sigma_-^1(\mathcal{O})$ and $\mathbf{WE} = \Sigma _+^1(\mathcal{O})$ are directed from pole $\mathbf{N}$ to pole $\mathbf{S}$;
\item[(iii)] edges are oriented towards the meridians, in $\mathbf{W}$, and away from the meridians, in $\mathbf{E}$, with the necessary exceptions at the poles $\mathbf{N}$, $\mathbf{S}$;
\item[(iv)] $\mathbf{W}$-faces, adjacent to $\mathbf{N}$ and a first meridian edge, possess an edge overlap with the $\mathbf{E}$-face, adjacent to $\mathbf{S}$ and the last edge on that same meridian.
\end{itemize}
\end{thm}

\begin{proof}[\textbf{Proof of Theorem~\ref{thm:4.1}(i)--(iii)}]
To prove (i), we first note that any directed path from equilibrium vertex $v_1$ to $v_2 \neq v_1$ in the 1-skeleton $\mathcal{C}^1$ of $\mathcal{A}$ implies
	\begin{equation}
	v_1(x) < v_2(x)\,,
	\label{eq:4.3}
	\end{equation}
for all $0 \leq x \leq 1$.
Therefore the orientation of the 1-skeleton is acyclic.
In particular there exists at least one local ``source'' $\mathbf{N}'$ of the orientation, and at least one local orientation ``sink'' $\mathbf{S}'$.
It is sufficient to show $\mathbf{N}' = \mathbf{N}$, and hence uniqueness of $\mathbf{N}'$; the arguments for $\mathbf{S}' = \mathbf{S}$ are analogous.

Suppose, indirectly, $\mathbf{N}' \neq \mathbf{N}$.
Then $z(\mathbf{N} -\mathcal{O})= 0_-$ blocks the heteroclinic orbit $\mathcal{O} \leadsto \mathbf{N}'$, unless $z(\mathbf{N}-\mathbf{N}') = 0_-$.
In particular $\mathbf{N}$ and $\mathbf{N}'$ are $i=0$ sink equilibria such that $\mathbf{N} < \mathbf{N}'$.
For monotone dynamical systems with hyperbolic equilibria it has been proved that $\mathbf{N}$ and $\mathbf{N}'$ are then separated by one or several $i=1$ saddle equilibria $v$ which are strictly ordered, by $z=0$, and are strictly between $\mathbf{N}$ and $\mathbf{N}'$.
See \cite{ma79} and \cite{po16}.
Let $v$ denote the largest of these saddles.
Then $v \leadsto \mathbf{N}'$ and $v < \mathbf{N}'$ implies $\mathbf{N}' = \Sigma_+^0(v)$, and the edge $W^u(v)$ is oriented towards the orientation ``source'' $\mathbf{N}'$.
This contradiction proves $\mathbf{N}' = \mathbf{N}$, and claim (i) is settled.

To prove claim (ii), we first observe that the meridians $\Sigma_\pm^1 (\mathcal{O})$ are disjoint, by definition.  
It is therefore sufficient to consider the meridian $\mathbf{WE} = \Sigma_+^1(\mathcal{O})$, without loss of generality.
Note that $z(v_1-v_2) =0$ for any two distinct equilibria $v_1,v_2 \in \mathrm{clos\,} \Sigma_+^1(\mathcal{O}) = \mathbf{WE}\, \cup \, \lbrace \mathbf{N}, \mathbf{S} \rbrace$;
see proposition~\ref{prop:3.1}(iv).
Therefore the equilibria on $\mathrm{clos\,} \mathbf{WE}$ are totally ordered, including the saddles and their unstable manifolds, by the bipolar orientation of the 1-skeleton.

To prove claim (iii), consider $\mathbf{W} = \Sigma_-^2(\mathcal{O})$ without loss of generality.
Else consider $-u(t,x)$ to reverse orientations, and interchange poles $\Sigma_\pm^0(\mathcal{O})$ and hemispheres $\Sigma_\pm^2(\mathcal{O})$.
Let $v \in \mathbf{W} = \Sigma_-^2(\mathcal{O})$ be any $i=1$ saddle with heteroclinic orbit
	\begin{equation}
	v \leadsto \Tilde{v} \in \Sigma_\pm^1(\mathcal{O})
	\label{eq:4.4}
	\end{equation}
to an $i=0$ sink $\Tilde{v} \neq \mathbf{N}, \mathbf{S}$ in a meridian.
Without loss of generality, after passing to $u(t, 1-x)$ if necessary, assume $\Tilde{v} \in \Sigma_+^1(\mathcal{O})= 
\mathbf{WE}$.
Then $v\in \Sigma_-^2(\mathcal{O})$, $\Tilde{v} \in \Sigma_+^1 (\mathcal{O})$ imply
	\begin{equation}
	v(0) < \mathcal{O}(0) < \Tilde{v}(0)\,.
	\label{eq:4.5}
	\end{equation}
But $z(\Tilde{v}-v) = 0_\pm$, by $\Tilde{v} \in \Sigma_\pm^0(v)$ in \eqref{eq:4.4}.
Therefore \eqref{eq:4.5} implies $z(\Tilde{v}-v)= 0_+$, i.e., the edge $c_v$ of the saddle $v$ in $\mathbf{W}$ is oriented towards the meridian at $\Tilde{v} \in \Sigma_+^1(\mathcal{O})$.
This proves claim (iii).
\end{proof}

Our proof of the overlap property of theorem~\ref{thm:4.1}(iv) is based on
lemma~\ref{lem:4.3} below.
We precisely identify the face barycenters of (iv) to coincide with the neighbors $w_\pm^\iota $ of $\mathcal{O}$ in the orders $h_\iota$ of the equilibrium set $\mathcal{E}$ at $x=\iota$; see \eqref{eq:1.19a} and \eqref{eq:4.24}.
We partially rely on the Wolfrum lemma~\ref{lem:4.2} which characterizes heteroclinicity $v_- \leadsto v_+$ in a more elegant way than \cite{firo96}; see also \cite[appendix]{firo3d-2}.

\begin{figure}[t]
\centering \includegraphics[width=0.6\textwidth]{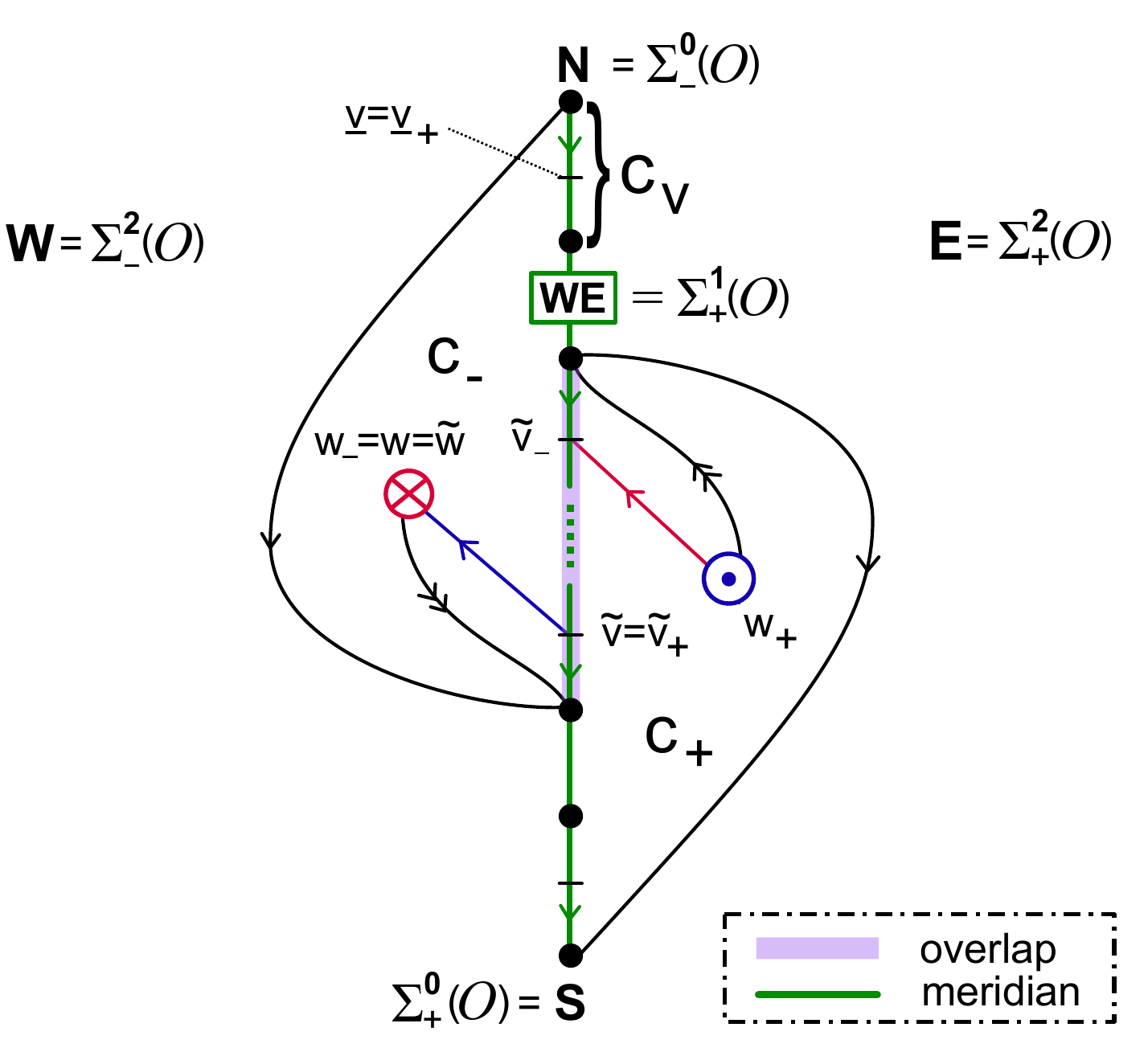}
\caption{\emph{
The overlap construction.
In the notation of \eqref{eq:4.10}, the left face $c_-=c_w \subseteq \Sigma_-^2 (\mathcal{O}) = \mathbf{W}$ is assumed to contain the first edge $c_{\underline{v}}$ on the $\mathbf{WE}$ meridian $\Sigma_+^1(\mathcal{O})$ from $\mathbf{N}$ to $\mathbf{S}$, in its boundary $\partial c_-$.
Next abbreviate the sources $w_-^0=w$ by $w_-$ and $w_+^1$ by $w_+$, with 2-cell faces $c_\pm$.
Then the boundary saddles $\Tilde{v}_\pm$ satisfy $\Tilde{v}_\pm \in \partial c_\mp\,\cap \, \Sigma_+^1(\mathcal{O})$ and $\Tilde{v}_+ \geq \Tilde{v}_-$.
In other words, $\Tilde{v}_-$ (nonstrictly) precedes $\Tilde{v}_+$ in the directed meridian $\Sigma_+^1 (\mathcal{O})$ from $\mathbf{N}$ to $\mathbf{S}$.
Because the meridian paths from $\mathbf{N}$ to the edge of $\Tilde{v}_+$, and from the edge of $\Tilde{v}_-$ to $\mathbf{S}$, are contained in the respective boundaries $\partial c_{\pm}$, entirely, this shows boundary overlap of the cells $c_\pm$ along at least one edge of the $\mathbf{WE}$ meridian.
}}
\label{fig:4.1}
\end{figure}

\begin{lem}[\textbf{[Wo02]}]\label{lem:4.2}
Let $\mathcal{A}$ be a general Sturm global attractor with all equilibria being hyperbolic.
Let $v_\pm \in \mathcal{E}$.
Then $v_- \leadsto v_+$ if, and only if, $i(v_-) > i(v_+)$ and $v_\pm$ are $z(v_+ -v_-)$-adjacent.
\end{lem}

Here the equilibria $v_\pm$ are called $k$-\emph{adjacent}, if there does not exist a third equilibrium $v$ strictly between $v_+$ and $v_-$, at $x=0$ or equivalently at $x=1$, such that
	\begin{equation}
	z(v_\pm -v) =k\,.
	\label{eq:4.9a}
	\end{equation}
In other words, the signed zero numbers of $v_\pm -v$ are $k$, of opposite sign index.
An equilibrium $v$, as above, which prevents $k$-adjacency of $v_\pm$ \emph{blocks} $v_- \leadsto v_+$.
For blocking, $k=z(v_+-v_-)$ is not required.
The existence of an equilibrium $v$ such that
	\begin{equation}
	z(v_+-v) > z(v_- -v)\,,
	\label{eq:4.9b}
	\end{equation}
also blocks $v_- \leadsto v_+$, of course, simply because $t \mapsto z(u(t, \cdot )-v)$ is nonincreasing, due to zero number dropping \eqref{eq:1.4}.

For example, blocking is impossible if $v_\pm$ are $h_0$- or $h_1$-neighboring equilibria, i.e. are neighbors at $x=0$ or $x=1$.
By adjacency of their Morse indices \eqref{eq:1.22a}, \eqref{eq:1.26}, \eqref{eq:1.22b}, this implies the existence of a heteroclinic orbit between $v_+$ and $v_-$, running from the higher to the lower Morse index.
In the particular case $i(v_-) = i(v_+)+1$ of adjacent Morse indices, lemma~\ref{lem:4.2} has already been obtained in \cite[lemma~1.7]{firo96}.

To prepare our proof of theorem~\ref{thm:4.1}(iv) we introduce the following eight notations for specific equilibria. We will show some of them in fact coincide.
Let $\mathcal{O}$ denote the equilibrium of Morse index $i=3$.
Our notation is based on the equilibrium orders $h_\iota$ of $\mathcal{E}$ at $x=\iota$; see~\eqref{eq:1.19a} and \eqref{eq:1.24a}:
	\begin{equation}
	\begin{aligned}
	 w  :\quad  & \text{the } 
	h_0\text{-last equilibrium } h_0\text{-before }	\mathcal{O}\,; \\
	 \Tilde{w} :\quad & \text{the } 
	h_1\text{-first source } h_1\text{-before }	\mathcal{O}\,;\\
	 \Tilde{v}  :\quad & \text{the } 
	h_1\text{-last equilibrium } h_1\text{-before }	\Tilde{w}\,;\\
	 \Tilde{v}_+  :\quad & \text{the } 
	h_1\text{-last saddle in }\mathcal{E}_+^1(\mathcal{O}),\  
	h_1\text{-before }	\Tilde{w}\,;\\
	 v_-  : \quad &  
	\text{the North pole } \Sigma_-^0(w) \text{ of the face of } w\,;\\
	 \underline{v}  : \quad & \text{the } 
	h_1\text{-first equilibrium } h_1\text{-after }	\mathbf{N}\,;\\
	 \underline{v}_+  :\quad & \text{the } 
	h_1\text{-first saddle in }\mathcal{E}_+^1(\mathcal{O}),\ 
	h_1\text{-after } \mathbf{N}\,.
	\end{aligned}
	\label{eq:4.10}
	\end{equation}

We recall the notation $\mathcal{E}_\pm^j(\mathcal{O}) = \Sigma_\pm^j(\mathcal{O}) \cap \mathcal{E}$ for hemisphere equilibria; see~\eqref{eq:1.9g}.
Existence of all equilibria, except $\Tilde{v}_+$, follows from adjacency~\eqref{eq:1.22a}, \eqref{eq:1.26} of Morse numbers, alias Morse indices~\eqref{eq:1.22b}, for $h_\iota$-adjacent equilibria.
Equilibrium $w$ is an $i=2$ source, automatically, by $i(\mathcal{O}) =3$.
Note how $w= w_-^0$, from \eqref{eq:1.24a}, has just been stripped of its decorative sub- and superscript.
Similarly $\underline{v}$ and $\Tilde{v}$ are $i=1$ saddles.
In particular lemma~\ref{lem:4.3}(i), below, which proves $\underline{v} = \underline{v}_+$, shows that $\underline{v} \in \mathcal{E}_+^1(\mathcal{O})$ occurs $h_1$-before $\Tilde{w}$ and hence implies the existence of $\Tilde{v}_+$.

By $h_\iota$-adjacency, the Wolfrum lemma~\ref{lem:4.2} immediately implies the following heteroclinic orbits in~\eqref{eq:4.10}:
	\begin{equation}
	\begin{aligned}
	\mathcal{O} \quad &\leadsto \quad w\,;\\
	\Tilde{w} \quad &\leadsto \quad \Tilde{v}\,;\\
	\underline{v} \quad &\leadsto \quad \mathbf{N}\,.
	\end{aligned}
	\label{eq:4.11}
	\end{equation}
	In particular $w \in \mathcal{E}_-^2(\mathcal{O})$.
By definition of $\Sigma_-^0(w)$ 
we also have a monotonically decreasing heteroclinic orbit
	\begin{equation}
	w \quad \leadsto\quad  v_- \quad < \quad w
	\label{eq:4.12}
	\end{equation}
along the fast unstable manifold of $w$.
See fig.~\ref{fig:4.1} for a partial notational illustration of the case $w=w_-^0$ and the closely related antipodal case $w=w_+^1$, which define candidates for a boundary overlap.
Also recall fig.~\ref{fig:3.1} for an illustration of a ``spaghetti template'', and the specific case of a solid octahedron, fig.~\ref{fig:6.5}.
Although these figures much inspire and illustrate the proofs, below, they will not be used in any technical sense.

\begin{lem}\label{lem:4.3}
In the above setting and notation~\eqref{eq:4.10} the equilibrium $\Tilde{v}_+$ exists.
Moreover
\begin{itemize}
\item[(i)] $\quad\underline{v} \phantom{(w_-)} \quad= \quad \underline{v}_+ \in \mathcal{E}_+^1 (\mathcal{O})$;
\item[(ii)] $\quad\Tilde{v}  \phantom{(w_-)} \quad= \quad \Tilde{v}_+ \in \mathcal{E}_+^1 (\mathcal{O})$;
\item[(iii)] $\quad w  \phantom{(v_-)} \quad = \quad \Tilde{w} \in \mathcal{E}_-^2(\mathcal{O})$;
\item[(iv)] $\quad v_- \phantom{(w)} \quad =  \quad \mathbf{N}$;
\item[(v)] $\quad \mathcal{E}_+^1 (w) \quad = \quad \lbrace v \in \mathcal{E}\; |\; \mathbf{N} (1) < v (1) < w(1) \rbrace \quad \subseteq \quad \mathcal{E}_+^1 (\mathcal{O})$.
\end{itemize}
In particular the face $c_w$ of the $h_0$-last equilibrium $w$ before $\mathcal{O}$ is the unique 2-cell in $\Sigma_-^2(\mathcal{O})$ which is adjacent to the first edge of the meridian $\Sigma_+^1(\mathcal{O})$ at $\mathbf{N}$.
\end{lem}

\begin{proof}[\textbf{Proof.}]
Existence of $\Tilde{v}_+$ follows from claim~(i).
Claims~(i), (iv) imply $\underline{v} = \underline{v}_+ \in \mathcal{E}_+^1(w)$ and hence the first edge $c_{\underline{v}} = c_{\underline{v}_+}$ of the meridian $\mathcal{E}_+^1 (\mathcal{O})$, with one end point at $\mathbf{N}$, is contained in the boundary $\partial {c_w}$ of the face $c_w$ of $w$.
It only remains to show claims~(i)--(v), successively.
Throughout we normalize $\mathcal{O} \equiv 0$.

To show claim~(i), indirectly, suppose $\underline{v}_+ \neq \underline{v}$.
Then definition~\eqref{eq:4.10} of $\underline{v}$ implies $\underline{v} \not\in \mathcal{E}_+^1 (\mathcal{O}) = \mathcal{E} \cap \Sigma_+^1(\mathcal{O})$, and hence $\underline{v} (0) <0$, $\underline{v}(1)<0$.
Since $\underline{v} \neq \mathbf{N}= \mathcal{E}_-^0(\mathcal{O})$ this implies $\underline{v} \in \mathcal{E}_-^2(\mathcal{O})$.
Because $\underline{v}$ was already identified as a saddle, definition \eqref{eq:4.10} of $\underline{v}_+$ and $\mathcal{O} \equiv 0$ further imply
	\begin{equation}
	\underline{v}(0) < 0< \underline{v}_+ (0) \quad \text{and} \quad
	\underline{v} (1) < \underline{v}_+(0)\,.
	\label{eq:4.13}
	\end{equation}
Since $\underline{v}, \underline{v}_+ \in \mathrm{clos\,} \Sigma_-^2(\mathcal{O})$, proposition~\ref{prop:3.1}(iv) therefore implies $z(\underline{v}_+ - \underline{v}) \leq 1$ and, with~\eqref{eq:4.13},
$z(\underline{v}_+ -\underline{v}) = 0_+$.
In other words
	\begin{equation}
	\underline{v} <  \underline{v}_+
	\label{eq:4.14}
	\end{equation}
is strictly monotonically ordered.
On the other hand, dissipativeness \eqref{eq:1.20} of the Sturm meander $\mathcal{M}$ and of the Sturm PDE \eqref{eq:1.1} imply
	\begin{equation}
	\mathbf{N}  <  v  < \mathbf{S}
	\label{eq:4.15}
	\end{equation}
for all non-pole equilibria $v \in \mathcal{E}$.
In particular
	\begin{equation}
	\underline{v}  >  \mathbf{N}\,.
	\label{eq:4.16}
	\end{equation}
By \eqref{eq:4.14}, \eqref{eq:4.16} the equilibrium $\underline{v}$ blocks $\underline{v}_+ \leadsto \mathbf{N}$, in the sense of \eqref{eq:4.9a}.

On the other hand, all equilibria on the meridian $\mathrm{clos\,} \Sigma_+^1(\mathcal{O})$ are ordered strictly monotonically by $z=0$; see proposition~\ref{prop:3.1}(iv).
With the saddle $\underline{v}_+$, the meridian $\Sigma_+^1(\mathcal{O})$ also contains the edge $W^u(\underline{v}_+)$ which defines the first edge of that meridian, adjacent to $\mathbf{N}$.
Therefore $\underline{v}_+ \leadsto \mathbf{N}$.
This contradicts the above blocking of $\underline{v}_+ \leadsto \mathbf{N}$ by $ \underline{v}$, and proves claim (i), as well as existence of $\Tilde{v}_+$.

We prove claim~(ii) next, indirectly.
Suppose $\Tilde{v} \neq \Tilde{v}_+$.
In \eqref{eq:4.11} we have already observed $\Tilde{w} \leadsto \Tilde{v}$ connects heteroclinically to the saddle $\Tilde{v} \not\in \mathcal{E}_+^1(\mathcal{O})$.
In particular $z(\Tilde{v}-\mathcal{O}) \in \lbrace 0,2 \rbrace$.
Furthermore $\Tilde{v} \neq \mathbf{N}= \mathcal{E}_-^0(\mathcal{O})$, because $\Tilde{v}$ is a saddle and $\mathbf{N}$ is a sink.
Therefore $z(\Tilde{v}-\mathcal{O})=2$, and $\Tilde{v}(1) < \Tilde{w}(1)<0$ implies $\Tilde{v} \in \mathcal{E}_-^2(\mathcal{O})$.

Let $\Tilde{w}_+$ denote the source in $\mathcal{E}_-^2(\mathcal{O})$ such that $\Tilde{w}_+ \leadsto \Tilde{v}_+ \in \mathcal{E}_+^1 (\mathcal{O})$.
In other words, $\Tilde{w}_+$ is the source of the face $c$ in $\Sigma_-^2(\mathcal{O})$ adjacent to the meridian edge $c_{\Tilde{v}_+} \subseteq \partial c \cap \Sigma_+^1(\mathcal{O})$.
Then proposition~\ref{prop:3.1}(iv) for $\Tilde{w}_+ , \Tilde{v} \in \mathcal{E}_-^2(\mathcal{O})$, and the definition \eqref{eq:4.10} of $\Tilde{v}, \Tilde{w}$ imply $\Tilde{w}_+(1) \geq \Tilde{w}(1) > \Tilde{v}(1)$ and $z(\Tilde{w}_+-\Tilde{v}) \leq 1$.
Hence
	\begin{equation}
	z(\Tilde{w}_+ - \Tilde{v}) \in \lbrace 0,1_-\rbrace\,.
	\label{eq:4.17}
	\end{equation}
On the other hand $\Tilde{v} \in \mathcal{E}_-^2 (\mathcal{O})$, $\Tilde{v}_+ \in \mathcal{E}_+^1(\mathcal{O})$ and $\Tilde{v}(1)> \Tilde{v}_+(1)$, by definition~\eqref{eq:4.10}.
In particular $\Tilde{v}_+(0) > 0 > \Tilde{v}(0)$, and proposition~\ref{prop:3.1}(iv) implies
	\begin{equation}
	z(\Tilde{v}_+-\Tilde{v}) = 1_+\,.
	\label{eq:4.18}
	\end{equation} 
Therefore $\Tilde{v}$ blocks $\Tilde{w}_+ \leadsto \Tilde{v}_+$, by \eqref{eq:4.17} and \eqref{eq:4.18}, in the sense of \eqref{eq:4.9b}.
This  contradicts the definition of $\Tilde{w}_+$; and proves claim (ii).

To prove claim (iii), indirectly, suppose $w \neq \Tilde{w}$.
Sources, like $w$ and $\Tilde{w}$, only reside in $\mathcal{E}_\pm^2(\mathcal{O})$.
Since $w(0)<0$, $\Tilde{w}(1)<0$, we have $w, \Tilde{w} \in \mathcal{E}_-^2 (\mathcal{O})$.
Moreover $w(0) > \Tilde{w}(0)$, by definition of $w$, and $w(1) >\Tilde{w}(1)$, by definition of $\Tilde{w}$.
Thus proposition~\ref{prop:3.1}(iv) implies $\Tilde{w} < w \in \mathcal{E}_-^2 (\mathcal{O})$, for all (omitted) arguments $0 \leq x \leq 1$.
Define $v$ as the South pole of the face of $\Tilde{w}$, i.e. $ v  $:= $ \Sigma_+^0(\Tilde{w})$.
By $z$-dropping, the monotonically increasing fast unstable heteroclinic manifold from $\Tilde{w}$ to $v$ must stay below $w$, i.e.
	\begin{equation}
	\Tilde{w} < v \leq w \in \mathcal{E}_-^2 (\mathcal{O})\,.
	\label{eq:4.19}
	\end{equation}

To complete the proof of claim~(iii), we obtain a contradiction to \eqref{eq:4.19} by showing
	\begin{equation}
	v(0) > 0\,.
	\label{eq:4.20}
	\end{equation}
Indeed $\mathcal{E}_-^2(\mathcal{O}) \ni\Tilde{w} \leadsto \Tilde{v} = \Tilde{v}_+ \in \mathcal{E}_+^1(\mathcal{O})$, by \eqref{eq:4.11} and property~(ii).
In particular \eqref{eq:4.19} implies $\Tilde{v} \neq v = \Sigma_+^0(\Tilde{w})$.
Hence proposition~\ref{prop:3.1}(iv) locates $\Tilde{v} = \Tilde{v}_+ \in \mathcal{E}_+^1(\Tilde{w})$.
The chain of saddle unstable manifolds in $\mathcal{E}_+^1(\Tilde{w})$ ascends monotonically to the pole $v=\Sigma_+^0(\Tilde{w})$.
In particular $v(0) > \Tilde{v}(0) >0$.
This establishes contradiction \eqref{eq:4.20}, and proves claim~(iii).

To show claim~(iv), indirectly, suppose $v_- \neq  \mathbf{N}$.
Then definition \eqref{eq:4.10} of the North pole $v_- = \Sigma_-^0(w)$ of the face of $w=\Tilde{w} \in \mathcal{E}_-^2 (\mathcal{O})$ implies
	\begin{equation}
	\mathbf{N} < v_- < w \leadsto v_-\,.
	\label{eq:4.21a}
	\end{equation}
In particular $v_- \in \mathcal{E}_-^2(\mathcal{O})$, because $v_-$ cannot block $\mathcal{O} \leadsto \mathbf{N}$.

Let $\hat{v}$ be any equilibrium such that
	\begin{equation}
	v_-(1) < \hat{v} (1) < w (1) < 0\,.
	\label{eq:4.21b}
	\end{equation}
Such equilibria exist, by Morse-adjacency \eqref{eq:1.26} of $h_1$-adjacent equilibria, because the pole $v_- = \mathcal{E}_-^0(w)$ of the $i=2$ source $w$ is an $i=0$ sink.

Suppose first that $\hat{v} \in \mathcal{E}_-^2(\mathcal{O})$.
Then $w,v_-, \hat{v} \in \mathcal{E}_-^2 (\mathcal{O})$, \eqref{eq:4.21b}, proposition~\ref{prop:3.1}(iv) and the definition of $w$ imply
	\begin{equation}
	z(w-\hat{v}) = 0_+ \quad \text{and} \quad z(v_--\hat{v}) 
	\in \lbrace 0_-,1_+ \rbrace\,.
	\label{eq:4.21c}
	\end{equation}
In particular $\hat{v}$ blocks $w \leadsto v_-$ by \eqref{eq:4.9a}, \eqref{eq:4.9b} -- a contradiction to \eqref{eq:4.21a}.

In the remaining case $\hat{v} \not\in \mathcal{E}_-^2(\mathcal{O})\, \cup\, \mathcal{E}_-^0(\mathcal{O})$, we conclude $\hat{v} \in \mathcal{E}_+^1(\mathcal{O})$.
Let us now be more specific: for $\hat{v}$ we choose the $h_1$-successor of the sink $v_-$.
Then $\hat{v}$ is an $i=1$ saddle, and $\hat{v} \leadsto v_-$ by $h_1$-adjacency.
But this implies that the whole edge $W^u(\hat{v})$ lies in the meridian $\Sigma_+^1(\mathcal{O})$.
Consequently $\mathbf{N} \neq v_- \in \mathcal{E}_+^1(\mathcal{O})$.
This contradicts \eqref{eq:4.21a}, $v_- \in \mathcal{E}_-^2(\mathcal{O})$ and proves claim~(iv), $v_- = \mathbf{N}$.

To show claim~(v) we first show 
	\begin{equation}
	\mathcal{E}_+^1(w) \subseteq \mathcal{E}_+^1(\mathcal{O}) \cap
	\lbrace v \in \mathcal{E} \, |\, \mathbf{N}(1) < v(1) < w(1) \rbrace\,.
	\label{eq:4.22a}
	\end{equation}
Indeed let $v \in \mathcal{E}_+^1(w)$.
In particular $v(0) > w(0)$ and $v(1) < w(1) < 0$.
With $v(0) > w(0)$, the definition of $w$ implies $v(0) >0$.
This shows $v \in \mathcal{E}_+^1(\mathcal{O})$ and $v \neq \mathbf{N} = \mathcal{E}_-^0(\mathcal{O})$, as claimed.

To prove (v) it remains to show the converse claim
	\begin{equation}
	\lbrace v \in \mathcal{E} \,|\, \mathbf{N} (1) < v(1) < w(1)\rbrace
	\subseteq \mathcal{E}_+^1(w)\,.
	\label{eq:4.22b}
	\end{equation}
We first observe that the $h_1$-path of all equilibria from $\mathbf{N}$ to $w = \Tilde{w}$ is an $\mathbf{N}$-polar $h_1$-serpent, emanating from $\mathbf{N}$, which terminates at the $h_1$-last saddle $\Tilde{v} = \Tilde{v}_+ \in \mathcal{E}_+^1(\mathcal{O})$ before $w = \Tilde{w}$.
Here we use definitions \eqref{eq:1.25}, \eqref{eq:4.10} and claims~(ii), (iii).
By theorem~\ref{thm:4.1}(iii) this path on the 1-skeleton of the Sturm dynamic complex cannot leave the meridian $\Sigma_+^1(\mathcal{O})$ which it follows from $\underline{v} = \underline{v}_+$ onwards; see claim~(i).
Therefore
	\begin{equation}
	\lbrace v \in \mathcal{E} \, | \, \mathbf{N}(1) < v(1) < w(1) \rbrace
	\subseteq \mathcal{E}_+^1(\mathcal{O})\,,
	\label{eq:4.23}
	\end{equation}
which just barely misses claim \eqref{eq:4.22b}.
Now the definition \eqref{eq:4.10} of $w$ allows us to conclude $z(v-w) = 1_+$, for all $v$ in the left hand side of \eqref{eq:4.23}.
Moreover these elements $v$ are ordered strictly monotonically, by $z=0$, along the meridian $\mathcal{E}_+^1(\mathcal{O})$.
Therefore each of these equilibria $v$ is 1-adjacent to $w$, and the Wolfrum lemma~\ref{lem:4.2} asserts heteroclinic orbits $w \leadsto v$, for each element $v$.
This proves \eqref{eq:4.22b}, claim~(v), and the lemma.
\end{proof}

Thanks to the four trivial equivalences \eqref{eq:2.4}--\eqref{eq:2.9} of definition~\ref{def:2.3}, generated by the linear involutions $u \mapsto -u$ and $x \mapsto 1-x$, the previous lemma comes in four variants.
From \eqref{eq:1.24a} we recall the definitions
	\begin{equation}
	\begin{aligned}
	w_-^\iota := & \quad \text{the } 
	h^\iota\text{-last equilibrium preceding } \mathcal{O}\,,\\
	w_+^\iota := & \quad \text{the } 
	h^\iota\text{-first equilibrium following } \mathcal{O}\,,
	\end{aligned}
	\label{eq:4.24}
	\end{equation}
for $\iota = 0,1$.
By adjacency of Morse numbers, all $w_\pm^\iota$ are $i=2$ sources.
Comparing the notations \eqref{eq:4.10} and \eqref{eq:4.24}, we observe that
	\begin{equation}
	w = w_-^0 \in \Sigma_-^2(\mathcal{O})\,.
	\label{eq:4.25}
	\end{equation}
By lemma~\ref{lem:4.3} and translation table \eqref{eq:1.17}, $w=w_-^0$ has been identified as the source of the unique 2-cell in the hemisphere $\Sigma_-^2(\mathcal{O}) = \mathbf{W}$ which contains the first edge $c_{\underline{v}} = c_{\underline{v}_+}$ of the $\mathbf{WE}$-meridian $\Sigma_+^1(\mathcal{O})$, at its $ \mathbf{N} = \Sigma_-^0(\mathcal{O})$-polar end.
In short: $c_w$ is edge-adjacent to $\mathbf{WE}$ at $\mathbf{N}$.
Define the four faces
	\begin{equation}
	c_\pm^\iota :=\quad c_{w_\pm^\iota}\,,
	\label{eq:4.26}
	\end{equation}
analogously to $c_-^0 = c_w$, with $\iota =0,1$.
Note how the two faces $c_-^\iota$ in the same hemisphere $\Sigma_-^2 (\mathcal{O}) = \mathbf{W}$ may happen to coincide with each other, in special cases, as may $c_+^\iota$ in $\Sigma_+^2(\mathcal{O}) = \mathbf{E}$.
We call, in short
	\begin{equation}
	\begin{aligned}
	 c_-^0 :\quad  & \text{the } 
	\mathbf{NE}\text{-face of } \mathbf{W}\,; \\
	c_-^1 :\quad  & \text{the } 
	\mathbf{NW}\text{-face of } \mathbf{W}\,; \\
	c_+^0 :\quad  & \text{the } 
	\mathbf{SE}\text{-face of } \mathbf{E}\,; \\
	c_+^1 :\quad  & \text{the } 
	\mathbf{SW}\text{-face of } \mathbf{E}\,.
	\end{aligned}
	\label{eq:4.27}
	\end{equation}
The four trivial equivalences of definition~\ref{def:2.3} then provide the following corollary to lemma~\ref{lem:4.3}.

\begin{cor}\label{cor:4.4}\hfill
\begin{itemize}
\item[(i)] the $\mathbf{NW}$-face $c_-^1$ of $\mathbf{W}$ is edge-adjacent to the meridian $\mathbf{EW}$ at $\mathbf{N}$;
\item[(ii)] the $\mathbf{NE}$-face $c_-^0$ of $\mathbf{W}$ is edge-adjacent to the meridian $\mathbf{WE}$ at $\mathbf{N}$;
\item[(iii)] the $\mathbf{SW}$-face $c_+^1$ of $\mathbf{E}$ is edge-adjacent to the meridian $\mathbf{WE}$ at $\mathbf{S}$;
\item[(iv)] the $\mathbf{SE}$-face $c_+^0$ of $\mathbf{E}$ is edge-adjacent to the meridian $\mathbf{EW}$ at $\mathbf{S}$.
\end{itemize}
\end{cor}

With these results and notations we are now exhaustively equipped to complete the proof of theorem~\ref{thm:4.1}(iv), i.e. the boundary edge overlap of the faces $\mathbf{NE}$ with $\mathbf{SW}$ along the meridian $\mathbf{WE}$, and of the faces $\mathbf{NW}$ with $\mathbf{SE}$ along the meridian $\mathbf{EW}$.

\begin{figure}[]
\centering \includegraphics[width=0.5\textwidth]{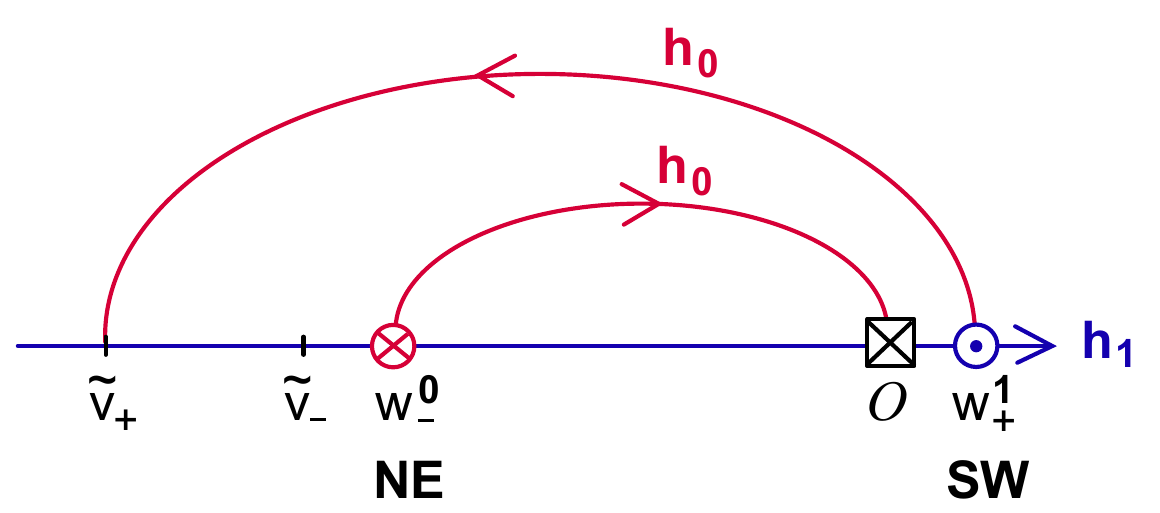}
\caption{\emph{
Edge overlap of face $\mathbf{NE} = c_{w_-^0}$ with face $\mathbf{SW} = c_{w_+^1}$ along the meridian $\mathbf{WE} = \Sigma_+^1(\mathcal{O})$.
Note how the $h_0$-successor $\Tilde{v}_-$ of $ w_+^1$ (nonstrictly) $h_1$-precedes the $h_1$-predecessor $\Tilde{v}_+$ of $w_-^0$, by the meander property of the non-intersecting oriented $h_0$-arcs $w_+^1\Tilde{v}_-$ and $w_-^0 \mathcal{O}$.
}}
\label{fig:4.2}
\end{figure}

\begin{proof}[\textbf{Proof of Theorem~\ref{thm:4.1}(iv)}]
By the trivial equivalence $x \mapsto 1-x$, which preserves $\Sigma_\pm^0(\mathcal{O})$, $\Sigma_\pm^2(\mathcal{O})$ but interchanges meridians and $\iota =0$ with $\iota =1$, it is sufficient to show edge overlap of face $\mathbf{NE} = c_- = c_{w_-^0}$ with face $\mathbf{SW} = c_+= c_{w_+^1}$ along the meridian $\mathbf{WE}= \Sigma_+^1(\mathcal{O})$.
We use the meander property of the shooting meander $h_0$ over the horizontal axis $h_1$; see fig.~\ref{fig:4.2}.
For the geometry in the dynamic Thom-Smale complex see also fig.~\ref{fig:4.1}.

As in \eqref{eq:4.10}, let $\Tilde{v}_+$ denote the $h_1$-last equilibrium $h_1$-before the source $w_-^0$.
In the notation list \eqref{eq:4.10} of lemma~\ref{lem:4.3}, we observe
	\begin{equation}
	w_-^0 = w\,, \quad \Tilde{v}_+ = 
	\Tilde{v} \in \mathcal{E}_+^1 ( \mathcal{O}) 
	\,\cap\, \mathcal{E}_+^1(w)\,.
	\label{eq:4.28}
	\end{equation}
In words, the saddle $\Tilde{v}_+$ in the boundary of the face $\mathbf{NE} = c_-$  of $w_-^0$ lies on the meridian $\mathbf{WE}= \Sigma_+^1(\mathcal{O})$, together with its unstable manifold edge $c_{\Tilde{v}_+}$.

Analogously, let $\Tilde{v}_-$ denote the $h_0$-first equilibrium $h_0$-after the source $w_+^1$.
The trivial equivalence $u \mapsto -u$ together with $x\mapsto 1-x$ then transforms $w_+^1$ to $w$ and $\Tilde{v}_-$ to $\Tilde{v}$, in \eqref{eq:4.10} and lemma~\ref{lem:4.3}.
As a result, analogously to \eqref{eq:4.28}, we obtain
	\begin{equation}
	\Tilde{v}_-  \in \mathcal{E}_+^1 ( \mathcal{O}) 
	\,\cap\, \mathcal{E}_+^1(w_+^1)\,.
	\label{eq:4.29}
	\end{equation}
In words, the saddle $\Tilde{v}_-$ in the boundary of the face $\mathbf{SW} = c_+$ of $w_+^1$ lies on the same meridian $\mathbf{WE}= \Sigma_+^1(\mathcal{O})$ as $\Tilde{v}_+$, together with its unstable manifold edge $c_{\Tilde{v}_-}$.

In fig~\ref{fig:4.2} we have illustrated the nested oriented $h_0$-arcs $w_-^0\mathcal{O}$ and $w_+^1 \Tilde{v}_-$.
Indeed the shooting meander of $h_0$ crosses the horizontal $h_1$-axis upwards, at even Morse numbers, and downwards at odd ones.
Moreover $h_0$ makes a right turn when Morse numbers increase, and a left turn when they decrease; see \eqref{eq:1.22a}.
By the Jordan curve property of meanders, the above two arcs cannot intersect.
Their Morse indices then imply their nesting.

The nesting property of the arc $w_-^0\mathcal{O}$ inside the arc of $w_+^1 \Tilde{v}_-$ implies
	\begin{equation}
	\Tilde{v}_-(1) \leq \Tilde{v}_+(1)\,.
	\label{eq:4.30}
	\end{equation}
Indeed the left end $\Tilde{v}_-$ of the outer arc precedes the left end $\Tilde{v}_+$ of the inner arc, nonstrictly, along the horizontal $h_1$-axis.
The meridian $\mathbf{WE} = \Sigma_+^1(\mathcal{O})$ carries the bipolar orientation from $\mathbf{N}$ to $\mathbf{S}$, by the monotonically increasing order of $z=0$.
In particular \eqref{eq:4.28}--\eqref{eq:4.30} imply that the saddle $\Tilde{v}_- \in \Sigma_+^1(w_+^1) \subseteq \Sigma_+^1(\mathcal{O}) = \mathbf{WE}$ (nonstrictly) precedes the saddle $\Tilde{v}_+\in \Sigma_+^1(w_-^0) \subseteq \Sigma_+^1(\mathcal{O}) = \mathbf{WE}$ on the $\mathbf{N}$ to $\mathbf{S}$ oriented meridian $\mathbf{WE}$, together with their unstable manifold edges; see fig.~\ref{fig:1.1}.
By lemma~\ref{lem:4.3}(v), the face boundaries $\Sigma_+^1(w_+^1)$ and $\Sigma_+^1(w_-^0)$ extend all the way from $c_{\Tilde{v}_-}$ to $\mathbf{S}$ and from $\mathbf{N}$ to $c_{\Tilde{v}_+}$, respectively, on the meridian $\mathbf{WE} = \Sigma_+^1(\mathcal{O})$.
Therefore the boundaries $\partial \mathbf{NE} = \partial c_- \supseteq c_{\Tilde{v}_+}$ and $\partial \mathbf{SW} = \partial c_+ \supseteq c_{\Tilde{v}_-}$ overlap in at least one edge of the meridian $\mathbf{WE}$.
The overlap includes the (possibly identical) edges $c_{\Tilde{v}_\pm}$.
This proves theorem~\ref{thm:4.1}(iv). 
\end{proof}


\section{From 3-cell templates to 3-meander templates}
\label{sec5}

To any prescribed abstract 3-cell template $\mathcal{C} = \bigcup_{v \in \mathcal{E}} c_v$ we formally assign a 3-meander template $\mathcal{M}$, in this section.
See definitions~\ref{def:1.2}, \ref{def:1.3} and figs.~\ref{fig:1.1}, \ref{fig:1.3} for these notions.
In definition~\ref{def:5.1} below, we formally assign an SZS-pair $(h_0,h_1)$ to the 3-cell template $\mathcal{C}$.
Each $h_\iota$, for $\iota=0,1$, can be viewed as a Hamiltonian path, from pole $\mathbf{N}$ to pole $\mathbf{S}$, in the abstract connection graph associated to the abstract regular cell complex $\mathcal{C}$:
vertices are the cell barycenters $v \in \mathcal{E}$, and undirected edges run between each $v$ and all barycenters $\Tilde{v}$ of boundary cells $c_{\Tilde{v}} \subseteq \partial c_v$ of maximal dimension $\dim c_{\Tilde{v}} = \dim c_v-1$.
Theorem~\ref{thm:5.2} then asserts that
	\begin{equation}
	\sigma:= h_0^{-1} \circ h_1
	\label{eq:5.1}
	\end{equation}
is a Sturm permutation and, in fact, the shooting meander $\mathcal{M}$ associated to $\sigma$ is a 3-meander template.
Here we associate Morse numbers, by~\eqref{eq:1.22a}, and a formal shooting curve $\mathcal{M}$ to any permutation $\sigma \in S_N$.
Indeed we may define the curve $\mathcal{M}$ to follow the labels $\sigma(j)$ of $j = 1, \ldots, N$ along the horizontal $h_1$-axis, switching sides at each vertex.
Properties (i)--(iv) of 3-meander templates $\sigma, \mathcal{M}$, as required in definition~\ref{def:1.3}, are proved in lemmata~\ref{lem:5.3}--\ref{lem:5.6} below, respectively.
The meander property of $\sigma, \mathcal{M}$ states that $\mathcal{M}$ is a Jordan curve.
This is proved in lemma~\ref{lem:5.7}.

We caution our reader, once again, that theorem~\ref{thm:5.2}, via the Sturm permutation $\sigma$, only associates ``some'' new Sturm attractor $\widetilde{\mathcal{A}}$ to the original Sturm 3-ball $\mathcal{A}_f = \mathrm{clos\,} W^u(\mathcal{O})$ and to its 3-cell Thom-Smale complex $\mathcal{C} = \mathrm{clos\,}c_\mathcal{O}$ for the 2-sphere hemisphere decomposition $\Sigma_\pm^j$, $0\, \leq j\leq 2$ of $\partial W^u(\mathcal{O})$.
We do not prove $\widetilde{\mathcal{A}} = \mathcal{A}_f$ here.
We do not even prove that $\widetilde{\mathcal{A}} = \mathrm{clos\,} \widetilde{W}^u(\mathcal{O})$ is a Sturm 3-ball.
Only the sequel \cite{firo3d-2} will address this cliffhanger.
In fact we will then show that the dynamic Thom-Smale complex of $\widetilde{\mathcal{A}}$ coincides with the prescribed complex $\mathcal{C}$.
This design of $\widetilde{\mathcal{A}}$ by $\mathcal{C}$ will complete the cycle of template implications~\eqref{eq:4.1}, and will justify our assignment of an SZS-pair $(h_0, h_1)$ of Hamiltonian paths which looks so arbitrary in the following definition.
See fig.~\ref{fig:5.1} for an illustration.

\begin{figure}[t!]
\centering \includegraphics[width=\textwidth]{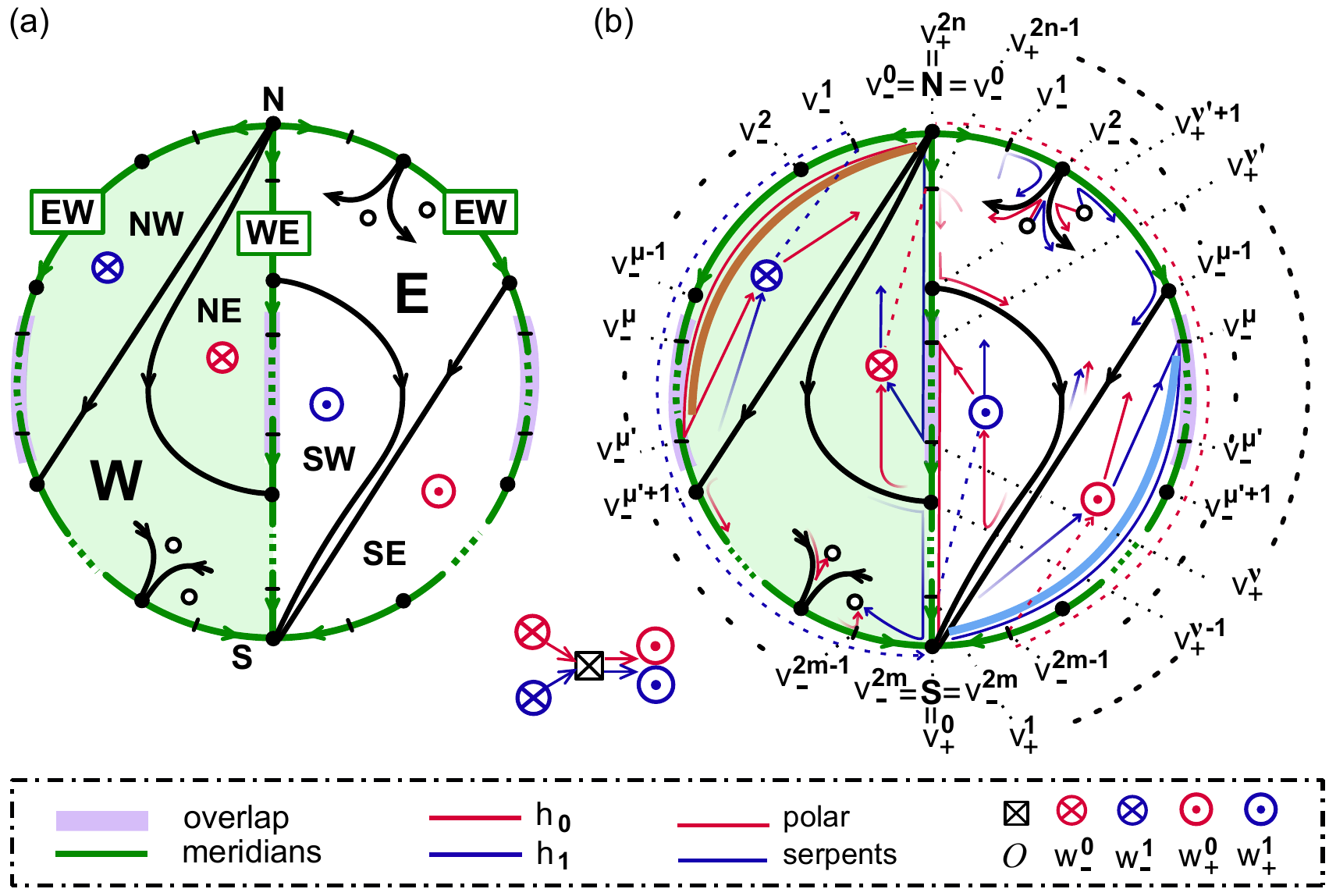}
\caption{\emph{
The SZS-pair $(h_0,h_1)$ in a 3-cell template $\mathcal{C}$, with poles $\mathbf{N}, \mathbf{S}$, hemispheres $\mathbf{W}, \mathbf{E}$  and meridians $\mathbf{EW}, \mathbf{WE}$.
Dashed lines indicate the $h_\iota$-ordering of vertices in the closed hemisphere, when $\mathcal{O}$ and the other hemisphere are ignored, according to definition~\ref{def:5.1}(i).
The actual paths $h_\iota$ tunnel, from $w_ -^\iota \in  \mathbf{W}$ through the 3-cell barycenter $\mathcal{O}$, and re-emerge at $w_+^\iota \in  \mathbf{E}$, respectively.
Note the boundary overlap of the faces $\mathbf{NW}, \mathbf{SE}$ of $w_-^1, w_+^0$ from $v_-^{\mu-1}$ to $v_-^{\mu ' +1}$ on the $\mathbf{EW}$ meridian.
Similarly, the boundaries of the faces $\mathbf{NE}, \mathbf{SW}$ of $w_-^0, w_+^1$ overlap from $v_+^{\nu -1}$ to $v_+^{\nu ' +1}$ along $\mathbf{WE}$.
}}
\label{fig:5.1}
\end{figure}

\begin{defi}\label{def:5.1}
Let $\mathcal{C} = \bigcup_{v \in \mathcal{E}} c_v$ be a 3-cell template with oriented 1-skeleton $\mathcal{C}^1$, poles $\mathbf{N}, \mathbf{S}$, hemispheres $\mathbf{W}, \mathbf{E}$, and meridians $\mathbf{EW}$, $\mathbf{WE}$.
A pair $(h_0, h_1)$ of bijections $h_\iota$: $ \lbrace 1, \ldots , N \rbrace \rightarrow \mathcal{E}$ is called the SZS-\emph{pair assigned to} $\mathcal{C}$ if the following conditions hold.
\begin{itemize}
\item[(i)] The restrictions of range $h_\iota$ to $\mathrm{clos\,} \mathbf{W}$ form an SZ-pair $(h_0, h_1)$, in the closed Western hemisphere.
The analogous restrictions form a ZS-pair $(h_0,h_1)$ in the closed Eastern hemisphere $\mathrm{clos\,} \mathbf{E}$.
See definition~\ref{def:2.2}.
\item[(ii)] In the notation of figs.~\ref{fig:1.1}, \ref{fig:2.3}, and for each $\iota =0,1$, the permutations $h_\iota$ traverse $w_-^\iota, \mathcal{O}, w_+^\iota$, successively.
\end{itemize}
\end{defi}

To see that, indeed, this definition uniquely defines the bijections $h_\iota$ we recall lem-ma~\ref{lem:2.7}.
By the orientation of boundary edges in definition~\ref{def:1.2}(iii) of a 3-cell template, the hemisphere closures $\mathrm{clos\,} \mathbf{W}$ and $\mathrm{clos\,} \mathbf{E}$ are Western and Eastern, in the sense of definition~\ref{def:2.6}.
For example, consider the SZ-pair $(h_0,h_1)$ in $\mathrm{clos\,} \mathbf{W}$.
Then the S-path $j \mapsto h_0(j)$, ordered by its labels $j \in \lbrace 1, \ldots , N \rbrace$, traverses all vertices $v$ of $\mathbf{N} \cup \mathbf{EW} \cup \mathbf{W}$, before finishing in bipolar order along $\mathbf{WE} \cup \mathbf{S}$.
Indeed, the $\mathbf{S}$-polar serpent of $h_0$, restricted to $\mathrm{clos\,} \mathbf{W}$, is full by lemma~\ref{lem:2.7}.
In particular $w_-^0$ is the last vertex of $h_0$, not only in $\mathbf{W}$ but, in $\mathbf{N} \cup \mathbf{EW} \cup \mathbf{W}$.
Similarly, the condition that $(h_0,h_1)$ be a ZS-pair in $\mathrm{clos\,} \mathbf{E}$ requires $j \mapsto h_0(j)$ to traverse all vertices of $\mathbf{N} \cup \mathbf{EW}$ before $\mathbf{S} \cup \mathbf{WE} \cup \mathbf{E}$, in the order of $j$, by lemma~\ref{lem:2.7}.
The first vertex of $h_0$ in $\mathbf{S} \cup \mathbf{WE} \cup \mathbf{E}$ is therefore $w_+^0$.
This shows $h_0$ is well-defined in $\mathrm{clos\,} \mathbf{W}$ and $\mathrm{clos\,} \mathbf{E}$, by the planar theorem~\ref{thm:2.4}.
Compatibility of the requirements in definition \ref{def:5.1}(i) on the intersection 
	\begin{equation}
	\mathrm{clos\,} \mathbf{W} \cap \mathrm{clos\,} \mathbf{E} = 
	( \mathbf{N} \cup \mathbf{EW} ) \cup (\mathbf{WE} \cup \mathbf{S})
	\label{eq:5.2}
	\end{equation}
follows from the consistent bipolar orientation of the 1-skeleton of the 3-cell template $\mathcal{C}$.
Compatibility of the requirements (i) and (ii) was shown above.

We include a second, perhaps more direct, argument for uniqueness of the SZS-pair $(h_0,h_1)$.
The bipolar orientation of the 3-cell template $\mathcal{C}$ fixes the orders of $h_0$ and $h_1$ uniquely on the 1-skeleton of $\mathcal{C}$.
The SZ- and ZS-requirements of (i) determine how $h_\iota$ traverses each face, except for the faces of the $h_\iota$-neighbors $w_\pm^\iota$ of $\mathcal{O}$.
That final missing piece is uniquely prescribed to be $w_-^\iota \mathcal{O}w_+^\iota$, by requirement (ii) of definition~\ref{def:5.1}.
This assigns a unique SZS-pair $(h_0, h_1)$ of Hamiltonian paths, from pole $\mathbf{N}$ to pole $\mathbf{S}$, for any given 3-cell template $\mathcal{C}$.

It is useful to summarize the precise \emph{ordering of vertices} $\xi, \eta$ in poles, hemispheres, and separating meridians, as induced by the bijective labelings $h_\iota$ in a more formal notation.
We define
	\begin{equation}
	\xi <_\iota \eta \quad : \Longleftrightarrow
	\quad h_\iota^{-1} ( \xi ) < h_\iota^{-1}(\eta )\,.
	\label{eq:5.4a}
	\end{equation}		
Let $(h_0,h_1)$ be the SZS pair of $\mathcal{C}$.
Then definition~\ref{def:5.1} implies the vertex orderings
	\begin{align}
	\mathbf{N} \cup \mathbf{EW} \cup \mathbf{W} 
	 \quad <_0 \quad \mathcal{O} \quad <_0 \quad
	\mathbf{E} \cup \mathbf{WE} \cup \mathbf{S}\,,
	\label{eq:5.4b}\\
	\mathbf{N} \cup \mathbf{WE} \cup \mathbf{W} 
	 \quad <_1	\quad \mathcal{O} \quad <_1 \quad
	\mathbf{E} \cup \mathbf{EW} \cup \mathbf{S}\,.
	\label{eq:5.4c}
	\end{align}
We cannot resist the temptation of a consistency check with our overall intentions here: 
the equivalence of Sturm 3-ball attractors $\mathcal{A}$, 3-cell templates $\mathcal{C}$, and 3-meander templates $\mathcal{M}$.
If the abstract 3-cell template $\mathcal{C}$ had originated as the dynamic Thom-Smale complex of a 3-ball attractor $\mathcal{A}$ -- \emph{which we do not assume} -- then the orderings \eqref{eq:5.4b}, \eqref{eq:5.4c} would indeed be implied by the decomposition of $\Sigma^2(\mathcal{O}) = \partial W^u(\mathcal{O})$ into the hemispheres $\Sigma_\pm^j(\mathcal{O})$, $\, j=0,1,2$.

Let us at least motivate our seemingly arbitrary definition of SZS-pairs $(h_0, h_1)$, in view of this overall intention.
Consider any 2-cell $\mathrm{clos\,} c_\mathcal{O}$, with the temporary notation $i(\mathcal{O}) =2$ of section~\ref{sec2}, in a Sturm global attractor $\mathcal{A}$.
We claim that, in general, $h_0$ and $h_1$ must traverse $\mathcal{O}$ as indicated, for a single 2-cell $\mathcal{A}= \mathrm{clos\,} c_\mathcal{O}$, in fig.~\ref{fig:2.2}.
Indeed we have heteroclinic orbits $\mathcal{O} \leadsto w_\pm^\iota$ for the immediate predecessors $w_-^\iota$ and successors $w_+^\iota$ of $\mathcal{O}$ in the ordering of $h^\iota$ at $x= \iota = 0,\,1$, because blocking of immediate neighbors is not possible.
The only possible exception arises if $i(w_\pm^\iota) =3$; we exclude this case for a moment and assume $i(w_\pm^\iota) =1$.
Note $w_\pm^0 \in \mathcal{E}_\pm^1(\mathcal{O})$ and $w_\pm^1 \in \mathcal{E}_\mp^1(\mathcal{O})$.
We claim that the $h_0$-neighbors $w_\pm^0$ of $\mathcal{O}$ are the $h_1$-extrema of $\mathcal{E}_\pm^1(\mathcal{O})$.
Indeed let $v \in \mathcal{E}_-^1(\mathcal{O}) \smallsetminus w_-^0$.
Then $z(v-w_-^0) =0$ by the bipolar orientation of $\mathrm{clos\,}\Sigma_-^1(\mathcal{O})$ from the pole $\mathbf{N}=\mathcal{E}_-^0(\mathcal{O})$ to $\mathbf{S}=\mathcal{E}_+^0(\mathcal{O})$.
Because $w_-^0$ is the $h_0$-predecessor of $\mathcal{O}$, we know $v-w_-^0 < 0$ at $x=1$.
This proves $h_1$-maximality of $w_-^0$ in $\mathcal{E}_-^1(\mathcal{O})$.
The argument for $h_1$-minimality of $w_+^0$ in $\mathcal{E}_+^1(\mathcal{O})$ is analogous.
Likewise, the $h_1$-neighbors $w_\pm^1$ of $\mathcal{O}$ are the $h_0$-extrema of $\mathcal{E}_\mp^1(\mathcal{O})$.
Together this proves that $(h_0,h_1)$ traverse the disk $\mathrm{clos\,} c_\mathcal{O} \subseteq \mathcal{A}$, in general, as indicated in fig.~\ref{fig:2.2} and as required for ZS-pairs.

Of course, one ambiguity arises in this argument.
We might well reflect fig.~\ref{fig:2.2}(b) through a vertical axis, i.e. interchange the boundary labels $\Sigma_+^1$ and $\Sigma_-^1$, to obtain an SZ-pair $(h_0,h_1)$ instead.
Returning to 3-cells $\mathcal{A} = \mathrm{clos\,} c_\mathcal{O}$, $\,i(\mathcal{O}) =3$, this ambiguity appears, necessarily, when we draw the hemispheres $\mathbf{W}$ and $\mathbf{E}$, alias $\Sigma_\pm^2(\mathcal{O})$, in the same plane; see fig.~\ref{fig:1.1}.
Then the meridians $\mathbf{WE} = \Sigma_+^1$ and $\mathbf{EW} = \Sigma_-^1$ enforce SZ-pairs $(h_0,h_1)$ in $\mathbf{W} = \Sigma_-^2$, and ZS-pairs $(h_0,h_1)$ in $\mathbf{E} = \Sigma_+^2$, as required in definition~\ref{def:5.1}, by the opposite planar orientations of $\mathbf{W}$ and $\mathbf{E}$ in their graphical representation.
We also note the exceptional role of the face barycenters $w_\pm^\iota$ which possess the unique equilibrium with Morse index $i=3$ as one $h_\iota$-neighbor.

\begin{figure}[t!]
\centering \includegraphics[width=\textwidth]{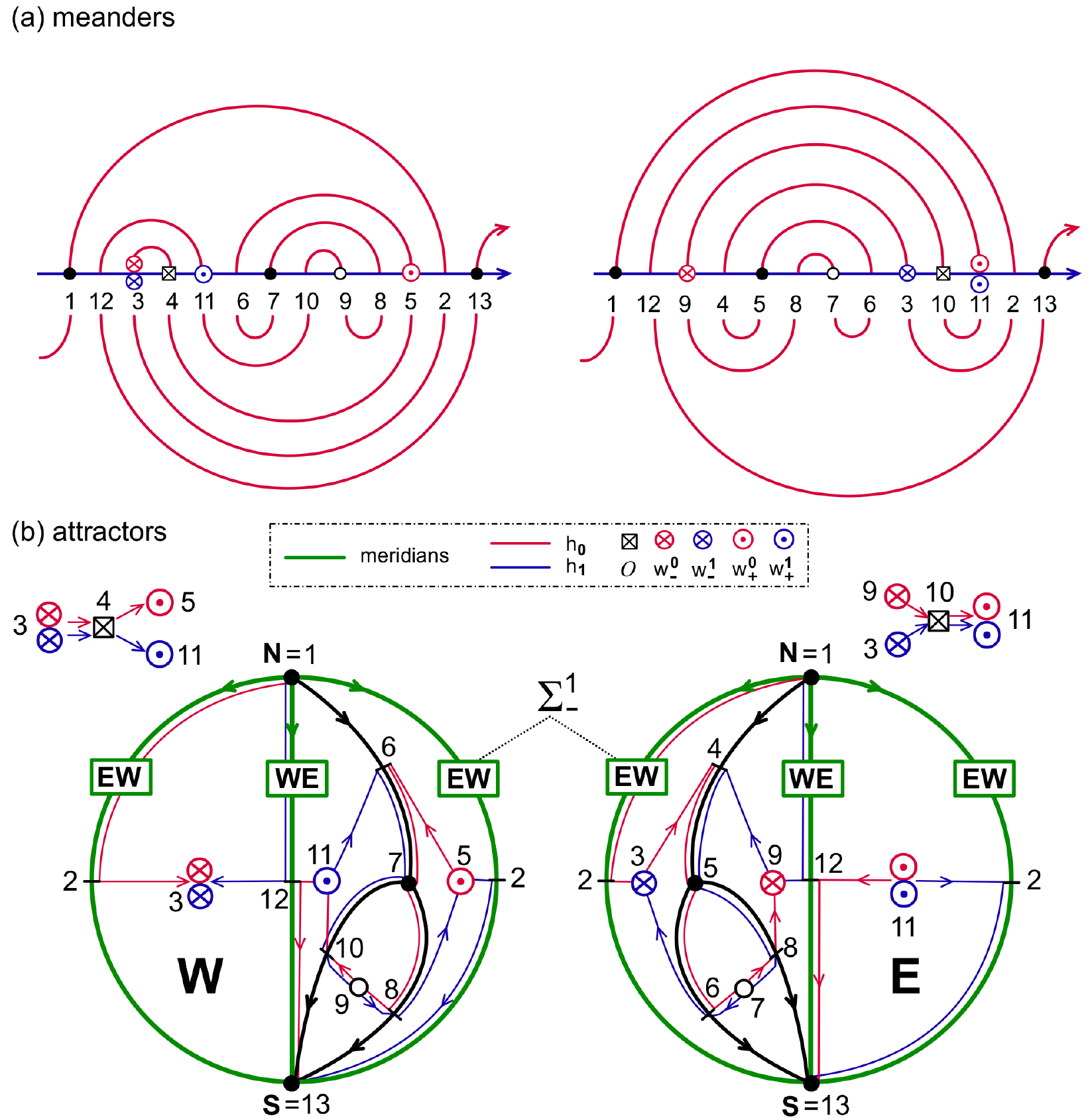}
\caption{\emph{
Two mirror-symmetric 3-ball attractors $\mathcal{A}^+$ (left) and $\mathcal{A}^-$ (right), in (b).
See (a) and \eqref{eq:5.3} for their respective meanders with Sturm permutations $\sigma^+$ (left) and $\sigma^-$ (right).
}}
\label{fig:5.2}
\end{figure}

Let us further illustrate this ambiguity by two Sturm 3-ball attractors $\mathcal{A}^\pm$ which coincide up to a reversal of orientation in $\mathbb{R}^3$.
Their Sturm permutations are
	\begin{equation}
	\begin{aligned}
	\sigma ^+
	&:=  \text{(1 12 3 4 11 6 7 10 9 8 5 2 13)}
	= \text{(2 12) (5 11) (8 10)}\,;\\
	\sigma ^-
	&:=  \text{(1 12 9 4 5 8 7 6 3 10 11 2 13)}
	= \text{(2 12) (3 9) (6 8)}\,.
	\end{aligned}
	\label{eq:5.3}
	\end{equation}
See fig.~\ref{fig:5.2}(a) for the respective meanders, and fig.~\ref{fig:5.2}(b) for the respective signed 2-hemisphere templates.
The templates are mirror symmetric to each other, in $\mathbb{R}^3$, and are also related by a $180^\circ$ rotation around the polar axis.
Note that the involutions $\sigma^\pm$ are not conjugate to each other by any of the trivial equivalences \eqref{eq:2.4}--\eqref{eq:2.9}.
To our knowledge this is the first, and simplest, such example for the closure of a single 3-cell $c_\mathcal{O}$.
Note, however, that neither of the permutations $\sigma^\pm$ can be realized as the Sturm permutation of any nonlinearity $f=f(u)$ which only depends on $u$; see \cite{firowo12}.
Indeed the core equivalent permutation cycles $\text{(5 11), (2 12)}$ in $\sigma^+$, and $\text{(3 9), (2 12)}$ in $\sigma^-$, respectively, are not centered.
A related example involving a 3-ball with an attached meridian disk and a polar spike was suggested, but not worked out, in \cite{wo02}.

We can now formulate the main result of this section.

\begin{thm}\label{thm:5.2}
Let $(h_0,h_1)$ be the SZS-pair associated to the 3-cell template $\mathcal{C}$, according to definition~\ref{def:5.1}.

Then the permutation $\sigma$:= $h_0^{-1} \circ h_1$ defines a Sturm meander $\mathcal{M}$ which is a 3-meander template.
\end{thm}

Throughout the proof of theorem~\ref{thm:5.2}, below, we assume $(h_0,h_1)$ to be the SZS-pair of $\mathcal{C}$.
It is helpful to consult and compare fig.~\ref{fig:5.1} of the SZS -pair $(h_0,h_1)$ for the general 3-cell template $\mathcal{C}$ with fig.~\ref{fig:1.3} of the general 3-meander template for the ``meander'' $\mathcal{M}$ which $\sigma=h_0^{-1} \circ h_1$ is claimed to define.
We use identical notation for corresponding equilibria, in these two figures.

Trivially, $\sigma$ is dissipative in the sense of \eqref{eq:1.20}.
Indeed both paths $h_\iota$ start at the pole $\mathbf{N} = h_\iota (1)$ and terminate at the opposite pole $\mathbf{S} = h_\iota(N)$.
Hence $\sigma(1) =1$ and $\sigma(N) =N$, as required in \eqref{eq:1.20}.

A central element in the proof of theorem~\ref{thm:5.2} is the following trivial consequence of definition~\ref{def:5.1}(i).
The ordering of equilibria defined by the restriction of both $h_\iota$ to $\mathrm{clos\,} \mathbf{W}$ coincides with the ordering of the defining SZ-pair $(h_0,h_1)$ on that closed hemisphere.
The same statement applies to the ZS-restriction of $(h_0,h_1)$ to $\mathrm{clos\,} \mathbf{E}$, of course.
We call this elementary observation the property of \emph{order restriction}.
Lemmata~\ref{lem:5.3}--\ref{lem:5.7} all assume $(h_0,h_1)$ to be the SZS-pair of the 3-cell template $\mathcal{C}$ with associated permutation $\sigma$:= $h_0^{-1} \circ h_1 \in S_N$.

\begin{lem}\label{lem:5.3}
The permutation $\sigma$ is Morse, with Morse numbers
	\begin{align}
	0 \leq\,  &i_v \leq 2\,, \quad \text{for all} \quad
	v \in \mathcal{E} \smallsetminus \lbrace \mathcal{O} \rbrace\,;
	\label{eq:5.5a}\\
	&i_\mathcal{O} = 3\,.
	\label{eq:5.5b}
	\end{align}
\end{lem}

\begin{proof}[\textbf{Proof.}]
In our discussion of definition~\ref{def:5.1} we have already observed how the path $h_0$: $\lbrace 1, \ldots ,$ $N \rbrace \rightarrow \mathcal{E}$ traverses all equilibria $v$ such that
	\begin{equation}
	v \in \mathbf{N} \cup \mathbf{EW} \cup \mathbf{W} 
	\subseteq \mathrm{clos\,} \mathbf{W}
	\label{eq:5.6}
	\end{equation}
before arriving at $\mathcal{O}$ from $w_-^0$.
See \eqref{eq:5.4b}.
By recursion \eqref{eq:1.22a} along $h_0$, order restriction to the planar cell complex $\mathrm{clos\,} \mathbf{W}$ shows $i_v \leq 2$ for all $v$ of \eqref{eq:5.6}; see section~\ref{sec2}.
Interchanging the roles of $h_0$ and $h_1$, by the trivial equivalence $x \mapsto 1-x$, the same statement holds true for all $v$ traversed by $h_1$ before $\mathcal{O}$, i.e. for all
	\begin{equation}
	v \in \mathbf{N} \cup \mathbf{WE} \cup \mathbf{W}\,.
	\label{eq:5.7}
	\end{equation}
Together, \eqref{eq:5.6} and \eqref{eq:5.7} prove claim \eqref{eq:5.5a} in $\mathrm{clos\,}\mathbf{W} \smallsetminus \mathbf{S}$.
The trivial equivalence $u \mapsto -u$, similarly, allows us to consider claim \eqref{eq:5.5a} proved for all 
	\begin{equation}
	v \in \mathbf{S} \cup \mathbf{EW} \cup \mathbf{E}\,.
	\label{eq:5.9}
	\end{equation}

It only remains to determine $i_\mathcal{O}$.
By recursion \eqref{eq:1.22a} and the normalization $i_\mathbf{N} = i_{h_0(1)} =0$, the even/odd parities of $j$ and $i_{h_0(j)}$ are opposite, for all $j$.
In particular $w_-^0 = h_0(2r-1)$, for some integer $r$, because $w_-^0 \in \mathbf{W}$ implies $i_{w_-^0}=i(w_-^0) =2$ for the face barycenter $w_-^0$, by \eqref{eq:5.6}.
By definition, $\mathcal{O} = h_0(2r)$.
Hence \eqref{eq:1.22a} implies
	\begin{equation}
	i_\mathcal{O} = i_{w_-^0} + 
	\text{sign }(h_1^{-1} (\mathcal{O})-h_1^{-1}(w_-^0)) 
	= 2+1 = 3\,.
	\label{eq:5.8}
	\end{equation}
Indeed $h_1$ traverses $w_-^0 \in \mathbf{W}$ (nonstrictly) $h_1$-before $w_-^1 \in \mathbf{W}$, and hence strictly $h_1$-before $\mathcal{O}$, by definition~\ref{def:5.1}.
This proves claim \eqref{eq:5.5b}, and the lemma. 
\end{proof}

\begin{lem}\label{lem:5.4}
Polar $h_\iota$-serpents overlap with the anti-polar $h_{1-\iota}$-serpents.
\end{lem}

\begin{proof}[\textbf{Proof.}]
By trivial equivalences \eqref{eq:2.4}--\eqref{eq:2.9} it is sufficient to establish the overlap, i.e. the nonempty intersection, of the $\mathbf{N}$-polar $h_0$-serpent with the anti-polar, i.e. $\mathbf{S}$-polar, $h_1$-serpent.
In the closed Western hemisphere $\mathrm{clos\,}\mathbf{W}$, the $\mathbf{N}$-polar serpent $h_0$ is easily identified.
See fig.~\ref{fig:5.1}.
Indeed the S-path $h_0$ in $\mathrm{clos\,} \mathbf{W}$ starts with $\mathbf{N}= v_-^0$ as $v_-^0 v_-^1 \ldots v_-^{\mu'} w_-^1 \ldots $, first leaving the meridian $\mathbf{EW}$ from $v_-^{\mu'}$ to the face barycenter $w_-^1$.
Since $i \leq 1$ along the meridian, in general, and along the left boundary of the face $\mathbf{NW} = c_{w_-^1}$, in particular, the $\mathbf{N}$-polar $h_0$-serpent is
	\begin{equation}
	h_0:\quad v_-^0v_-^1 \ldots v_-^{\mu'}\,.
	\label{eq:5.10}
	\end{equation}
For the Z-path $h_1$ in $\mathrm{clos\,} \mathbf{E}$, on the other hand, we obtain the termination sequence $h_1$: $\ldots w_+^0v_-^\mu \ldots$ $v_-^{2m-1}v_-^{2m}$ to $v_-^{2m} = \mathbf{S}$, with $\mathbf{S}$-polar $h_1$-serpent
	\begin{equation}
	h_1: \quad v_-^\mu \ldots v_-^{2m-1} v_-^{2m}\,.
	\label{eq:5.11}
	\end{equation}
The meridian overlap condition (iv) for the boundaries of the faces $\mathbf{NW} = c_-^1$ and $\mathbf{SE} = c_+^0$ in definition~\ref{def:1.2} of a 3-cell template $\mathcal{C}$ implies $\mu \leq \mu'$ and hence nonempty overlap	
	\begin{equation}
	v_-^\mu \ldots v_-^{\mu'}
	\label{eq:5.12}
	\end{equation}
of the anti-polar serpents \eqref{eq:5.10} and \eqref{eq:5.11}.
This proves the lemma.
\end{proof}

\begin{lem}\label{lem:5.5}
The meander intersection $v=\mathcal{O}$ is located between the two intersection points, in the order of $h_1$, of the polar arc of any polar $h_0$-serpent. 
See the polar arcs $v_\pm^0 v_\pm^1$ in fig. \ref{fig:1.3}. 
The same statement holds true with $h_0$ and $h_1$ interchanged.
\end{lem}

\begin{proof}[\textbf{Proof.}]
Again we invoke the trivial equivalences \eqref{eq:2.4}--\eqref{eq:2.9} to consider the polar arc from $\mathbf{N} = v_-^0 = h_0^{-1}(1)$ to $h_0^{-1}(2) = v_-^1$ of the $\mathbf{N}$-polar $h_0$-serpent \eqref{eq:5.10}, only, without loss of generality.
We have to show
	\begin{equation}
	1 = h_1^{-1}(v_-^0) < h_1^{-1}(\mathcal{O})
	< h_1^{-1}(v_-^1)\,;
	\label{eq:5.13}
	\end{equation}
see \eqref{eq:5.4a}.
Because $\mathcal{O} \neq \mathbf{N}$, the left inequality is trivial.
To show the right inequality we first note that the $i=1$ saddle $v_-^1$ is the immediate $h_1$-successor of $w_-^1$ along the Z-path $h_1$ on $\mathrm{clos\,} \mathbf{W}$, in the restricted order of $h_1$ on that closed hemisphere.
However, the actual path $h_1$ tunnels to the Eastern hemisphere 
	\begin{equation}
	h_1: \quad \ldots w_-^1 \mathcal{O} w_+^1 \ldots v_-^1 \ldots
	\label{eq:5.14}
	\end{equation}
immediately after $w_-^1$, by definition~\ref{def:5.1}(ii).
By the restricted order in $\mathrm{clos\,} \mathbf{W}$, however, this still implies $h_1^{-1} (v_-^1)>h_1^{-1}(w_+^1)=h_1^{-1}(\mathcal{O})+1$.
This proves the lemma.
\end{proof}

\begin{lem}\label{lem:5.6}
The $h_\iota$-neighbors $w_\pm^\iota$ of $\mathcal{O}$ are the $h_{1-\iota}$-extreme $i=2$ sources.
\end{lem}

\begin{proof}[\textbf{Proof.}]
By trivial equivalences \eqref{eq:2.4}--\eqref{eq:2.9} it is sufficient to prove that $w_-^1 \in \mathbf{W}$ is the $h_0$-first $i=2$ source in $\mathbf{W}$.
This follows as in the proof of lemma~\ref{lem:5.4} where we observed the starting sequence
	\begin{equation}
	h_0: \quad v_-^0 v_-^1 \ldots v_-^{\mu'} w_-^1 \ldots
	\label{eq:5.15}
	\end{equation}
in our discussion of the $\mathbf{N}$-polar $h_0$-serpent \eqref{eq:5.10} from $\mathbf{N} =v_-^0$ to $v_-^{\mu'}$.
Because $i \leq 1$ on serpents, this shows that $w_-^1$ is indeed the $h_0$-first $i=2$ source -- and the lemma is proved.
\end{proof}

\begin{figure}[t!]
\centering \includegraphics[width=\textwidth]{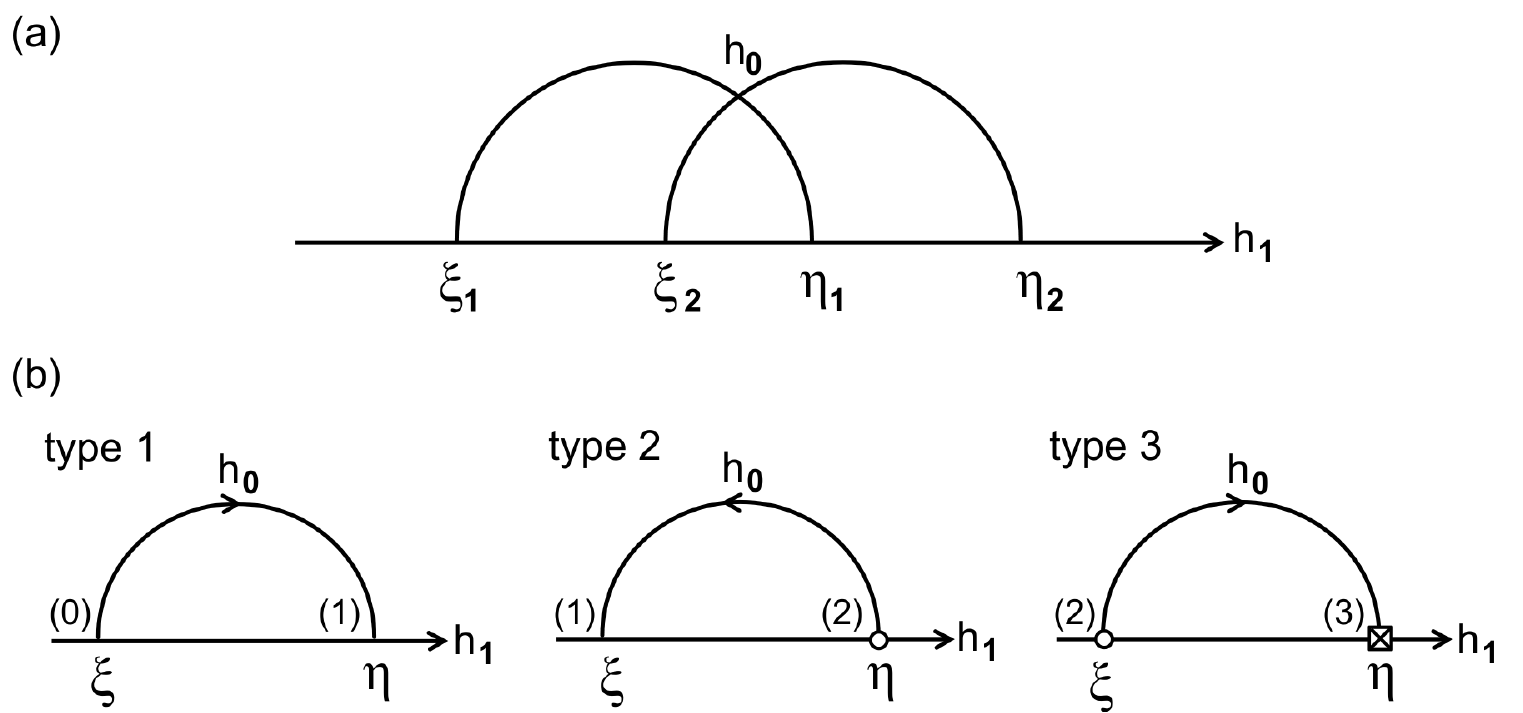}
\caption{\emph{
Upper arcs of $h_0$ between $h_0$-adjacent vertices $\xi <_1 \eta$, in the $h_1$-order;
(a) conflicting arcs $(\xi_1, \eta_1)$ and $(\xi_2, \eta_2)$ contradict the meander property of $\sigma$:= $h_0^{-1} \circ h_1$;
(b) the three types of upper arcs for SZS-pairs $(h_0,h_1)$ associated to 3-cell templates $\mathcal{C}$:
 type $i$ is characterized by $\text{max } \lbrace i_\xi , i_\eta \rbrace = i = 1,2,3$.
Note the orientation of $h_0$ induced by opposite even/odd parities of $\xi \equiv i_\eta \not\equiv i_\xi \equiv \eta \pmod {2}$.
 Morse numbers are indicated in parentheses and by vertex symbols.
}}
\label{fig:5.3}
\end{figure}

\begin{lem}\label{lem:5.7}
The permutation $\sigma$:= $h_0^{-1} \circ h_1$ is a meander permutation.
\end{lem}

As a preparation for several cases in the proof of this meander lemma, for $\sigma$ and the associated formal shooting curve $\mathcal{M}$, we first comment on $h_0$\emph{-arcs} $(\xi, \eta)$ of $\mathcal{M}$ over the horizontal $h_1$-axis.
These are defined by $h_0$-adjacency and $h_1$-ordering
	\begin{equation}
	h_0^{-1}(\eta) =
	h_0^{-1}(\xi) \pm 1\,, \quad
	\xi <_1 \eta\,.
	\label{eq:5.16}
	\end{equation}
Here, as in \eqref{eq:5.4a}, we abbreviate the $h_\iota$-order $h_\iota^{-1}(\xi) < h_\iota^{-1}(\eta)$ by $\xi <_\iota \eta$.
The order $\xi <_1 \eta$ simply labels the left vertex of the arc as $\xi$, and the right vertex as $\eta$.
By the trivial equivalence $u \mapsto -u$ of \eqref{eq:2.4}, i.e. by orientation reversal of $h_0$ and $h_1$, it is sufficient to avoid \emph{conflicts of upper arcs}, as illustrated in fig.~\ref{fig:5.3}(a).
In symbols, our indirect proof will derive a contradiction from the conflict assumption
	\begin{equation}
	\xi_1 \, <_1 \, \xi_2 \, <_1 \, \eta_1 \, <_1 \, \eta_2
	\label{eq:5.17}
	\end{equation}
on any two arcs $(\xi_j, \eta_j)$ above the horizontal $h_1$-axis.
By adjacency \eqref{eq:1.22a} of Morse numbers at $h_0$-adjacent vertices along $\mathcal{M}$, and because upper and lower arcs alternate along the shooting curve defined by the ``arbitrary'' permutation $\sigma = h_0^{-1} \circ h_1$, the Morse numbers and $h_0$-orientations of any upper arc $(\xi, \eta)$ belong to one of the three types $i= \text{max } \lbrace i_\xi , i_\eta \rbrace = 1,2,3$ in fig.~\ref{fig:5.3}(b).
The following observation will be used repeatedly in the proof of meander lemma~\ref{lem:5.7}.

\begin{prop}\label{prop:5.8}
Suppose an upper $h_0$-arc $(\xi, \eta)$ satisfies $\eta \neq \mathcal{O}$, i.e. the upper arc of $\mathcal{M}$ is of type 1 or type 2.
Then
	\begin{equation}
	\xi \in \mathbf{W}\quad \Longrightarrow \quad \eta \in \mathbf{W}\,.
	\label{eq:5.18}
	\end{equation}
\end{prop}

\begin{proof}[\textbf{Proof.}]
We recall the $h_\iota$-orders $<_\iota$ of poles, hemispheres, and meridians according to \eqref{eq:5.4b}, \eqref{eq:5.4c}.
By \eqref{eq:5.4b}, the assumption $\xi \in \mathbf{W}$ implies $\xi <_0 \mathcal{O}$.
Since $\eta \neq \mathcal{O}$ is $h_0$-adjacent to $\xi$, along the arc $(\xi, \eta)$, this implies $\eta <_0 \mathcal{O}$.
Equivalently,
	\begin{equation}
	\eta \in \mathbf{N} \cup \mathbf{EW} \cup \mathbf{W}
	\subseteq \mathrm{clos\,} \mathbf{W}
	\label{eq:5.19}
	\end{equation}
again by \eqref{eq:5.4b}.
It remains to exclude $\eta = \mathbf{N}$, and $\eta \in \mathbf{EW}$.

To exclude $\eta = \mathbf{N}$, we only have to observe $i_\mathbf{N} = 0 < i_\eta$, for all types.

Next, suppose indirectly that $\eta \in \mathbf{EW}$.
Then $i_\eta \in \lbrace 0,1 \rbrace$, as always on any meridian, by restricted ordering on $\mathrm{clos\,}\mathbf{W}$.
This identifies the upper arc $(\xi, \eta)$ to be of type 1, according to the list of types in fig.~\ref{fig:5.3}(b).
In particular
	\begin{equation}
	i_\xi =0\,, \quad i_\eta =1\,,
	\label{eq:5.20}
	\end{equation}
and $\eta$ is the $h_0$-successor of $\xi$.
On the other hand, $h_0$ is an S-path in the planar domain \eqref{eq:5.19}, by definition~\ref{def:5.1}(i).
By \eqref{eq:5.20}, both $\xi$ and $\eta$ belong to the 1-skeleton of $\mathrm{clos\,} \mathbf{W}$.
Hence $\eta \in \mathbf{EW}$ and $i_\eta =1$ imply $\xi \in \mathbf{N} \cup \mathbf{EW}$.
This contradicts our assumption $\xi \in \mathbf{W}$, and proves the lemma.
\end{proof}

\begin{proof}[\textbf{Proof of lemma~\ref{lem:5.7}.}]
We show, indirectly and without loss of generality, that a conflict \eqref{eq:5.17} of upper arcs $(\xi_1,\eta_1)$, $(\xi_2, \eta_2)$ cannot arise; see fig.~\ref{fig:5.3}(a).

First suppose the vertex pairs $(\xi_1,\eta_1)$ and $(\xi_2, \eta_2)$ belong to one and the same of the two hemispheres
	\begin{equation}
	\begin{aligned}
	&\mathrm{clos\,}\mathbf{W} 
	&=& \,\mathbf{N} \cup \mathbf{S} \cup \mathbf{EW} 
	\cup \mathbf{WE} \cup \mathbf{W}
	&=& \,\mathcal{C} \smallsetminus 
	\lbrace \mathbf{E} \cup \mathcal{O}\rbrace\,,\\
	&\mathrm{clos\,}\mathbf{E} 
	&=& \, \mathbf{N} \cup \mathbf{S} \cup \mathbf{EW} 
	\cup \mathbf{WE} \cup \mathbf{E}
	&=& \,\mathcal{C} \smallsetminus 
	\lbrace \mathbf{W} \cup \mathcal{O}\rbrace\,.
	\end{aligned}
	\label{eq:5.21}
	\end{equation}
Then the property of order restriction, explained above, prevents any arc conflict \eqref{eq:5.17}, in view of theorem \ref{thm:2.4}(i). 
See our comments on theorem~\ref{thm:5.2} above. 
We may therefore assume that the vertex pairs $(\xi_1,\eta_1)$ and $(\xi_2, \eta_2)$ are not members of the same closed hemisphere in\eqref{eq:5.21}.

Suppose one vertex of $(\xi_j, \eta_j)$ is the barycenter $\mathcal{O}$ of the 3-cell template $\mathcal{C} = \mathrm{clos\,} c_\mathcal{O}$.
Then the list of upper arcs in fig.~\ref{fig:5.3}(b) implies
	\begin{equation}
	\xi_j = w_-^0\,, \quad \eta_j = \mathcal{O}\,.
	\label{eq:5.22}
	\end{equation}
We postpone this case, for the moment.

For the remaining cases we may assume $\xi_j, \eta_j \neq \mathcal{O}$, for all $j=1,2$.
Then the pairs $(\xi_j, \eta_j)$ must lie in opposite closed hemispheres \eqref{eq:5.21}.
The total order $<_1$ by $h_1$ then implies that $\mathcal{O}$ cannot be strictly below or strictly above the $h_1$-order of the conflict assumption \eqref{eq:5.17}; see the $h_1$-order \eqref{eq:5.4c} of poles, meridians, and hemispheres.
This leaves us with the three cases
	\begin{align}
	&\xi_1 <_1 \mathcal{O} <_1 \xi_2 <_1 \eta_1 <_1 \eta_2\,;
	\label{eq:5.23}\\
	&\xi_1 <_1 \xi_2 <_1 \mathcal{O} <_1 \eta_1 <_1 \eta_2\,;
	\label{eq:5.24}\\
	&\xi_1 <_1 \xi_2 <_1 \eta_1 <_1 \mathcal{O} <_1 \eta_2\,;
	\label{eq:5.25}
	\end{align}
which we now treat, one by one.

Consider the $h_1$-order \eqref{eq:5.23} first. Then $\xi_1 \in \mathbf{W}$, or else both pairs $(\xi_j,\eta_j)$ belong to $\mathrm{clos\,}\mathbf{E}$; see \eqref{eq:5.4c} and \eqref{eq:5.21}.
Since $\eta_1 \neq \mathcal{O}$, proposition~\ref{prop:5.8} implies $\eta_1 \in \mathbf{W}$; see \eqref{eq:5.18}.
Hence $\eta_1 <_1 \mathcal{O}$, by \eqref{eq:5.4c}, which contradicts our assumption \eqref{eq:5.23}.

Next assume the $h_1$-order \eqref{eq:5.24}.
Since $(\xi_1, \eta_1)$ and $(\xi_2, \eta_2)$ must belong to opposite closed hemispheres, \eqref{eq:5.24} and \eqref{eq:5.4c} together imply at least one of $\xi_1, \xi_2 \in \mathbf{W}$ and at least one of $\eta_1, \eta_2 \in \mathbf{E}$; see \eqref{eq:5.21}.
Since definition~\ref{def:5.1}(ii) prevents direct $h_0$-passages from $\mathbf{W}$ to $\mathbf{E}$ we conclude
	\begin{equation}
	\xi_1 \in \mathbf{W}\,, \, \eta_2\in \mathbf{E}\,, \quad
	\text{or} \quad \xi_2 \in \mathbf{W}\,, \, \eta_1 \in \mathbf{E}\,.
	\label{eq:5.26}
	\end{equation}
In either case we obtain $\xi_j, \eta_j \in \mathbf{W}$ for at least one arc $(\xi_j, \eta_j)$; see proposition~\ref{prop:5.8}, \eqref{eq:5.18} again.
In either case, \eqref{eq:5.4c} provides a contradiction to assumption \eqref{eq:5.24} for $\eta_j$.

The last case \eqref{eq:5.25} is of a slightly different flavor.
Opposite closed hemispheres \eqref{eq:5.4c} and \eqref{eq:5.25} imply
	\begin{equation}
	\eta_2 \in \mathbf{E}\,, \quad
	\xi_2 \not\in \lbrace \mathbf{N}, \mathbf{S} \rbrace \,,
	\label{eq:5.27}
	\end{equation}
and one of $\xi_1, \eta_1 \in \mathbf{W}$, this time.
By $h_0$-adjacency of $\xi_1, \eta_1$, the orderings \eqref{eq:5.4b} and \eqref{eq:5.4c}, respectively, with \eqref{eq:5.25} then imply the two conclusions
	\begin{equation}
	\begin{aligned}
	&\xi_1, \eta_1 \in \mathbf{N} \cup \mathbf{EW} \cup \mathbf{W}\,,
	\quad \text{and}\\
	&\xi_1, \eta_1 \in \mathbf{N} \cup \mathbf{WE} \cup \mathbf{W}\,.
	\end{aligned}
	\label{eq:5.28}
	\end{equation}
Because $\xi_1, \eta_1$ are $h_0$-adjacent, with one vertex in $\mathbf{W}$, nontriviality of the $\mathbf{N}$-polar $h_0$-serpent allows us to conclude
	\begin{equation}
	\xi_1, \eta_1 \in \mathbf{W}\,,
	\label{eq:5.29}
	\end{equation}
from \eqref{eq:5.28}, for both vertices.
Since $\xi_2 \not\in \lbrace \mathbf{N}, \mathbf{S}\rbrace$ is $h_1$-between $\xi_1$ and $\eta_1$, by assumption \eqref{eq:5.25}, but $h_0$-adjacent to $\eta_2 \in \mathbf{E}$, \eqref{eq:5.4b} and \eqref{eq:5.4c} similarly imply
	\begin{equation}
	\xi_2 = v_+^k \in \mathbf{WE}\,,
	\label{eq:5.30}
	\end{equation}
for some $1 \leq k \leq 2m-1$.
See fig.~\ref{fig:5.1}.
More precisely, $k < \nu \leq \nu '$.
Indeed the $h_1$-position of $\xi_2$ is $h_1$-between $\xi_1, \eta_1 \in \mathbf{W}$, and therefore $\xi_2$ cannot belong to the $\mathbf{N}$-polar $h_1$-serpent $v_+^{2n} \ldots v_+^{\nu}$.
By overlap lemma~\ref{lem:5.4} for the $\mathbf{N}$-polar $h_1$-serpent with the antipodal $\mathbf{S}$-polar $h_0$-serpent $v_+^{\nu'} \ldots v_+^0$ along $v_+^{\nu'} \ldots v_+^\nu$, however,
	\begin{equation}
	\xi_2 = v_+^k \quad \text{with} \quad k < \nu \leq \nu'
	\label{eq:5.31}
	\end{equation}
belongs to that $\mathbf{S}$-polar $h_0$-serpent, excepting its initial point $v_+^{\nu'}$.
Therefore the $h_0$-adjacent neighbor $\eta_2$ of $\xi_2$ must belong to the same $\mathbf{S}$-polar $h_0$-serpent in $\mathbf{WE} \cup \mathbf{S}$.
This contradicts $\eta_2 \in \mathbf{E}$ and eliminates case \eqref{eq:5.27}.

To complete the proof of meander lemma~\ref{lem:5.7} it only remains to consider the two postponed cases \eqref{eq:5.22} with indirect ordering assumption \eqref{eq:5.17}, and reach a contradiction.
We first consider the case $j=1$ in \eqref{eq:5.22}, i.e. $\xi_1= w_-^0$,  $\ \eta_1 =\mathcal{O}$, and
	\begin{equation}
	\mathbf{W} \ni w_-^0 <_1 \xi_2 <_1 \mathcal{O} <_1 \eta_2\,.
	\label{eq:5.32}
	\end{equation}
In this case the $h_1$-order \eqref{eq:5.4c} implies $\xi_2 \in \mathbf{N} \cup \mathbf{WE} \cup \mathbf{W}$ and $\eta_2 \in \mathbf{S} \cup \mathbf{EW} \cup \mathbf{E}$.
By \eqref{eq:5.18} of proposition~\ref{prop:5.8}, this excludes $\xi_2 \in\mathbf{W}$.
The $\mathbf{N}$-polar $h_1$-serpent $v_+^{2n} \ldots v_+^{\nu}<_1 w_-^0$ cannot contain $\xi_2 >_1 w_-^0 \in \mathbf{W}$ either.
Therefore overlap lemma~\ref{lem:5.4} again implies \eqref{eq:5.31}, i.e. both $\xi_2$ and $\eta_2$ belong to the same $\mathbf{S}$-polar $h_0$-serpent $v_+^{\nu'} \ldots v_+^0$ in $\mathbf{WE} \cup \mathbf{S}$.
In particular
	\begin{equation}
	\eta_2 \in ( \mathbf{S} \cup \mathbf{EW} \cup \mathbf{E}) \cap
	(\mathbf{WE} \cup \mathbf{S}) = \mathbf{S}
	\label{eq:5.33}
	\end{equation}
has Morse number $i_{\eta_2}=0$.
This contradicts the type classification of fig.~\ref{fig:5.3}(b) for the upper $h_0$-arc $(\xi_2, \eta_2)$.

The last remaining case is $j=2$ in \eqref{eq:5.22}, i.e. $\xi_2 = w_-^0$, $\eta_2 = \mathcal{O}$ and
	\begin{equation}
	\xi_1 <_1 w_-^0 <_1 \eta_1 <_1 \mathcal{O}\,.
	\label{eq:5.34}
	\end{equation}
By lemma~\ref{lem:5.6}, $\xi_1$ still belongs to the $\mathbf{N}$-polar $h_1$-serpent $v_+^{2n} \ldots v_+^{\nu} \subseteq \mathbf{N} \cup \mathbf{WE}$, whereas its $h_0$-neighbor $\eta_1 >_1 w_-^0$ does not.
Suppose $\xi_1 = \mathbf{N}$.
Then $(\xi_1,\eta_1) = (\mathbf{N}, v_-^1)$ is the first $\mathbf{N}$-polar $h_0$-arc, and hence the $\eta_1 = v_-^1 >_1 \mathcal{O} >_1 \mathbf{N}$ by lemma~\ref{lem:5.5}.
This contradicts \eqref{eq:5.34}.

This leaves $\xi_1 \in v_+^{2n-1} \ldots v_+^\nu \subseteq \mathbf{WE}$ to be considered, with 
	\begin{equation}
	\mathcal{O} <_0 \xi_1, \eta_1 \quad \text{and} \quad
	\eta_1 \in \mathbf{S} \cup \mathbf{WE} \cup \mathbf{E}\,,
	\label{eq:5.35}
	\end{equation}
by \eqref{eq:5.4b}.
On the other hand $\eta_1 <_1 \mathcal{O}$ and \eqref{eq:5.4c} imply
$\eta_1 \in \mathbf{N} \cup \mathbf{WE} \cup \mathbf{W}$, and hence
	\begin{equation}
	\eta_1 \in \mathbf{WE}\,.
	\label{eq:5.36}
	\end{equation}
Summarizing, this shows $\xi_1 = v_+^k\,, \, \eta_1 = v_+^\ell$, with
	\begin{equation}
	0 < \ell < \nu \leq k < 2n\,,
	\label{e	q:5.37}
	\end{equation}
because $\xi_1$ belongs to the $\mathbf{N}$-polar $h_1$-serpent, whereas $\eta_1 \in \mathbf{WE}$ does not.
By $h_0$-adjacency of $(\xi_1, \eta_1)$ this implies
	\begin{equation}
	\xi_1 = v_+^{\nu}\,,\, \eta_1 = v_+^{\nu-1}\,.
	\label{eq:5.38}
	\end{equation}
Here $i_{\xi_1} =1$, because the termination $\xi_1$ of the $\mathbf{N}$-polar $h_1$-serpent is the $h_1$-predecessor of the $i=2$ source $w_-^0$; see lemma~\ref{lem:5.6}.
Hence $\eta_1 \in \mathbf{WE}$ implies $i_{\eta_1}=0$.
This contradicts the types in the upper $h_0$-arc list of fig.~\ref{fig:5.3}(b).

This final contradiction completes our indirect proof of the meander lemma~\ref{lem:5.7}.
\end{proof}

\begin{proof}[\textbf{Proof of theorem~\ref{thm:5.2}.}]
We summarize the above results.
Let $(h_0,h_1)$ be the unique SZS-pair of bijective paths $h_\iota$: $\lbrace 1, \ldots , N \rbrace \rightarrow \mathcal{E}$ through the vertex set $v \in \mathcal{E}$ of the 3-cell template bipolar regular cell complex $\mathcal{C}= \bigcup_{v \in \mathcal{E}} c_v$, from pole $\mathbf{N} = h_\iota (1)$ to pole $\mathbf{S} = h_\iota (N)$, $\, \iota=0,1$; see definition~\ref{def:5.1}.
In particular the permutation $\sigma$:= $h_0^{-1} \circ h_1 \in S_N$ is dissipative.
By lemma~\ref{lem:5.3}, $\sigma$ is Morse.
By lemma~\ref{lem:5.7}, $\sigma$ is a meander permutation with dissipative Morse meander $\mathcal{M}$.
Hence $\sigma, \mathcal{M}$ are Sturm; see definition~\ref{def:1.3} and \eqref{eq:1.19a}--\eqref{eq:1.24}.

In lemmata~\ref{lem:5.3}--\ref{lem:5.6}, respectively, we have shown that the Sturm meander $\mathcal{M}$ also satisfies conditions (i)--(iv) of a 3-meander template; see definition~\ref{def:1.3}.
This completes the proof of theorem~\ref{thm:5.2}.
\end{proof}

We conclude this section with a brief review of fig.~\ref{fig:1.3}.
We recall how the 3-meander template $\mathcal{M}$ resulted from the SZS-construction of $h_0,h_1$, via the permutation $\sigma$:= $h_0^{-1}\circ h_1$.
First note the overlapping antipodally polar pairs of serpents $h_\iota, h_{1-\iota}$ of lemma~\ref{lem:5.4}.
By lemma~\ref{lem:5.5} the polar $h_0$ arcs overarch the single $i=3$ vertex $\mathcal{O}$, as an upper arc from $\mathbf{N}$ and as a lower arc from $\mathbf{S}$.
The $\mathbf{N}$-polar $h_0$-serpent $v_-^0 \ldots v_-^{\mu'}$ then continues, left to right by \eqref{eq:1.22a}, with alternating Morse numbers, to the immediate $h_0$-predecessor $v_-^{\mu'}$ of the immediate $h_1$-predecessor $w_-^1$ of $\mathcal{O}$.
See lemma~\ref{lem:5.6}.
Similarly, the immediate $h_1$-successor $v_-^\mu$ of the immediate $h_0$-successor $w_+^0$ of $\mathcal{O}$ is the termination point of the $\mathbf{S}$-polar $h_1$-serpent $v_-^\mu \ldots v_-^{2m}$.
Note how $v_-^\mu$ also belongs to the $\mathbf{N}$-polar $h_0$-serpent, by overlap $\mu \leq \mu'$.
Also by overlap lemma~\ref{lem:5.4}, the remaining equilibria $v_-^{\mu} \ldots v_-^{2m}$ between $v_-^{\mu}$ and $v_-^{2m} = \mathbf{S}$ on the $h_1$-axis are of alternating Morse numbers 1 and 0, because they belong to the $\mathbf{S}$-polar $h_1$-serpent.
This identifies the upper arc boundary of the 3-meander template of fig.~\ref{fig:1.3} to enumerate the meridian closure $\mathbf{N} \cup \mathbf{EW} \cup \mathbf{S}$.
Similarly, the lower arc boundary of the vertices $v_+^{2n} \ldots v_+^0$ enumerates the other meridian closure $\mathbf{N} \cup \mathbf{WE} \cup \mathbf{S}$ of the 3-cell template of fig.~\ref{fig:1.1}; see also fig.~\ref{fig:5.1}.

The realization of the planar regular bipolar cell complex $\mathrm{clos\,}\mathbf{W}$ by the restricted SZ-pair $(h_0,h_1)$, as the planar Sturm attractor of the restricted Sturm meander can be visualized by the following \emph{Eastern scoop} construction, in the 3-meander template of fig.~\ref{fig:1.3}.
In fact we have to remove $\mathcal{O}$ and all vertices in $\mathbf{E}$, by the scoop.
By orderings \eqref{eq:5.4b}, \eqref{eq:5.4c}, $v \in \mathbf{E}$ are characterized by
	\begin{equation}
		w_+^0 \leq_0 v\leq_0 w_+^1 \quad \text{and} \quad w_+^1 \leq_1 v \leq_1 w_+^0\,, 
	\label{eq:5.39}
	\end{equation}
due to antipodal polar serpent overlap and definition~\ref{def:5.1}(ii).
In fig.~\ref{fig:1.3} these are precisely the vertices from $w_+^1$ to $w_+^0$, on the horizontal $h_1$-axis, which do not belong to an upper boundary $h_0$-arc of the meridian $\mathbf{EW}$.
The scoop construction removes all these vertices, together with $\mathcal{O}$, and replaces this segment of $\mathcal{M}$ by a single left-oriented upper $h_0$-arc from $w_-^0$ to $v_+^{2n-1}$.
The $\mathbf{S}$-polar $h_0$-serpent becomes full, accordingly, spanning all of $v_+^{2n-1} \ldots v_+^1 v_+^0$.
Indeed all vertices from $v_+^{2n-1}$ downs to $v_+^{\nu'}$ have lost their possible $h_0$-arc partners in $\mathbf{E}$ by our scoop.
An analogous \emph{Western scoop} can remove $\mathcal{O}$ and $\mathbf{W}$, instead.

In the sequel \cite{firo3d-2}, these scoop constructions will serve as a first step to show that, in fact, the dynamic Thom-Smale complex $\widetilde{\mathcal{C}}$ of the global Sturm attractor $\widetilde{\mathcal{A}}$, constructed from $h_0, h_1$ and
$\sigma = h_0^{-1} \circ h_1$, $\mathcal{M}$ above, coincides with the given cell complex $\mathcal{C}$.
At least for the closed hemispheres $\mathrm{clos\,}\mathbf{W}$ and $\mathrm{clos\,}\mathbf{E}$ this is sketched, but not proved, by the above scoop.


\section{Some solid Sturm octahedra}
\label{sec6}

We illustrate some of our results on Sturm 3-cell templates and 3-meander templates for the case of solid three-dimensional octahedra $\mathbb{O}= \mathrm{clos\,} c_\mathcal{O}$.
The 3-cell boundary $\partial c_\mathcal{O}$ of $\mathcal{O}=27$ consists of eight face triangle 2-cells $19, \ldots , 26$, six 0-cell vertices $1, \ldots , 6$, and twelve 1-cell edges $7, \ldots , 18$.
See fig.~\ref{fig:6.1} for this template in planar form.
The face 26, with vertices $1,2,3$ and boundary edges $7,8,11$, can be thought of as hidden behind the seven faces $19, \ldots , 25$.
For an easy insertion of the Hamiltonian pole-to-pole paths we instead draw the $i=2$ barycenter 26 of the backwards face, in duplicate, outside the triangle $1, 2,3$.
We omit the interior $i=3$ ball vertex $\mathcal{O} =27$, at this stage.

\begin{figure}[t!]
\centering \includegraphics[width=\textwidth]{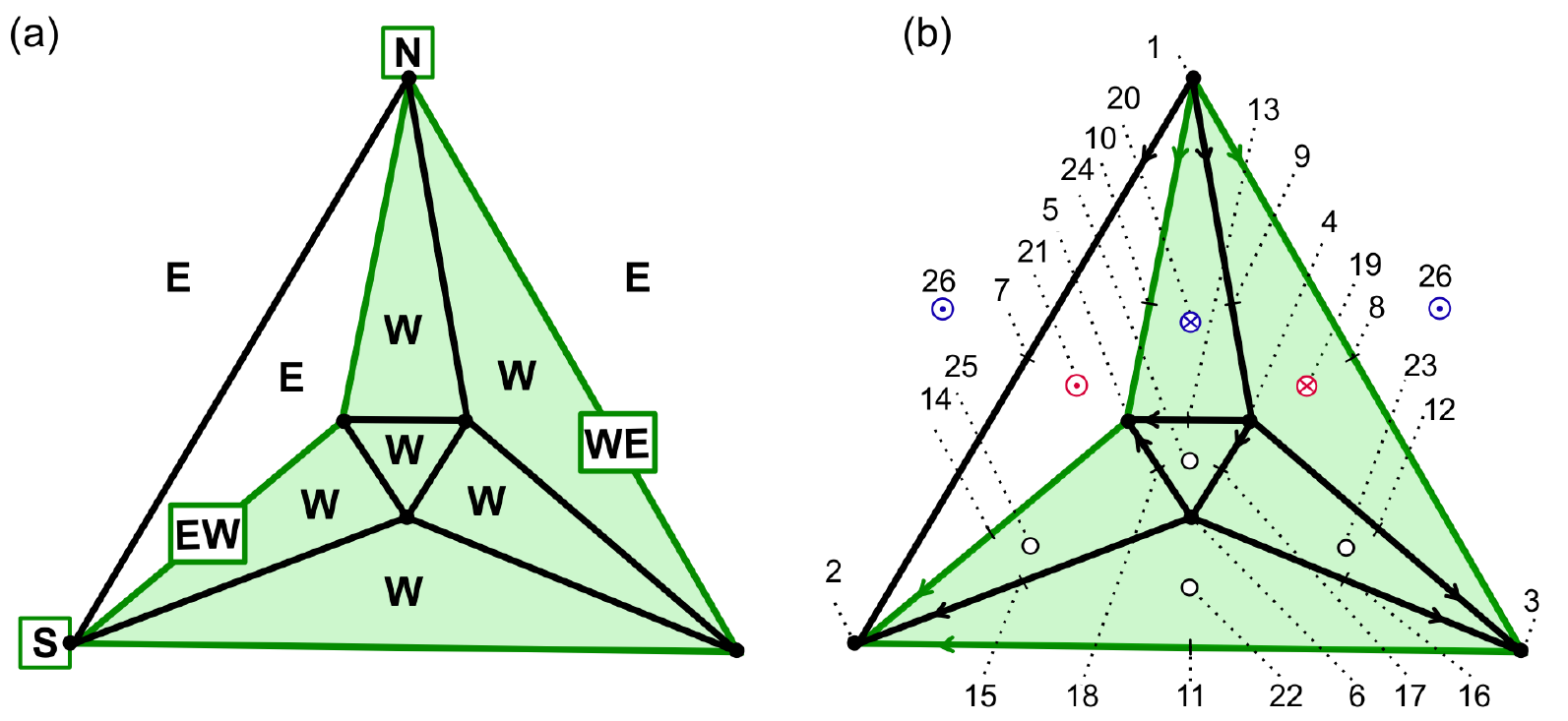}
\caption{\emph{
The solid three-dimensional Sturm octahedron $\mathbb{O}$.
(a) A meridian hemisphere decomposition $\mathbf{E}, \mathbf{W}$ into $2+6$ faces.
(b) The 3-cell template; vertex $i=0$ cells $\bullet = 1, \ldots , 6$; edge $i=1$ cells $7, \ldots , 18$; face $i= 2$ cells $\circ = 19, \ldots , 26$. Labels indicate barycenters.
The backwards face barycenter $26$ of triangle $1,2,3$ is drawn outside that triangle, in duplicate.
The $i=3$ cell barycenter $\mathcal{O} =27$ of the resulting octahedral 2-sphere $\partial c_{\mathcal{O}}$ is omitted.
The bipolar orientation of the 1-skeleton is implied by bipolarity and the boundary orientation of theorem~\ref{thm:4.1}(i),(iii).
The $h_\iota$-neighbors $w_\pm^\iota$ of $\mathcal{O} =27$ are $w_-^0 =19,\, w_-^1=20,\, w_+^0=21,$ and $w_+^1=26.$
}}
\label{fig:6.1}
\end{figure}

About a decade ago, for lack of scientific understanding, we launched a brute force computational attack on Sturm realizations of the solid octahedron.
The orientation poles $\lbrace \mathbf{N} , \mathbf{S} \rbrace$:= $\lbrace P_1, P_2 \rbrace$ must be either edge adjacent or antipodally opposite vertices.
We determined all Hamiltonian path candidates $h_\iota$, from pole to pole.
Adjacent poles provided 62552 paths $h_\iota$, and antipodal poles provided 70944 choices.
We then scanned for pairs $(h_0,h_1)$ which might lead to Sturm permutations $\sigma$:= $h_0^{-1} \circ h_1$.
Alas, not a single antipodal Sturm meander \eqref{eq:1.19b} emerged.
Thinking of the poles as the uppermost and lowermost equilibria of the Sturm attractor, in the strictly monotone order of $z=0$, we found this result rather counterintuitive: 
at first, we suspected a programming error.
On the other hand, a few consistent Sturm cases did appear, for edge adjacent poles $\mathbf{N}, \mathbf{S}$.

In proposition~\ref{prop:6.1}, based on our analysis in the present paper, we prove that solid Sturm octahedra necessarily possess edge adjacent poles $P_1,P_2$.
This excludes antipodal poles.
In the case of adjacent poles, proposition~\ref{prop:6.2} shows that the hemispheres $\lbrace \mathbf{E}, \mathbf{W}\rbrace = \lbrace H_1, H_2\rbrace$ either consist of $1+7$ or of $2+6$ faces.
We also derive the SZS-pair $(h_0,h_1)$ and the 3-meander template $\sigma = h_0^{-1} \circ h_1$, for the essentially unique $2+6$ case; see figs.~\ref{fig:6.3} and \ref{fig:6.4}.
For a full discussion, including the remaining type $1+7$ and many more examples, we refer to our sequel \cite{firo3d-3}.

\begin{figure}[t!]
\centering \includegraphics[width=\textwidth]{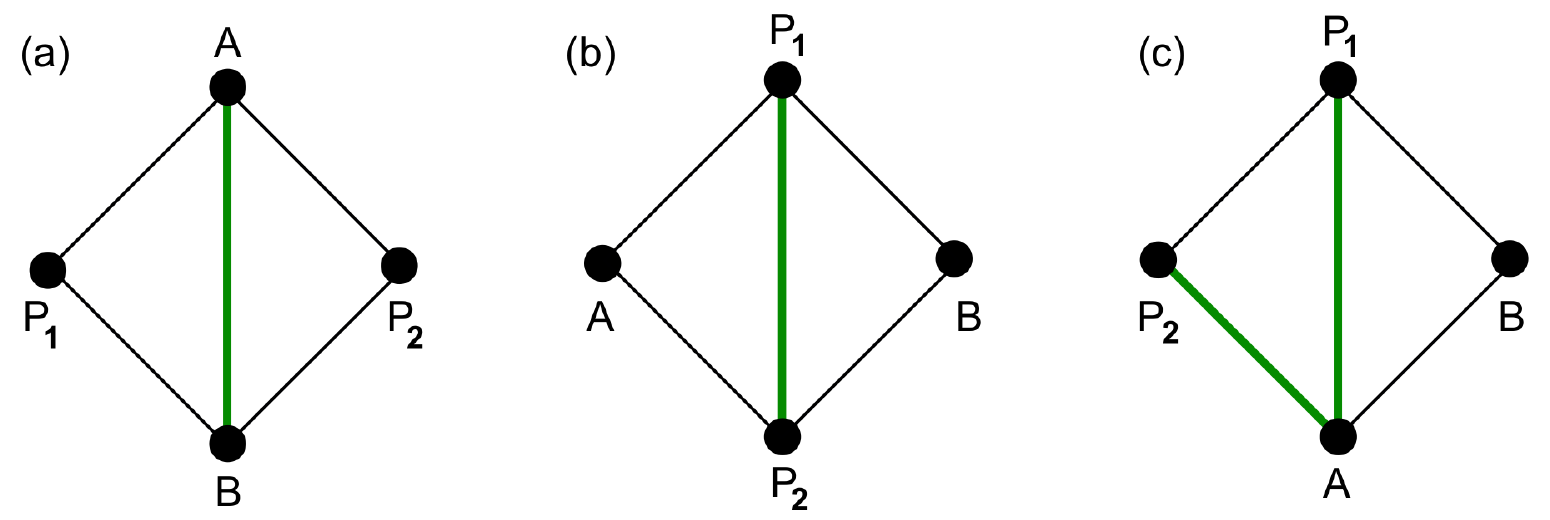}
\caption{\emph{
Meridian overlap of anti-polar triangle faces in opposite octahedral hemispheres:
(a) antipodal poles;
(b) and (c) edge adjacent poles.
Meridian edges are indicated in bold face green.
See also propositions~6.1 and 6.2.
}}
\label{fig:6.2}
\end{figure}

\begin{prop}\label{prop:6.1}
Let $\mathcal{A}$ be a Sturm global attractor with a dynamic regular Thom-Smale complex $\mathcal{C}$ which is a solid three-dimensional octahedron $\mathbb{O} = \mathrm{clos\,} c_\mathcal{O}$.

Then the poles $\mathbf{N}, \mathbf{S} = \Sigma_\pm^0(\mathcal{O})$ are edge-adjacent in $\mathcal{C}$.
\end{prop}

\begin{proof}[\textbf{Proof.}]
By theorem~\ref{thm:4.1}, the octahedral dynamic Thom-Smale complex $\mathcal{C}$ must be a 3-cell template.
Let $P_1, P_2$ denote the poles $\mathbf{N}, \mathbf{S} = \Sigma_\pm^0(\mathcal{O})$, ignoring the details of the bipolar orientation.
Suppose indirectly that the poles are antipodal, rather than edge-adjacent, in the octahedron $\mathcal{C}$.
By meridian edge overlap, theorem~\ref{thm:4.1}(iv), anti-polar faces with boundary in the same meridian share an edge along that meridian.
Because both overlapping faces are triangles, and because the poles are antipodal, this leads to the situation of fig.~\ref{fig:6.2}(a): the meridian overlap edge $AB$ is disjoint from the poles $P_1,P_2$.

By theorem~\ref{thm:4.1}(ii),(iv), the other meridian must also provide an overlap edge $A'B'$, disjoint from $AB$ and from the poles.
Consider the square $ABA'B'$ of nonpolar vertices and edges, in $\mathcal{C} = \mathrm{clos\,}c_\mathcal{O}$, which separates the poles in $\partial c_\mathcal{O}$.
Because the two meridians are disjoint, by theorem~\ref{thm:4.1}(ii), the remaining two edges $BA'$ and $B'A$ must belong to hemisphere interiors.
Consider the hemisphere $H_1 \in \lbrace \mathbf{W} , \mathbf{E} \rbrace$ of $BA'$.
In $H_1$, the edge $BA'$ connects the non-polar meridian boundary points $B$ and $A'$.
Whatever the bipolar orientation may be, on $BA'$, and whichever hemisphere $H_1$ we may consider, this contradicts the requirement of theorem~\ref{thm:4.1}(iii) that non-polar edges either be all oriented towards the meridians, or else all away from them, in any one hemisphere.
This contradiction proves the proposition.
\end{proof}

\begin{prop}\label{prop:6.2}
Consider octahedral Sturm global attractors $\mathcal{A}, \mathcal{C}$ as in proposition~\ref{prop:6.1}.
Then each meridian possesses at most two edges.
In particular, one of the hemispheres possesses at most two faces.
\end{prop}

\begin{proof}[\textbf{Proof.}]
By proposition~\ref{prop:6.1} the poles $P_1,P_2 =\Sigma_\pm^0$ are edge-adjacent.
Consider two anti-polar triangles with meridian edge overlap, again, given by theorem~\ref{thm:4.1}(iv) as in the proof of proposition~\ref{prop:6.1} or in fig.~\ref{fig:6.2}.
Since the poles $P_1,P_2$ are edge-adjacent, this time, we arrive at one of the situations of fig.~\ref{fig:6.2}(b),(c), for each meridian.
Here we have used that the two triangles must be polar to opposite poles along opposite sides of the same meridian.
In particular this excludes a situation where $P_2 P_1$ is a meridian edge in fig.~\ref{fig:6.2}(c), but $P_2A$ is not.

Since each meridian in fig.~\ref{fig:6.2}(b),(c) happens to be complete, pole-to-pole, this proves that each of the two meridians can possess at most two edges.
The meridians decompose $S^2= \partial c_\mathcal{O}$ into two hemispheres.
The smaller hemisphere consists of at most four triangle faces.
We discuss the arising cases separately.

First suppose the smaller hemisphere consists of three triangles.
The circumference for three edge contingent triangles in $\partial c_\mathcal{O}$ is 5.
This exceeds the maximal budget of 4 edges, two for each full meridian of type fig.~\ref{fig:6.2}(c).

Next suppose the smaller hemisphere consists of four triangles.
However, the minimum circumference 4 of four contingent triangles, all sharing one common vertex, cannot be realized either. 
Rather, the two required copies of fig.~\ref{fig:6.2}(c) then produce two meridians which encompass the edge $P_1 P_2$. That results in only two triangles for the encompassed smaller hemisphere.
Combining fig.~\ref{fig:6.2}(b) and (c) leads to a single face hemisphere, of course.
This proves the proposition.
\end{proof}

\begin{figure}[t!]
\centering \includegraphics[width=\textwidth]{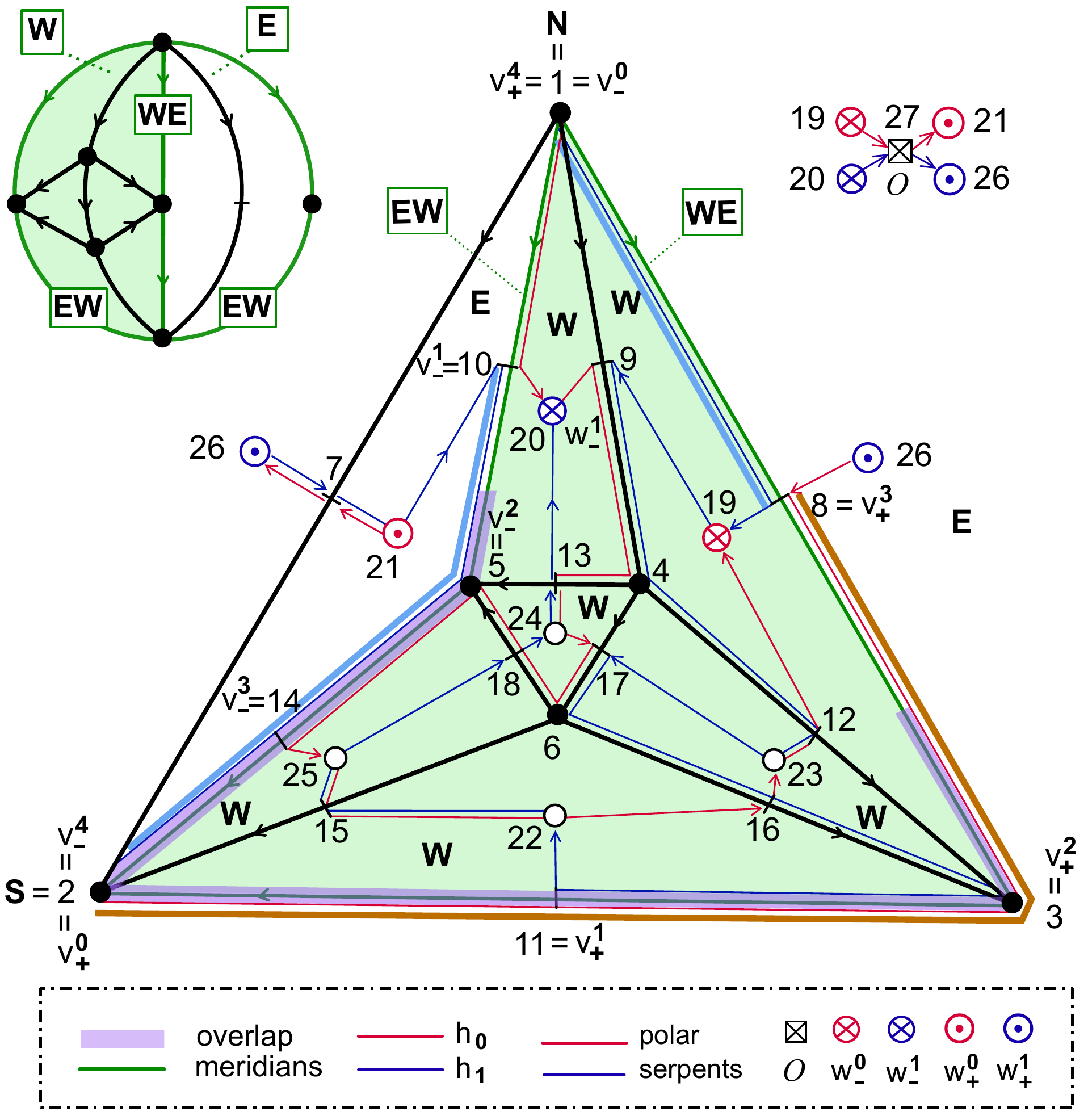}
\caption{\emph{
Construction of the SZS-pair $(h_0,h_1)$, for the bipolar octahedral 3-cell template of fig.~\ref{fig:6.1}(b).
See definition~\ref{def:5.1} and detailed text comments.
The resulting paths $h_\iota$ with respect to equilibrium labels in the figure are listed in \eqref{eq:6.1}.
See also \eqref{eq:6.2} for the resulting Sturm permutation $\sigma = h_0^{-1} \circ h_1$, and fig.~\ref{fig:6.4} for the resulting Sturm meander.
}}
\label{fig:6.3}
\end{figure}

We conclude this section with a construction of the unique solid octahedron with a 2-face hemisphere.
Uniqueness is understood up to trivial equivalences \eqref{eq:2.4}--\eqref{eq:2.9}.
See fig.~\ref{fig:6.1}(b) and fig.~\ref{fig:6.3}.

By proposition~\ref{prop:6.1} the poles $\mathbf{N}, \mathbf{S}$ are edge-adjacent.
Without loss of generality $\mathbf{N}=1$, $\mathbf{S}=2$ in fig.~\ref{fig:6.1}(b).
Because we require the smaller hemisphere to consist of two faces, the proof of proposition~\ref{prop:6.2} and fig.~\ref{fig:6.2}(c) imply that the meridians are given by the edges $8,11$ and $10,14$, respectively.
The trivial equivalence $x \mapsto 1-x$ of \eqref{eq:2.8} interchanges the meridians.
Without loss of generality, therefore, let $\mathbf{WE} =\lbrace 8,3,11 \rbrace$ and $\mathbf{EW} =\lbrace 10,5,14 \rbrace$.
This determines the hemisphere decomposition and the polar-meridian faces $w_\pm^\iota$ of fig.~\ref{fig:6.1}(b).

It remains to determine the bipolar orientation of the edges $7, \ldots , 18$.
The poles $\mathbf{N} =1$ and $\mathbf{S} =2$ determine the orientations of the polar edges $7,8,9,10$ and $11,14,15$.
By the meridian boundary edge orientation of theorem~\ref{thm:4.1}(iii), in the six-face hemisphere $\mathbf{W}$, the edges $12,16$ and $13,18$ are oriented towards the meridian vertices 3 and 5, respectively.
The remaining edge 17 must be oriented from 4 to 6.
Otherwise 6 becomes a second orientation source, besides $\mathbf{N}$, in contradiction to bipolarity, theorem~\ref{thm:4.1}(i).
This determines the orientation of all edges, as in fig.~\ref{fig:6.1}(b).
Reverting all orientations by the trivial equivalence $u \mapsto -u$ of \eqref{eq:2.4}, incidentally, reproduces the same Sturm attractor $\mathcal{A}$ and octahedral complex $\mathcal{C}$, with interchanged hemisphere labels.

\begin{figure}[t!]
\centering \includegraphics[width=\textwidth]{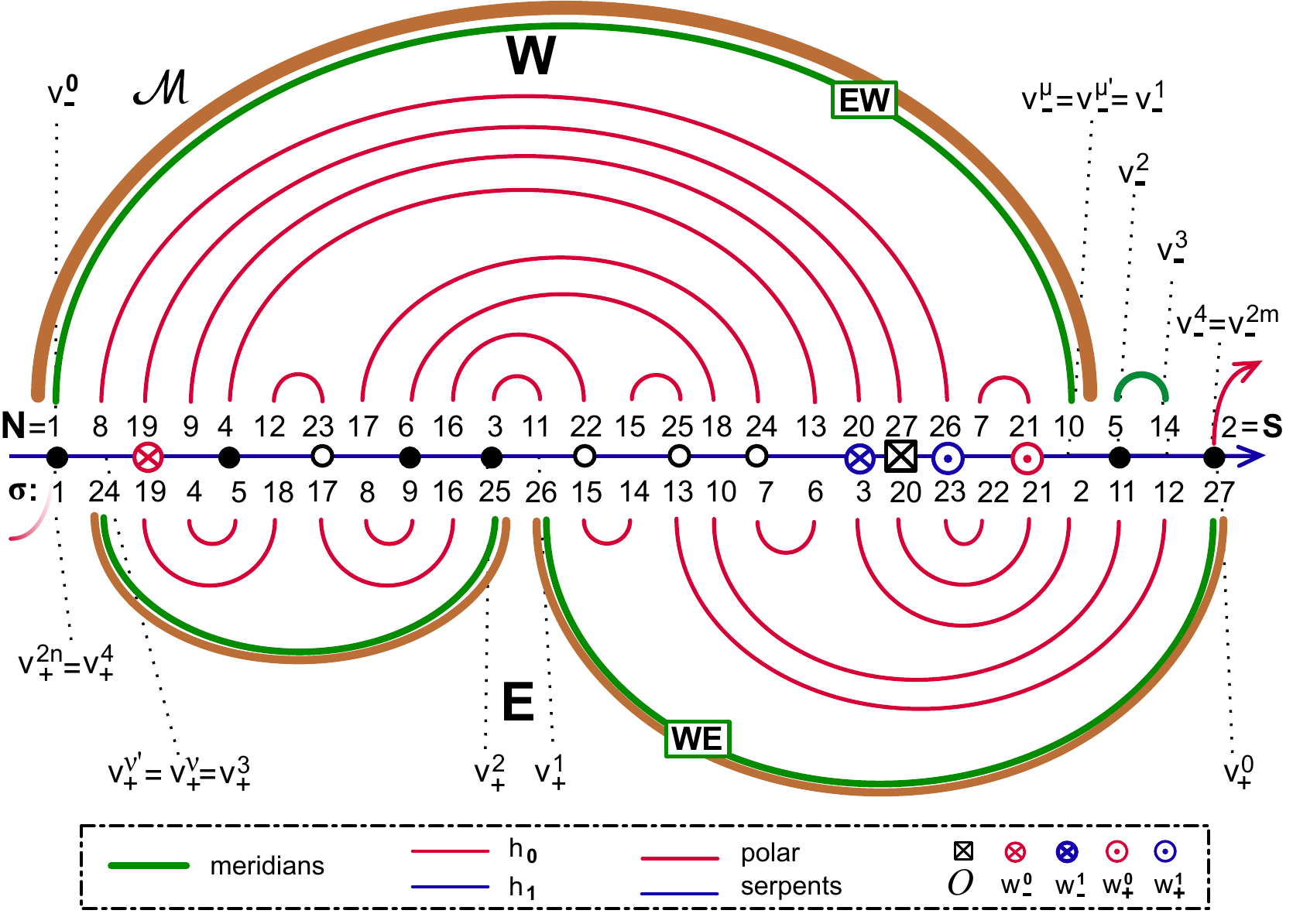}
\caption{\emph{
The Sturm 3-meander template $\mathcal{M}$ of the $6+2$ octahedron complex $\mathcal{C}$.
Equilibrium labels above the horizontal $h_1$-axis, and Sturm permutation $\sigma = H_0^{-1} \circ h_1$ below.
Note how the shooting curve $h_0$ and the horizontal axis $h_1$ follow the equilibrium labels according to their enumerations \eqref{eq:6.1}.
Polar $h_0$-serpents (orange) and anti-polar $h_1$-serpents (blue) overlap at  $v_-^{\mu '} = v_-^{\mu} = v_-^1 =10,\, v_+^{\nu '} = v_+^\nu = v_+^3 =8$.
The $h_1$-extreme faces $w_\pm^0$ are $h_0$-neighbors of $\mathcal{O}$.
The first polar arcs of the $h_0$-polar serpents overarch $\mathcal{O}$.
Note consistency of all Morse numbers $i_v$ with Morse indices $i(v)$, for all equilibria $v$, according to \eqref{eq:1.22a} and fig.~\ref{fig:6.3}.
}}
\label{fig:6.4}
\end{figure}

\begin{figure}[t!]
\centering \includegraphics[width=\textwidth]{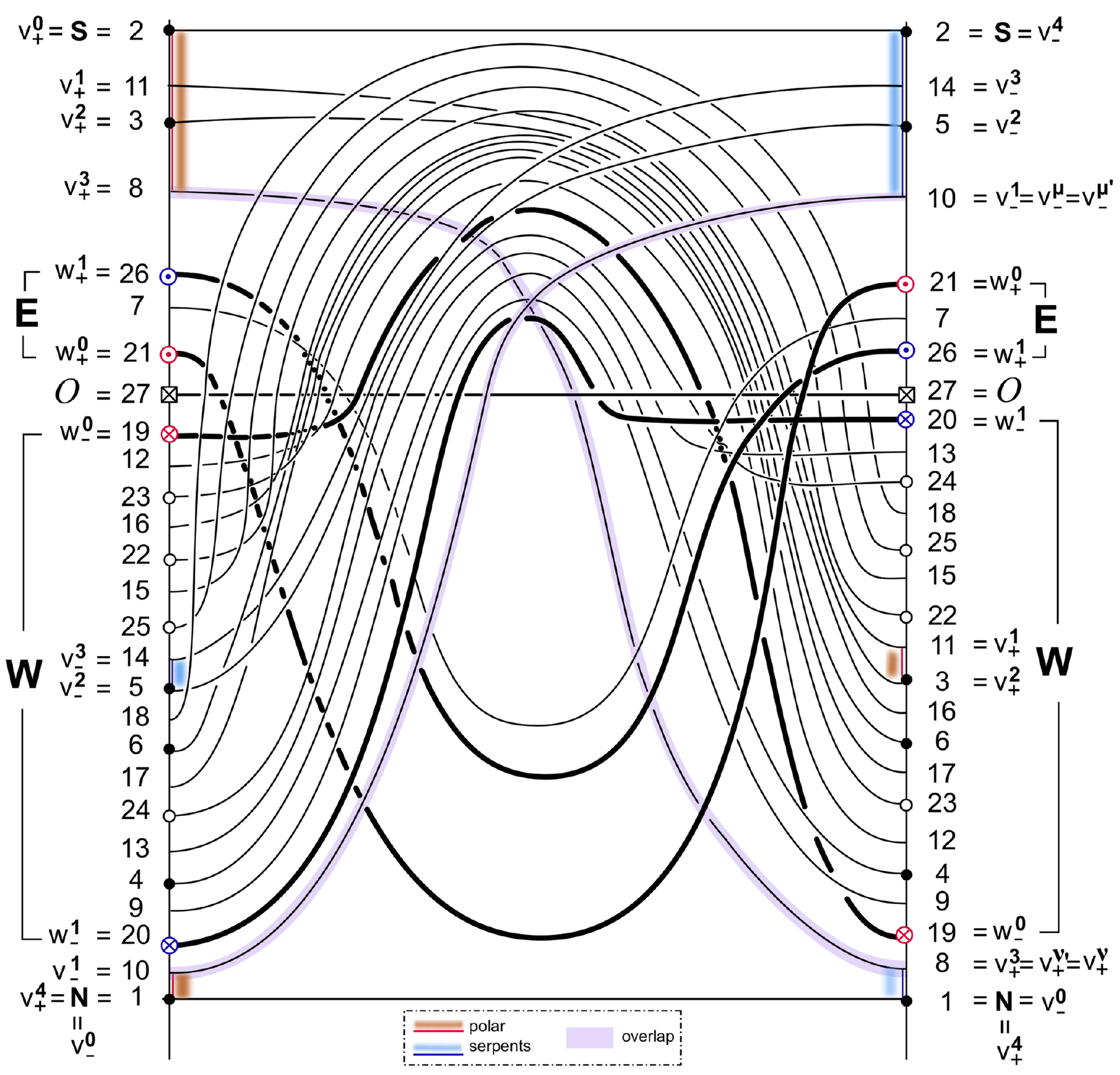}
\caption{\emph{
A sketch of the spatial profiles $v(x)$, for all equilibria in the solid Sturm octahedron of figs.~\ref{fig:6.1}(b) and \ref{fig:6.3}.
Also see the legend of fig.~\ref{fig:1.1} and the general 3-ball template of fig.~\ref{fig:3.1}.
The equilibria $1, \ldots , 27$ are ordered by $h_0$, $h_1$ along the left, right vertical axis $x=0,\, 1$, respectively.
}}
\label{fig:6.5}
\end{figure}

As in section~\ref{sec5} we can now construct a unique SZS-pair of Hamiltonian paths $(h_0,h_1)$ for the meridian decomposition and the bipolar 3-cell template $\mathcal{C}$; see fig.~\ref{fig:6.3}.
We first construct the path $h_0$, as it traverses each triangle face $19,20, 22, \ldots , 25$ in $\mathbf{W}$ from the lower left boundary edge to the upper right boundary edge.
In $\mathbf{E}$ we reverse the roles of ``left'' and ``right''.
Since the back triangle $26 \in \mathbf{E}$ is depicted in the wrong planar orientation, $h_0$ traverses this face from lower left 7 to upper right 8.
At $w_-^0 = 19 \in \mathbf{W}$, of course, the path $h_0$ leaves $\mathbf{N}\cup\mathbf{EW}\cup\mathbf{W}$ and tunnels through the barycenter $\mathcal{O}=27$ to re-emerge at $w_+^0 = 21\in \mathbf{E}$.
The rules for the Hamiltonian path $h_1$ are analogous, by reflection; see definition~\ref{def:5.1}.
The paths $h_\iota$ from $\mathbf{N}=1$ to $\mathbf{S}=2$ are 
	\begin{equation}
	\begin{aligned}
	h_0: \,\, 1\;\; &\text{10 20 9 4 13 24 17 6 18 5 14 25 15
	 22 16 23 12 19 27 21 7 26 8 3 11 2}\,;\\
	h_1: \,\, 1\;\; &\text{8  19 9  4  12 23 17 6  16 3  11 22 15
	25 18 24 13 20 27 26 7  21 10 5  14 2}\,.
	\end{aligned}
	\label{eq:6.1}
	\end{equation}
The resulting meander $\mathcal{M}$ of $\sigma = h_0^{-1} \circ h_1$ is depicted in fig.~\ref{fig:6.4}:
	\begin{equation}
	\begin{aligned}
	\sigma 
	&=  \lbrace 1,\text{24, 19, 4, 5, 18, 17, 8, 9, 16, 25, 26, 15, 14,}\\
	&\phantom{=\lbrace 1,} \text{ 13, 10, 7, 6, 3, 20, 23, 22, 21, 2, 11, 12, 27}\rbrace =\\
	&= \text{(2 24) (3 19) (6 18) (7 17) (10 16) (11 25) (12 26) (13 15) (21 23)}\,.
	\end{aligned}
	\label{eq:6.2}
	\end{equation}
Although $\sigma$ is an involution, $\sigma$ cannot be realized as the Sturm permutation of any nonlinearity $f=f(u)$ which only depends on $u$; see \cite{firowo12}.
Indeed the permutation cycles $\text{(11 25)}$ and $\text{(12 26)}$, for example, are not nested.

By theorem~\ref{thm:5.2} the meander $\mathcal{M}$ is Sturm and a 3-meander template.
We use analogous notation to facilitate the comparison of the octahedral Sturm meander $\mathcal{M}$ in fig.~\ref{fig:6.4} with the general 3-meander template of fig.~\ref{fig:1.3}.
A direct inspection of the Sturm meander \eqref{eq:6.2} with the methods of \cite{firo96} confirms that all zero numbers and the connection graph of $\sigma$ coincide with the connection graph prescribed by the octahedral 3-cell template of fig.~\ref{fig:6.1}(b).
By \eqref{eq:1.22a} it is easy to explicitly check, at least, how the Morse numbers $i_v$ coincide with the prescribed Morse indices $i(v)$ for all equilibrium labels $v=1, \ldots , 27$.
For a ``spaghetti'' sketch of spatial equilibrium profiles $v = v(x)$ see fig.~\ref{fig:6.5}.

It remains to show how the dynamic Thom-Smale complex $\mathcal{C}_\sigma$ of the Sturm global attractor $\mathcal{A}_\sigma$ of $\sigma$ coincides with the prescribed 3-cell template $\mathcal{C}$, not just for our octahedral example but, in complete generality.
For this last design step in our study of Sturm 3-ball attractors we must refer to the sequel \cite{firo3d-2}.



{\small

\begin{thebibliography}{99999999}

\bibitem[An86]{an86}
S.~Angenent.
The {M}orse-{S}male property for a semi-linear parabolic equation.
{\em J.~Diff.~Eqns.} \textbf{62} (1986), 427--442.

\bibitem[An88]{an88}
S.~Angenent.
The zero set of a solution of a parabolic equation.
{\em J. Reine Angew. Math.} \textbf{390} (1988), 79--96.


\bibitem[Ar88]{ar88}
V.I.~Arnold. A branched covering $CP^2 \rightarrow S^4$, hyperbolicity and projective topology.
{\em Sib.~Math.~J.} \textbf{29} (1988) 717--726.

\bibitem[ArVi89]{arvi89}   
V.I.~Arnold, M.I.~Vishik et al. Some solved and unsolved problems in the theory of differential equations and mathematical physics.
{\em Russ. Math. Surv.} \textbf{44} (1989) 157--171.

\bibitem[BaVi92]{bavi92}
A.V.~Babin and M.I.~Vishik.
{\em Attractors of Evolution Equations}.
North Holland, Amsterdam, 1992.

\bibitem[BaHu04]{bahu04}
A.~Banyaga and D.~Hurtubise.
\emph{Lectures on Morse Homology}. Springer-Verlag, Berlin, 2004.

\bibitem[BiZh92]{bizh92}
J.-M.~Bismut and W.~Zhang.
{\em An extension of a theorem by Cheeger and M\"uller. With an appendix by Fran\c{c}ois Laudenbach}.
Ast\'erisque \textbf{205}, Soc.~Math.~de France, 1992. 

\bibitem[Bo88]{bo88}
R.~Bott.
Morse theory indomitable.
{\em Public.~ Math.~I.H.\'E.S.} \textbf{68} (1988), 99--114.

\bibitem[Br90]{br90}
P.~Brunovsk\'y. The attractor of the scalar
reaction diffusion equation is a smooth graph.
{\em J.~Dyn. Diff. Eqns.}  \textbf{2} (1990), 293--323.

\bibitem[BrFi86]{brfi86}
P.~Brunovsk\'y and B.~Fiedler.
Numbers of zeros on invariant manifolds in reaction-diffusion equations.
\emph{Nonlin.Analysis, TMA} \textbf{10} (1986), 179--193.

\bibitem[BrFi88]{brfi88}
P.~Brunovsk\'y and B.~Fiedler.
 Connecting orbits in scalar reaction diffusion equations.
 {\em Dynamics Reported} \textbf{1} (1988), 57--89.

\bibitem[BrFi89]{brfi89}
P.~Brunovsk\'y and B.~Fiedler.
 Connecting orbits in scalar reaction diffusion equations {II}: The
  complete solution.
 {\em J.~Diff.~Eqns.} \textbf{81} (1989), 106--135.

\bibitem[ChIn74]{chin74}
N. Chafee and E. Infante.
A bifurcation problem for a nonlinear parabolic equation.
{\em J. Applicable Analysis} \textbf{4} (1974), 17--37.

\bibitem[ChVi02]{chvi02}
V.V.~Chepyzhov and M.I.~Vishik.
{\em Attractors for {E}quations of {M}athematical {P}hysics}.
Colloq. AMS, Providence, 2002.

\bibitem[Edetal94]{edetal94}
A. Eden, C. Foias, B. Nicolaenko, R. Temam.
\emph{Exponential Attractors for Dissipative Evolution Equations. }
Wiley, Chichester, 1994.

\bibitem[Fi02]{fi02}
B. Fiedler (ed.)  {\em {H}andbook of {D}ynamical
{S}ystems} \textbf{2}, {E}lsevier, {A}msterdam, 2002.

\bibitem[FiRo96]{firo96}
B. Fiedler and C. Rocha.
Heteroclinic orbits of semilinear parabolic equations.
{\em  J. Diff. Eqns.} \textbf{125} (1996), 239--281.

\bibitem[FiRo99]{firo99}
B.~Fiedler and C.~Rocha.
Realization of meander permutations by boundary value problems.  {\em
J. Diff. Eqns.} \textbf{156} (1999), 282--308.

\bibitem[FiRo00]{firo00}
B.~Fiedler and C.~Rocha.
Orbit equivalence of global attractors of semilinear parabolic
differential equations.
{\em Trans. Amer. Math. Soc.} \textbf{352} (2000), 257--284.

\bibitem[FiRo08]{firo08}
B. Fiedler and C. Rocha.
Connectivity and design of planar global attractors of Sturm type, II: Connection graphs.
{\em J.\ Diff. Eqns.} \textbf{244} (2008), 1255--1286.

\bibitem[FiRo09]{firo09}
B. Fiedler and C. Rocha.
Connectivity and design of planar global attractors of Sturm type, I: Bipolar orientations and Hamiltonian paths. 
{\em J.~Reine~Angew.~Math.} \textbf{635} (2009), 71--96.

\bibitem[FiRo10]{firo10}
B. Fiedler and C. Rocha.
Connectivity and design of planar global attractors of Sturm type, III: Small and Platonic examples.
{\em J. Dyn. Diff. Eqns.} \textbf{22} (2010), 121--162.

\bibitem[FiRo14]{firo14}
B.~Fiedler and C.~Rocha.
Nonlinear Sturm global attractors: unstable manifold decompositions as regular CW-complexes.
{\em Discr. Cont. Dyn. Sys.} \textbf{34} (2014), 5099-5122.

\bibitem[FiRo15]{firo13}
B.~Fiedler and C.~Rocha.
Schoenflies spheres as boundaries of bounded unstable manifolds in gradient Sturm systems.
{\em J.~Dyn.~Diff.~Eqns.} \textbf{27} (2015), 597--626.

\bibitem[FiRo17a]{firo3d-2}
B.~Fiedler and C.~Rocha.
Sturm 3-balls and global attractors 2: Design of Thom-Smale complexes. 
arXiv:1704.00344, J. Dyn. Diff. Eqs., submitted 2017.

\bibitem[FiRo17b]{firo3d-3}
B.~Fiedler and C.~Rocha.
Sturm 3-balls and global attractors 3: Examples of Thom-Smale complexes. 
 arXiv:1708.00690, Discr. Cont. Dyn. Syst. A, submitted 2017. 
 
\bibitem[FiSc03]{fisc03}
B.~Fiedler and A.~Scheel.
Spatio-temporal dynamics of reaction-diffusion patterns. In
\emph{Trends in Nonlinear Analysis}, M.~Kirkilionis et al.~(eds.),
Springer-Verlag, Berlin 2003, 23--152.

\bibitem[Fietal12]{firowo12}
B.~Fiedler, C.~Rocha and M.~Wolfrum.
A permutation characterization of Sturm global attractors of Hamiltonian type.
{\em J. Diff. Eqns.} \textbf{252} (2012), 588--623.

\bibitem[Fietal14]{fietal14}
B.~Fiedler, C.~Grotta-Ragazzo and C.~Rocha.
An explicit Lyapunov function for reflection symmetric parabolic differential equations on the circle.
{\em Russ. Math. Surveys.} \textbf{69} (2014), 419--433.

\bibitem[Fr79]{fr79}
J.M.~Franks.
Morse-Smale flows and homotopy theory.
{\em Topology} \textbf{18} (1979), 199--215.

\bibitem[Fretal95]{fretal95}
H.~de Fraysseix, P.O.~de Mendez and P.~Rosenstiehl.
Bipolar orientations revisited. 
{\em Discr. Appl. Math.} \textbf{56} (1995), 157--179.

\bibitem[FrPi90]{frpi90}
R.~Fritsch and R.A.~Piccinini.
\emph{Cellular Structures in Topology.}
Cambridge University Press, 1990.

\bibitem[FuOl88]{fuol88}
G.~Fusco and W.~Oliva.
Jacobi matrices and transversality.
{\em Proc. Royal Soc. Edinburgh A} \textbf{109} (1988) 231--243.


\bibitem[FuRo91]{furo91}
G.~Fusco and C.~Rocha.
 A permutation related to the dynamics of a scalar parabolic {PDE}.
 {\em J. Diff. Eqns.} \textbf{91} (1991), 75--94.

\bibitem[Ga04]{ga04}
V.A. Galaktionov.
\emph{Geometric Sturmian Theory of Nonlinear Parabolic Equations and Applications. }
Chapman \& Hall, Boca Raton, 2004.

\bibitem[Ha88]{ha88}
J.K.~Hale.
 {\em Asymptotic Behavior of Dissipative Systems}.
 Math. Surv. \textbf{25}. AMS, Providence, 1988.

\bibitem[Haetal02]{haetal02}
J.K. Hale,  L.T. Magalh\~aes, and W.M. Oliva.
\emph{Dynamics in Infinite Dimensions.} Springer-Verlag, New York, 2002.

\bibitem[He81]{he81}
D. Henry.
{\em Geometric Theory of Semilinear Parabolic Equations.}
Lect. Notes Math. \textbf{804}, Springer-Verlag, New York, 1981.

\bibitem[He85]{he85}
D.~Henry.
Some infinite dimensional {M}orse-{S}male systems defined by parabolic differential equations.
 {\em J. Diff. Eqns.} \textbf{59} (1985), 165--205.

\bibitem[Hu11]{hu11}
B.~Hu.
{\em Blow-up Theories for Semilinear Parabolic Equations.}
Lect. Notes Math. \textbf{2018}, Springer-Verlag, Berlin, 2011.

\bibitem[Jo89]{jo89}
M.S.~Jolly.
Explicit construction of an inertial manifold for a reaction
diffusion equation.
{\em J. Diff. Eqns.} \textbf{78} (1989), 220--261.

\bibitem[La91]{la91}
O.A.~Ladyzhenskaya.
 {\em Attractors for Semigroups and Evolution Equations.}
 Cambridge University Press, 1991.

\bibitem[Ma78]{ma78}
H.~Matano.
Convergence of solutions of one-dimensional semilinear parabolic equations.
{\em J. Math. Kyoto Univ.} \textbf{18} (1978), 221--227.

\bibitem[Ma79]{ma79}
H.~Matano.
Asymptotic behavior and stability of solutions of semilinear diffusion equations.
{\em Publ. Res. Inst. Math. Sci.}, \textbf{15} (1979), 401--454.

\bibitem[Ma82]{ma82}
H.~Matano.
Nonincrease of the lap-number of a solution for a one-dimensional
semi-linear parabolic equation.
{\em J. Fac. Sci. Univ. Tokyo Sec. IA} \textbf{29} (1982), 401--441.

\bibitem[MaNa97]{mana97}
H.~Matano and K.-I.~Nakamura.
The global attractor of semilinear parabolic equations on ${S^1}$.
{\em Discr. Cont. Dyn. Sys.} \textbf{3} (1997), 1--24.

\bibitem[MP88]{mp88}
J.~Mallet-Paret.
Morse decompositions for delay-differential equations.
{\em J. Diff. Eqns.} \textbf{72} (1988), 270--315.

\bibitem[Ol83]{ol83}
W.~Oliva.
{\em Stability of Morse-Smale maps.}
Technical Report, Dept. Applied Math. IME-USP \textbf{1}
 (1983).

\bibitem[PaMe82]{pame82}
J.~Palis and W.~de~Melo.
{\em Geometric Theory of Dynamical Systems. An Introduction.}
Springer-Verlag, New York, 1982.

\bibitem[PaSm70]{pasm70}
J. Palis and S. Smale.
{\em Structural stability theorems.}
Global Analysis. Proc. Simp. in Pure Math. AMS, Providence, 1970.

\bibitem[Pa83]{pa83}
A.~Pazy.
{\em Semigroups of Linear Operators and Applications to Partial Differential Equations}.
Springer-Verlag, New York, 1983.

\bibitem[Po16]{po16}
P.~Polá\v{c}ik.
Personal communication, 2016.

\bibitem[Ra02]{ra02}
G.~Raugel.
Global attractors.
In \cite{fi02}, 2002, 885--982.

\bibitem[Ro91]{ro91} 
C.~Rocha. 
Properties of the attractor of a scalar parabolic PDE. 
{\em J. Dyn. Diff. Eqns.} \textbf{3} (1991), 575-591. 

\bibitem[SeYo02]{seyo02}
G.R. Sell, Y. You.
{\em Dynamics of Evolutionary Equations.}
Springer-Verlag, New York, 2002.

\bibitem[St1836]{st1836}
C.~Sturm.
Sur une classe d'\'equations \`a diff\'erences partielles.
{\em J.~Math.~Pure~Appl.} \textbf{1} (1836), 373--444.

\bibitem[Ta79]{ta79}
H.~Tanabe.
{\em Equations of Evolution}.
Pitman, Boston, 1979.

\bibitem[Te88]{te88}
R.~Temam.
{\em Infinite-Dimensional Dynamical Systems in Mechanics and
Physics}.
Springer-Verlag, New York, 1988.
 
\bibitem[Wo02]{wo02}
M.~Wolfrum.
Geometry of heteroclinic cascades in scalar parabolic differential equations.
{\em J. Dyn. Diff. Eqns.} \textbf{14} (2002), 207--241.

\bibitem[Ze68]{ze68}
T.I. Zelenyak.
Stabilization of solutions of boundary value problems for a second order parabolic equation with one space variable.
{\em Diff.~Eqns.} \textbf{4} (1968), 17--22.





\end{thebibliography}
}
\end{document}